\pdfoutput=1
\documentclass[a4paper]{amsart}
\usepackage[T1]{fontenc}
\usepackage{mathrsfs}
\usepackage{amstext}
\usepackage{amsthm}
\usepackage{amssymb}
\usepackage{stackrel}
\usepackage[unicode=true,pdfusetitle,
 bookmarks=true,bookmarksnumbered=false,bookmarksopen=false,
 breaklinks=false,pdfborder={0 0 0},pdfborderstyle={},backref=false,colorlinks=false]
 {hyperref}

\makeatletter

\numberwithin{equation}{section}
\numberwithin{figure}{section}
\theoremstyle{plain}
\newtheorem{thm}{\protect\theoremname}[section]
\theoremstyle{definition}
\newtheorem{example}[thm]{\protect\examplename}
\theoremstyle{plain}
\newtheorem{qst}[thm]{\protect\questionname}
\theoremstyle{definition}
\newtheorem{defn}[thm]{\protect\definitionname}
\theoremstyle{remark}
\newtheorem{rem}[thm]{\protect\remarkname}
\theoremstyle{plain}
\newtheorem{prop}[thm]{\protect\propositionname}
\theoremstyle{plain}
\newtheorem{lem}[thm]{\protect\lemmaname}
\theoremstyle{plain}
\newtheorem{cor}[thm]{\protect\corollaryname}

\usepackage{tikz-cd}
\usepackage{mathtools}
\usepackage{adjustbox}
\usepackage{enumitem}
\tikzcdset{scale cd/.style={every label/.append style={scale=#1},
    cells={nodes={scale=#1}}}}
\usepackage{quiver}
\usepackage{eucal}
\usepackage{mleftright}
\mleftright
\usepackage{tocvsec2}

\makeatother

\providecommand{\corollaryname}{Corollary}
\providecommand{\definitionname}{Definition}
\providecommand{\examplename}{Example}
\providecommand{\lemmaname}{Lemma}
\providecommand{\propositionname}{Proposition}
\providecommand{\questionname}{Question}
\providecommand{\remarkname}{Remark}
\providecommand{\theoremname}{Theorem}

\begin{document}
\global\long\def\sf#1{\mathsf{#1}}%

\global\long\def\scr#1{\mathscr{{#1}}}%

\global\long\def\cal#1{\mathcal{#1}}%

\global\long\def\bb#1{\mathbb{#1}}%

\global\long\def\bf#1{\mathbf{#1}}%

\global\long\def\frak#1{\mathfrak{#1}}%

\global\long\def\fr#1{\mathfrak{#1}}%

\global\long\def\u#1{\underline{#1}}%

\global\long\def\tild#1{\widetilde{#1}}%

\global\long\def\mrm#1{\mathrm{#1}}%

\global\long\def\pr#1{\left(#1\right)}%

\global\long\def\abs#1{\left|#1\right|}%

\global\long\def\inp#1{\left\langle #1\right\rangle }%

\global\long\def\br#1{\left\{  #1\right\}  }%

\global\long\def\norm#1{\left\Vert #1\right\Vert }%

\global\long\def\hat#1{\widehat{#1}}%

\global\long\def\opn#1{\operatorname{#1}}%

\global\long\def\bigmid{\,\middle|\,}%

\global\long\def\Top{\sf{Top}}%

\global\long\def\Set{\sf{Set}}%

\global\long\def\SS{\sf{sSet}}%

\global\long\def\Kan{\sf{Kan}}%

\global\long\def\Cat{\mathcal{C}\sf{at}}%

\global\long\def\imfld{\cal M\mathsf{fld}}%

\global\long\def\ids{\cal D\sf{isk}}%

\global\long\def\ich{\cal C\sf h}%

\global\long\def\SW{\mathcal{SW}}%

\global\long\def\SHC{\mathcal{SHC}}%

\global\long\def\Fib{\mathcal{F}\mathsf{ib}}%

\global\long\def\Bund{\mathcal{B}\mathsf{und}}%

\global\long\def\B{\sf B}%

\global\long\def\Spaces{\sf{Spaces}}%

\global\long\def\Mod{\sf{Mod}}%

\global\long\def\Nec{\sf{Nec}}%

\global\long\def\Fin{\sf{Fin}}%

\global\long\def\Ch{\sf{Ch}}%

\global\long\def\Ab{\sf{Ab}}%

\global\long\def\SA{\sf{sAb}}%

\global\long\def\P{\mathsf{POp}}%

\global\long\def\Op{\mathcal{O}\mathsf{p}}%

\global\long\def\Opg{\mathcal{O}\mathsf{p}_{\infty}^{\mathrm{gn}}}%

\global\long\def\Tup{\mathsf{Tup}}%

\global\long\def\Del{\mathbf{\Delta}}%

\global\long\def\id{\operatorname{id}}%

\global\long\def\Aut{\operatorname{Aut}}%

\global\long\def\End{\operatorname{End}}%

\global\long\def\Hom{\operatorname{Hom}}%

\global\long\def\Ext{\operatorname{Ext}}%

\global\long\def\sk{\operatorname{sk}}%

\global\long\def\ihom{\underline{\operatorname{Hom}}}%

\global\long\def\N{\mathrm{N}}%

\global\long\def\-{\text{-}}%

\global\long\def\op{\mathrm{op}}%

\global\long\def\To{\Rightarrow}%

\global\long\def\rr{\rightrightarrows}%

\global\long\def\rl{\rightleftarrows}%

\global\long\def\mono{\rightarrowtail}%

\global\long\def\epi{\twoheadrightarrow}%

\global\long\def\comma{\downarrow}%

\global\long\def\ot{\leftarrow}%

\global\long\def\corr{\leftrightsquigarrow}%

\global\long\def\lim{\operatorname{lim}}%

\global\long\def\colim{\operatorname{colim}}%

\global\long\def\holim{\operatorname{holim}}%

\global\long\def\hocolim{\operatorname{hocolim}}%

\global\long\def\Ran{\operatorname{Ran}}%

\global\long\def\Lan{\operatorname{Lan}}%

\global\long\def\Sk{\operatorname{Sk}}%

\global\long\def\Sd{\operatorname{Sd}}%

\global\long\def\Ex{\operatorname{Ex}}%

\global\long\def\Cosk{\operatorname{Cosk}}%

\global\long\def\Sing{\operatorname{Sing}}%

\global\long\def\Sp{\operatorname{Sp}}%

\global\long\def\Spc{\operatorname{Spc}}%

\global\long\def\Ho{\operatorname{Ho}}%

\global\long\def\Fun{\operatorname{Fun}}%

\global\long\def\map{\operatorname{map}}%

\global\long\def\diag{\operatorname{diag}}%

\global\long\def\Gap{\operatorname{Gap}}%

\global\long\def\cc{\operatorname{cc}}%

\global\long\def\Ob{\operatorname{Ob}}%

\global\long\def\Map{\operatorname{Map}}%

\global\long\def\Rfib{\operatorname{RFib}}%

\global\long\def\Lfib{\operatorname{LFib}}%

\global\long\def\Tw{\operatorname{Tw}}%

\global\long\def\Equiv{\operatorname{Equiv}}%

\global\long\def\Arr{\operatorname{Arr}}%

\global\long\def\Cyl{\operatorname{Cyl}}%

\global\long\def\Path{\operatorname{Path}}%

\global\long\def\Alg{\operatorname{Alg}}%

\global\long\def\ho{\operatorname{ho}}%

\global\long\def\Comm{\operatorname{Comm}}%

\global\long\def\Triv{\operatorname{Triv}}%

\global\long\def\triv{\operatorname{triv}}%

\global\long\def\Env{\operatorname{Env}}%

\global\long\def\Act{\operatorname{Act}}%

\global\long\def\act{\operatorname{act}}%

\global\long\def\loc{\operatorname{loc}}%

\global\long\def\Assem{\operatorname{Assem}}%

\global\long\def\Nat{\operatorname{Nat}}%

\global\long\def\lax{\mathrm{lax}}%

\global\long\def\weq{\mathrm{weq}}%

\global\long\def\fib{\mathrm{fib}}%

\global\long\def\cof{\mathrm{cof}}%

\global\long\def\inj{\mathrm{inj}}%

\global\long\def\univ{\mathrm{univ}}%

\global\long\def\Ker{\opn{Ker}}%

\global\long\def\Coker{\opn{Coker}}%

\global\long\def\Im{\opn{Im}}%

\global\long\def\Coim{\opn{Im}}%

\global\long\def\coker{\opn{coker}}%

\global\long\def\im{\opn{\mathrm{im}}}%

\global\long\def\coim{\opn{coim}}%

\global\long\def\gn{\mathrm{gn}}%

\global\long\def\Mon{\opn{Mon}}%

\global\long\def\Un{\opn{Un}}%

\global\long\def\St{\opn{St}}%

\global\long\def\cun{\widetilde{\opn{Un}}}%

\global\long\def\cst{\widetilde{\opn{St}}}%

\global\long\def\Sym{\operatorname{Sym}}%

\global\long\def\CA{\operatorname{CAlg}}%

\global\long\def\rd{\mathrm{rd}}%

\global\long\def\xmono#1#2{\stackrel[#2]{#1}{\rightarrowtail}}%

\global\long\def\xepi#1#2{\stackrel[#2]{#1}{\twoheadrightarrow}}%

\global\long\def\adj{\stackrel[\longleftarrow]{\longrightarrow}{\bot}}%

\global\long\def\btimes{\boxtimes}%

\global\long\def\ps#1#2{\prescript{}{#1}{#2}}%

\global\long\def\ups#1#2{\prescript{#1}{}{#2}}%

\global\long\def\hofib{\mathrm{hofib}}%

\global\long\def\cofib{\mathrm{cofib}}%

\global\long\def\Vee{\bigvee}%

\global\long\def\w{\wedge}%

\global\long\def\t{\otimes}%

\global\long\def\bp{\boxplus}%

\global\long\def\rcone{\triangleright}%

\global\long\def\lcone{\triangleleft}%

\global\long\def\S{\mathsection}%

\global\long\def\p{\prime}%

\global\long\def\pp{\prime\prime}%

\global\long\def\W{\overline{W}}%

\global\long\def\o#1{\overline{#1}}%

\title[The Grothendieck Construction for Categorical Patterns]{The Grothendieck Construction for $\infty$-Categories Fibered over Categorical Patterns}
\begin{abstract}
We show how to treat families of $\infty$-categories fibered in categorical
patterns (e.g., $\infty$-operads and monoidal $\infty$-categories)
in terms of fibrations by relativizing the Grothendieck construction.
As applications, we construct an analog of the universal cocartesian
fibration and explain how to compute limits and colimits of $\infty$-categories
fibered in categorical patterns.
\end{abstract}

\author{Kensuke Arakawa}
\address{Department of Mathematics, Kyoto University, Kyoto, 606-8502, Japan}
\email{arakawa.kensuke.22c@st.kyoto-u.ac.jp}

\maketitle
\tableofcontents{}

\section{Introduction}

\settocdepth{section}

\subsection{\texorpdfstring{$\infty$}{infinity}-Categories with Structures}

The Grothendieck construction in ordinary category theory establishes
an equivalence between $\sf{Cat}$-valued pseudofunctors and Grothendieck
opfibrations \cite[$\S$8.3]{Borceux2}. An analog of this construction
in $\infty$-category theory, due to Lurie \cite[Chapter 3]{HTT},
is called the \textbf{straightening--unstraightening} equivalence:
It consists of a pair of categorical equivalences
\[
\St:\cal C\sf{oCart}\pr{\cal C}\stackrel[\simeq]{\simeq}{\rl}\Fun\pr{\cal C,\Cat_{\infty}}:\Un
\]
between the $\infty$-category of cocartesian fibrations over $\cal C$
and the $\infty$-category of $\Cat_{\infty}$-valued functors\footnote{Pseudofunctors do not enter the picture here because every functor
of $\infty$-categories is more or less something like a pseudofunctor.} on $\cal C$. 

The utility of this equivalence comes from the fact that cocartesian
fibrations require less choices than the corresponding functor $\cal C\to\Cat_{\infty}$.
Indeed, while a functor $F:\cal C\to\Cat_{\infty}$ associates with
each morphism $f:C\to D$ in $\cal C$ a \textit{specific} functor
$Ff:FC\to FD$ of $\infty$-categories, a cocartesian fibration $p:\cal E\to\cal C$
associates with $f$ a \textit{contractible space} of functors $p^{-1}\pr C\to p^{-1}\pr D$,
any of which may be called the functor induced by $f$. This gives
homotopy theorists a strong incentive to prefer cocartesian fibrations,
for when objects are well-defined only up to contractible ambiguity,
making arbitrary (yet explicit) choices is often unnatural and difficult.

The straightening--unstraightening equivalence is especially useful
in dealing with $\infty$-categories with structures. The definition
of symmetric monoidal $\infty$-categories illustrates this well.
It is tempting to define a symmetric monoidal $\infty$-category as
an $\infty$-category $\cal C$ equipped with a bifunctor $\otimes:\cal C\times\cal C\to\cal C$,
a unit object $I\in\cal C$, and coherent natural equivalences, satisfying
various compatibility conditions. However, there are two problems
with this approach:
\begin{enumerate}
\item Writing down the compatibility conditions (much less verifying them)
will be agonizing.
\item There is often no canonical choice for the tensor product nor the
unit object. 
\end{enumerate}

Hopefully, there is a neat way to get around problem (1). To explain
this, we recall that the data of a commutative monoid can be encoded
as a functor. More precisely, let $\Fin_{\ast}$ denote the category
of the pointed sets $\inp n=\pr{\{\ast,1,\dots,n\},\ast}$, where
$n\geq0$, and pointed maps between them. For each $n\geq0$, let
$\cal I_{n}$ denote the discrete category with $n$ objects $1,\dots,n$,
and let $\cal I_{n}^{\lcone}$ denote the category obtained by adjoining
an initial object $\infty$ to $\cal I_{n}$. There is a functor $\rho:\cal I_{n}^{\lcone}\to\Fin_{\ast}$
which carries the morphism $\infty\to i$ to the morphism $\rho^{i}:\inp n\to\inp 1$
such that $\pr{\rho^{i}}^{-1}\pr 1=\{i\}$. Given a cartesian monoidal
category $\cal A$, a commutative monoid in $\cal A$ is equivalent
to a functor $F:\Fin_{\ast}\to\cal A$ satisfying the \textbf{Segal
condition}: For each $n\ge0$, the composite
\[
\cal I_{n}^{\lcone}\xrightarrow{\rho}\Fin_{\ast}\to\cal A
\]
is a limit diagram. (If $n=0$, this means that $F\inp 0$ is the
terminal object.) The equivalence goes as follows: If $M$ is a commutative
monoid in $\cal A$, the corresponding functor is defined by $F\inp n=M^{n}$,
with obvious structure maps.

Whatever a symmetric monoidal $\infty$-category is, it should give
rise to a commutative monoid in $\Cat_{\infty}$. Therefore, it seems
reasonable to define a symmetric monoidal $\infty$-category as a
functor $F:N\pr{\Fin_{\ast}}\to\Cat_{\infty}$ satisfying the Segal
condition. This definition will solve problem (1); however, it does
not address problem (2), for to define such an $F$, we will have
to make \textit{specific }choices on tensor products, unit objects,
and coherent natural equivalences. The actual definition of symmetric
monoidal $\infty$-categories avoids this issue by using cocartesian
fibrations: A symmetric monoidal $\infty$-category is defined to
be a cocartesian fibration $p:\cal C^{\t}\to N\pr{\Fin_{\ast}}$ such
that, for every $n\geq0$, the composite
\[
N\pr{\cal I_{n}^{\lcone}}\to N\pr{\Fin_{\ast}}\xrightarrow{F}\Cat_{\infty}
\]
is a limit diagram, where $F$ denotes the functor classifying $p$. 

\subsection{Categorical Patterns}

We have seen that the straightening--unstraightening equivalence
is useful when we talk about $\infty$-categories with structures.
More precisely, to define ``$\infty$-categories with structures,''
we encode the structure as a condition on $\Cat_{\infty}$-valued
functor on some $\infty$-category $\cal D$, and then pass to the
corresponding cocartesian fibrations over $\cal D$.\footnote{The idea of presenting algebraic structures in terms of limit conditions
is reminiscent of the ideas of Lawvere's theories {\cite{Law63}}
and sketches {\cite[Section 4]{BW05}}. See {\cite{CH19}} for an
extensive account of the development of this idea in the $\infty$-categorical
setting.}

It turns out that a slightly more general construction is useful to
express a wider class of structures on $\infty$-categories. Instead
of considering cocartesian fibrations $\cal E\to\cal D$, we consider
functors which may not have cocartesian lifts on some morphisms of
$\cal D$. Intuitively, this amounts to considering a ``partial functor''
$\cal D\dasharrow\Cat_{\infty}$, a functor-like object that are defined
only on certain morphisms of $\cal D$. The theories of $\infty$-operads
\cite{HA}, $\Phi$-quasioperads (a generalization of $\infty$-operads)
\cite{Barwick18} and $\infty$-bicategories \cite{GH15}, are all
presented in this way.\footnote{These ideas culminate in Haugseng and Chu's weak Segal $\cal O$-fibrations
\cite{CH19}, which are also presented by categorical patterns.}

The formalism of partial functors is expressed in the language of
\textbf{categorical patterns} \cite[Appendix B]{HA}. A categorical
pattern on $\cal D$ is an additional data that specify which partial
functors $\cal D\dasharrow\Cat_{\infty}$ we wish to consider. Explicitly,
it is a pair\footnote{For convenience, we are simplifying the definition of categorical
patterns here. The definition we just made corresponds to a special
subclass of categorical patterns, called \textit{commutative} categorical
patterns in the main body of the paper. We will continue this simplification
until the end of the introduction.} $\frak P=\pr{M_{\cal D},\{p_{\alpha}:K_{\alpha}^{\lcone}\to\cal D\}_{\alpha\in A}}$,
where $M_{\cal D}$ is a set of edges of $\cal D$ containing all
equivalences, and each $p_{\alpha}$ is a diagram which carries the
edges of $K_{\alpha}^{\lcone}$ into $M_{\cal D}$. A functor of $\infty$-categories
$\cal E\to\cal D$ is said to be \textbf{$\frak P$-fibered} if it
roughly corresponds to a partial functor $F:\cal D\dasharrow\Cat_{\infty}$
defined on the edges in $M_{\cal D}$, such that for each $\alpha\in A$,
the composite $Fp_{\alpha}:K_{\alpha}^{\lcone}\to\Cat_{\infty}$ is
a limit diagram. The $\infty$-category of $\frak P$-fibered objects
and functors over $\cal D$ which preserve cocartesian edges over
$M_{\cal D}$ is denoted by $\frak P\-\Fib$ (Definition \ref{def:P-fib}).
\begin{example}
Consider the categorical pattern $\frak P=\pr{\{\text{all}\},\{\cal I_{n}^{\lcone}\to N\pr{\Fin_{\ast}}\}_{n\geq0}}$
on $N\pr{\Fin_{\ast}}$. A functor $\cal C^{\t}\to N\pr{\Fin_{\ast}}$
is $\frak P$-fibered if and only if it is a symmetric monoidal $\infty$-category.
\end{example}

\subsection{What this Paper is about}

Let $\frak P$ be a categorical pattern on an $\infty$-category $\cal D$.
We frequently want to consider a collection of $\frak P$-fibered
objects, parametrized by another $\infty$-category. The collection
$\{\pr{\Mod_{A},\otimes_{A}}\}_{A\in\sf{CRing}}$ of symmetric monoidal
categories of modules over commutative rings is one such example.\footnote{Examples of this sort, i.e., pseudofunctors with values in the $2$-category
of monoidal categories, are often called \textit{indexed monoidal
categories} and appear in various contexts, such as logic programming
and the study of monads {\cite{AnAn93, HoDe06, Shul13}}.} However, as in the case of $\Cat_{\infty}$-valued functors, realizing
such a collection as a functor $\cal C\to\frak P\-\Fib$ is often
inconvenient or unnatural. We thus ask the following question:
\begin{qst}
\label{que:1}Is there an analog of the Grothendieck construction
of functors with values in $\frak P\-\Fib$?
\end{qst}

In this paper, we answer Question \ref{que:1} by using $\frak P$-bundles.
To motivate the definition of $\frak P$-bundles, observe that a functor
$\cal C\to\frak P\-\Fib$ of $\infty$-categories determines a functor
$F:\cal C\to\Cat_{\infty}$ (by composing the forgetful functor $\frak P\-\Fib\to\Cat_{\infty}$)
and a natural transformation $F\to\delta\pr{\cal D}$, where $\delta\pr{\cal D}$
denotes the constant functor at $\cal D$. Under the straightening--unstraightening
equivalence, this corresponds to a functor $\cal E\to\cal C\times\cal D$
over $\cal C$, where $\cal E\to\cal C$ denote the unstraightening
of $F$. The definition of $\frak P$-bundles is an axiomatization
of functors arising in this way:
\begin{defn}
[Definition \ref{def:P-bundle}]A \textbf{$\frak P$-bundle }(\textbf{over
}$\cal C$) is a commutative diagram 
\[\begin{tikzcd}
	{\mathcal{X}} && {\mathcal{C}\times \mathcal{D}} \\
	& {\mathcal{C}}
	\arrow["p", from=1-1, to=1-3]
	\arrow["{\operatorname{pr}}", from=1-3, to=2-2]
	\arrow["q"', from=1-1, to=2-2]
\end{tikzcd}\]of simplicial sets which satisfies the following conditions:

\begin{enumerate}[label=(\alph*)]

\item The map $q:\cal X\to\cal C$ is a cocartesian fibration.

\item The map $p$ is a categorical fibration which preserves cocartesian
edges over $\cal C$.

\item For each object $C\in\cal C$, the map $\cal X_{C}=\cal X\times_{\cal C}\{C\}\to\cal D$
is $\frak P$-fibered.

\item For each morphism $f:C\to C'$ in $\cal C$, the induced functor
$f_{!}:\cal X_{C}\to\cal X_{C'}$ is a morphism of $\frak P$-fibered
objects.

\end{enumerate}
\end{defn}

\begin{rem}
The definition of $\frak P$-bundles is similar to that of families
of $\infty$-operads (\cite[Definition 2.3.1.10]{HA}). In fact, if
$\frak{Op}$ denotes the categorical pattern for $\infty$-operads,
every $\frak{Op}$-bundle is a family of $\infty$-operads.
\end{rem}

We may understand $\frak P$-bundles as a \textit{relative }version
of cocartesian fibrations: In its crudest form, it is just a morphism
of cocartesian fibrations over $\cal C$. As such, $\frak P$-bundles
are easier to handle and more natural than functors taking values
in $\frak P\-\Fib$, as is already implicit from the widespread use
of families of $\infty$-operads in \cite{HA}.

Let $\frak P\-\Bund\pr{\cal C}$ denote the $\infty$-category of
$\frak P$-bundles over $\cal C$ and functors over $\cal C\times\cal D$
which preserves cocartesian edges over $\cal C\times\cal D$ whose
images in $\cal D$ are marked by $\frak P$ (Definition \ref{def:eq_P-bund}).\textbf{
}The following theorem is our answer to Question \ref{que:1}:
\begin{thm}
[Corollary \ref{cor:main}]\label{thm:intro}The straightening-unstraightening
equivalence lifts to a categorical equivalence
\[
\frak P\-\Bund\pr{\cal C}\simeq\Fun\pr{\cal C,\frak P\-\Fib}.
\]
\end{thm}

\begin{rem}
Both $\infty$-categories appearing in Theorem \ref{thm:intro} can
be presented by model categories. This suggests that the equivalence
of the theorem lifts to a Quillen equivalence, and this is what we
will actually prove. Strictly speaking, proving this stronger statement
is not necessary if one is only interested in Theorem \ref{thm:intro}.
However, it has several distinctive advantages. The biggest advantage
is that many applications of Theorem \ref{thm:intro} (Section \ref{sec:applications})
find concise formulations in the language of model categories. Also,
it is natural to prove an equivalence of underlying $\infty$-categories
of model categories by establishing a Quillen equivalence. In addition,
Quillen equivalences give us very explicit equivalences. With all
these benefits, the author decided that a model-categorical approach
is worth a detour.
\end{rem}

\subsection{What is it Good for?}

The straightening--unstraightening equivalence has various applications
in proving results on cocartesian fibrations and $\Cat_{\infty}$-valued
functors. Theorem \ref{thm:intro} offers generalizations of these
results for $\frak P$-bundles and $\frak P\-\Fib$-valued functors.
In this paper, we will focus on two of them:

\begin{enumerate}[label=(\Roman*)]

\item (\textbf{Structure Theory of $\frak P$-bundles}) Just like
there is a universal cocartesian fibration, there is a universal $\frak P$-bundle
which classifies all $\frak P$-bundles (Subsection \ref{subsec:Classification}).

\item (\textbf{Limits and Colimits of $\frak P$-Fibered Objects})
Theorem \ref{thm:intro} can be used to facilitate computations of
limits and colimits in $\frak P\-\Fib$. For example, there are explicit
formulas for limits and colimits in $\frak P\-\Fib$ (Corollaries
\ref{cor:limit_formula} and \ref{cor:colimit_formula}) and criteria
for a diagram in $\frak P\-\Fib$ to be a limit or a colimit diagram
in terms of the associated bundle (Proposition \ref{prop:3.3.3.1}
and \ref{prop:3.3.4.2}). 

\end{enumerate}

The colimit criterion in (II) is closely related to Lurie's theory
of \textit{assembly of $\infty$-operads} \cite[$\S$ 2.3]{HA}. In
our future work, we will use this observation to show that a certain
diagram in the $\infty$-category of $\infty$-operads is a colimit
diagram.

\settocdepth{subsection}

\subsection*{Outline of the Paper}

We start by establishing basic facts on categorical patterns in Section
\ref{sec:catpat}. In Section \ref{sec:P-bundles}, we introduce $\frak P$-bundles,
the main subject of this paper. Section \ref{sec:St_Un} is devoted
to the review of the straightening--unstraightening equivalence.
A particular emphasis is placed on the explicit description of the
unstraightening functor, which seems to be lacking in the literature.
After these preparations, we will prove the main theorem of this paper
(Theorem \ref{thm:intro}) in Section \ref{sec:Rectification}. Section
\ref{sec:applications} discusses applications of the main theorem:
We will construct the universal $\frak P$-bundle, and explain how
to compute limits and colimits of $\frak P$-bundles.

\subsection*{Notation and Terminology}

\label{subsec:not_term}

We will mainly adopt the terminology of \cite{HTT} and \cite{HA},
with the following exceptions:
\begin{itemize}
\item If $\cal C$ is an $\infty$-category, we will refer to its maximal
sub Kan complex as the \textbf{core} of $\cal C$ and denote it by
$\cal C^{\simeq}$. Equivalently, the core of $\cal C$ is its subcategory
spanned by the equivalences.
\item The symbol $\SS$ denotes the category of simplicial sets and $\SS^{+}$
the category of marked simplicial sets. We sometimes regard these
categories as simplicial categories; $\SS$ is endowed with the enrichment
by its internal hom, and $\SS^{+}$ is enriched by the mapping complex
$\Map^{\sharp}\pr{-,-}$ of \cite[$3.1.3$]{HTT}.
\item Following Joyal, we will refer to the weak equivalences of the Joyal
model structure \cite[$\S$2.2.5]{HTT} on $\SS$ as the \textbf{weak
categorical equivalences}.
\item If $\overline{X}=\pr{X,M}$ is a marked simplicial set, we let $\overline{X}_{\flat}$
denote the simplicial set $X$ and call it the \textbf{underlying
simplicial set} of $\overline{X}$. If $F:\cal C\to\SS^{+}$ is a
simplicial functor , then we let $F_{\flat}$ denote the simplicial
functor $\cal C\to\SS$ given by $C\mapsto F\pr C_{\flat}$. 
\item Given a simplicial functor $F:\cal C\to\SS$, we define a simplicial
functor $F^{\sharp}:\cal C\to\SS^{+}$ by $F^{\sharp}\pr C=F\pr C^{\sharp}$.
If $F$ takes values in the full simplicial subcategory spanned by
the $\infty$-categories, and if for each pair of objects $X,Y\in\cal C$,
the map $\cal C\pr{X,Y}\to\Fun\pr{F\pr X,F\pr Y}$ factors through
the core of $\Fun\pr{F\pr X,F\pr Y}$, we will write $F^{\natural}:\cal C\to\SS^{+}$
for the simplicial functor defined by $F^{\natural}\pr C=F\pr C^{\natural}$. 
\item If $\cal C$ and $\cal D$ are simplicial categories, we will write
$\Fun^{s}\pr{\cal C,\cal D}$ for the category of simplicial functors
$\cal C\to\cal D$ and simplicial natural transformations between
them. 
\item Following \cite{Nguyen2019}, we will say that a morphism is \textbf{marked
right anodyne} if it is marked anodyne in the sense of \cite[$\S$3.1]{HTT}.
Morphsims of marked simplicial sets having the right lifting property
for marked right anodyne extensions will be called \textbf{marked
right fibrations}. The opposite of a marked right anodyne map will
be called a \textbf{marked left anodyne} map, and \textbf{marked left
fibrations} are defined similarly.
\item For integers $i\leq j$, we will write $P_{i,j}$ for the poset of
subsets of $\bb Z$ whose minimum element is $i$ and whose maximum
element is $j$, ordered by inclusion. We will write $\widetilde{\fr C}[\Delta^{n}]$
for the simplicial category whose hom-simplicial sets are given by
$\widetilde{\fr C}[\Delta^{n}]\pr{i,j}=N\pr{P_{i,j}}^{\op}$, with
composition induced by inclusion.
\item By the \textbf{homotopy coherent nerve} of a simplicial category $\cal C$,
we mean the simplicial set $N\pr{\cal C}$ whose $n$-simplices are
the simplicial functors $\widetilde{\fr C}[\Delta^{n}]\to\cal C$.
(Note that this convention is the opposite of the one adopted in \cite{HTT},
and it agrees with the one in \cite{kerodon}).\footnote{The choice of the convention essentially boils down to whether one
wants to prioritize covariant unstraightening/straightening (i.e.,
classification of cocartesian fibrations) or contravariant unstraightening/straightening.
Our convention is better suited for covariant unstraightening.} The association $\cal C\mapsto N\pr{\cal C}$ determines a functor
from the category of small simplicial categories to $\SS$. Its left
adjoint will be denoted by $\widetilde{\fr C}[-]$.
\item We will often indicate a model structure by using subscripts. For
example, if $\cal C$ is a combinatorial model category and $\cal I$
is a small category, then $\Fun\pr{\cal I,\cal C}_{{\rm proj}}$ denotes
the model category equipped with the projective model structure.
\item If $K$ is a simplicial set, we let $\infty$ denote the cone point
of the simplicial sets $K^{\lcone}$ and $K^{\rcone}$.
\end{itemize}

\subsection*{Acknowledgment}

The author appreciates Daisuke Kishimoto and Mitsunobu Tsutaya for
their constant support and encouragement. He also appreciates Ryo
Horiuchi and Takumi Maegawa for commenting on earlier drafts of this
paper.

\section{\label{sec:catpat}Categorical Patterns}

This section is about categorical patterns. In Subsection \ref{subsec:recollection},
we will review basic notions related to categorical patterns. A categorical
pattern $\frak P$ on a simplicial set $S$ is an additional datum
on $S$ which, among other things, makes $S$ into a marked simplicial
set $\overline{S}$. Each categorical pattern $\frak P$ on $S$ gives
rise to a model structure on $\SS_{/\overline{S}}^{+}$, denoted by
$\SS_{/\frak P}^{+}$. We recall what the fibrant objects of this
model structure are. In Subsection \ref{subsec:recognition}, we will
introduce special classes of categorical patterns for which fibrations
and weak equivalences of the associated model structure can be recognized
easily. Finally, in Subsection \ref{subsec:cat_inv}, we consider
when categorical equivalences of $\infty$-categories equipped with
categorical patterns induce Quillen equivalences of the associated
model structure.

\subsection{\label{subsec:recollection}Recollection}

In this subsection, we briefly review the categorical patterns and
related notions, following \cite[Appendix B]{HA}.
\begin{defn}
\cite[Definition B.0.19, Remark B.0.26]{HA}\label{def:categorical_pattern}
Let $S$ be a simplicial set. A \textbf{categorical pattern} on $S$
is a triple $\frak P=\pr{M_{S},T,\{p_{\alpha}:K_{\alpha}^{\lcone}\to S\}_{\alpha\in A}}$,
where:
\begin{itemize}
\item $M_{S}$ is a set of edges of $S$ containing all degenerate edges.
\item $T$ is a set of $2$-simplices of $S$ containing all degenerate
$2$-simplices.
\item $\{p_{\alpha}:K_{\alpha}^{\lcone}\to S\}_{\alpha\in A}$ is a (small)
set of maps of simplicial sets such that for each $\alpha\in A$,
the map $p_{\alpha}$ carries each edge of $K_{\alpha}^{\lcone}$
into $M_{S}$ and each $2$-simplex of $K_{\alpha}^{\lcone}$ into
$T$.
\end{itemize}
If $\frak P'=\pr{M'_{S},T',\{p'_{\beta}:K_{\beta}^{\lcone}\to S\}_{\beta\in B}}$
is another categorical pattern on $S$, we will write $\frak P\subset\frak P'$
to mean that $M_{S}\subset M_{S}'$, that $T\subset T'$, and that
for each $\alpha\in A$, there is some $\beta\in B$ such that $p_{\alpha}=p'_{\beta}$.

A \textbf{marked simplicial set over} $\frak P$ is a map $\pr{X,M}\to\pr{S,M_{S}}=\overline{S}$
of marked simplicial sets. We let $\SS_{/\frak P}^{+}$ denote the
category of marked simplicial sets over $\frak P$. 

A marked simplicial set $\overline{X}=\pr{X,M}\in\SS_{/\frak P}^{+}$
over $\frak P$ is said to be \textbf{$\frak P$-fibered} if the following
conditions are satisfied:
\begin{itemize}
\item [(1)]The map $p:X\to S$ of simplicial sets is an inner fibration.
\item [(2)]For each marked edge $\Delta^{1}\to S$, the induced map $p':X\times_{S}\Delta^{1}\to\Delta^{1}$
is a cocartesian fibration.
\item [(3)]An edge $e$ of $X$ belongs to $M$ if and only if $p\pr e$
belongs to $M_{S}$ and $e$ is locally $p$-cocartesian.
\item [(4)]Given a commutative diagram 
\[\begin{tikzcd}
	{\Delta^{\{0,1\}}} & X \\
	{\Delta^2} & S
	\arrow[from=1-1, to=2-1]
	\arrow["e", from=1-1, to=1-2]
	\arrow["p", from=1-2, to=2-2]
	\arrow["\sigma"', from=2-1, to=2-2]
\end{tikzcd}\]with $e\in M$ and $\sigma\in T$, the induced edge of $X\times_{S}\Delta^{2}$
is $p'$-cocartesian, where $p':X\times_{S}\Delta^{2}\to\Delta^{2}$
denotes the projection.
\item [(5$'$)]For each $\alpha\in A$ and every cocartesian section $s_{0}$
of $X\times_{S}K_{\alpha}\to K_{\alpha}$, there is a cocartesian
section $s$ of $X\times_{S}K_{\alpha}^{\lcone}\to K_{\alpha}^{\lcone}$
which extends $s_{0}$. (See Remark \ref{rem:rezk}.)
\item [(6)]For each index $\alpha\in A$ and every cocartesian section
$s:K_{\alpha}^{\lcone}\to X\times_{S}K_{\alpha}^{\lcone}$ of the
projection $X\times_{S}K_{\alpha}^{\lcone}\to K_{\alpha}^{\lcone}$,
the composite 
\[
K_{\alpha}^{\lcone}\xrightarrow{s}X\times_{S}K_{\alpha}^{\lcone}\to X
\]
is a $p$-limit diagram in $X$. (See Remarks \ref{rem:relative_limit}
and \ref{rem:rezk}.)
\end{itemize}
\end{defn}

\begin{rem}
\label{rem:relative_limit}In \cite{HTT}, relative limits are defined
only for inner fibrations between $\infty$-categories, so condition
(6) of Definition \ref{def:categorical_pattern} needs an elaboration.
Given an inner fibration $p:X\to S$ of simplicial sets and a diagram
$\overline{f}:K^{\rcone}\to X$, we say that $\overline{f}$ is a
\textbf{$p$-limit }diagram if the map
\[
X_{/\overline{f}}\to X_{/f}\times_{S_{/pf}}S_{/p\overline{f}}
\]
is a trivial fibration, where $f=\overline{f}\vert K$. 
\end{rem}

\begin{rem}
\label{rem:rezk}Let $\frak P$ be a categorical pattern on a simplicial
set $S$ and let $\pr{X,M}\in\SS_{/\frak P}^{+}$ be an object satisfying
conditions (1), (2), (3), and (4) of Definition \ref{def:categorical_pattern}.
Then for each map $K\to S$ of simplicial sets which maps every edge
into $M_{S}$ and every $2$-simplex into $T$, the map
\[
\overline{X}\times_{\overline{S}}K^{\sharp}\to K^{\sharp}
\]
is a marked left fibration. In other words, the map $p':X\times_{S}K\to K$
is a cocartesian fibration, and an edge of $X\times_{S}K$ is $p'$-cocartesian
if and only if its image in $X$ belongs to $M$. (In particular,
conditions (5$'$) and (6) make sense.) This follows from the following
more general assertion:
\begin{itemize}
\item [($\ast$)]Let $p:X\to S$ be an inner fibration of simplicial sets,
and let $e:x\to y$ be an edge of $X$. Suppose that, for each commutative
diagram 
\[\begin{tikzcd}
	{\Delta^{\{0,1\}}} & X \\
	{\Delta^2} & S,
	\arrow[from=1-1, to=2-1]
	\arrow["e", from=1-1, to=1-2]
	\arrow["p", from=1-2, to=2-2]
	\arrow["\sigma"', from=2-1, to=2-2]
\end{tikzcd}\]the induced edge $e_{\sigma}:\Delta^{1}\to X\times_{S}\Delta^{2}$
is a cartesian edge over $\Delta^{2}$. Then $e$ is $p$-cartesian.
\end{itemize}
To prove ($\ast$), we must show that the map $\theta:X_{e/}\to X_{x/}\times_{S_{p\pr x/}}S_{p\pr e/}$
is a trivial fibration. Since $\theta$ is a left fibration, it suffices
to show that its fibers are contractible. So let $\pr{f,\sigma}\in X_{x/}\times_{S_{p\pr x/}}S_{p\pr e/}$
be an arbitrary vertex, where $f:x\to z$ is an edge of $X$ and $\sigma$
is a $2$-simplex of $S$ such that $d_{1}\sigma=p\pr f$ and $d_{2}\sigma=p\pr e$.
We must show that $\theta^{-1}\pr{f,\sigma}$ is contractible. Set
$X'=X\times_{S}\Delta^{2}$. Since the square 
\[\begin{tikzcd}
	{X'_{e_\sigma/}} & {X_{e/}} \\
	{X'_{(x,0)/}\times _{\Delta^2_{0/}}\Delta^2_{0\to1/}} & {X_{x/}\times _{S_{p(x)/}}S_{p(e)/}}
	\arrow[from=1-1, to=1-2]
	\arrow["\theta", from=1-2, to=2-2]
	\arrow[from=2-1, to=2-2]
	\arrow["{\theta'}"', from=1-1, to=2-1]
\end{tikzcd}\]is cartesian, there is an isomorphism of simplicial sets $\pr{\theta^{\p}}^{-1}\pr{\pr{f,0\to2},\id_{\Delta^{2}}}\cong\theta^{-1}\pr{f,\sigma}$.
By hypothesis, the map $\theta'$ is a trivial fibration, so its fibers
are contractible. Hence $\theta^{-1}\pr{f,\sigma}$ is contractible,
as required.
\end{rem}

\begin{rem}
\cite[Remark B.0.23]{HA}\label{rem:B.0.23} Let $\frak P$ be a categorical
pattern on a simplicial set $S$ and let $p:X\to S$ be an inner fibration
of simplicial sets. If there is a marking on $X$ with respect to
which $p$ is $\frak P$-fibered, then condition (3) of Definition
\ref{def:categorical_pattern} completely determines the marked edges
of $X$; they are the locally $p$-cocartesian morphisms over the
marked edges of $S$. Because of this, we will say that $p$ (or $X$,
if $p$ is clear from the context) is \textbf{$\frak P$-fibered}
if there is a marking on $X$ which makes $p$ into a $\frak P$-fibered
object. If $X\to S$ and $Y\to S$ are $\frak P$-fibered maps, then
a map $X\to Y$ over $S$ is called a \textbf{morphism of $\frak P$-fibered
objects} if it preserves locally cocartesian morphisms over the marked
edges of $S$.
\end{rem}

\begin{rem}
\label{rem:mlfib}Some of the defining conditions of $\frak P$-fibered
objects can be stated more succinctly. Let $\frak P=\pr{M_{S},T,\{p_{\alpha}:K_{\alpha}^{\lcone}\to S\}_{\alpha\in A}}$
be a categorical pattern on a simplicial set $S$ and let $\overline{X}\in\SS_{/\frak P}^{+}$
be an object. Using Remark \ref{rem:rezk}, we find that $\overline{X}$
satisfies conditions (5') and (6) of Definition \ref{def:categorical_pattern}
if and only if it satisfies the following condition:
\begin{itemize}
\item [(B)]Every map $\pr{K_{\alpha}}^{\sharp}\to\overline{X}$ over $S$
extends to $\pr{K_{\alpha}^{\lcone}}^{\sharp}$, and any such extension
is a $p$-limit diagram.
\end{itemize}
Suppose that $T$ contains every $2$-simplex $\Delta^{2}\to S$ whose
restriction to $\Delta^{\{0,1\}}$ belongs to $M_{S}$. (For instance,
this is true if $\frak P$ is commutative in the sense of Definition
\ref{def:creative_commutative}.) Then by assertion ($\ast$) of Remark
\ref{rem:rezk} and \cite[Proposition 3.1.1.6]{HTT}, we deduce that
$\overline{X}$ satisfies conditions (1), (2), (3), and (4) if and
only if it satisfies the following condition:
\begin{itemize}
\item [(A)]The map $\overline{X}\to\pr{S,M_{S}}$ is a marked left fibration.
\end{itemize}
\end{rem}

\begin{rem}
Let $\frak P=\pr{M_{S},T,\{p_{\alpha}\}_{\alpha\in A}}$ be a categorical
pattern on a simplicial set $S$ and let $\frak P'$ be another categorical
pattern obtained from $\frak P$ by replacing $T$ with the set of
all $2$-simplices of $S$. Suppose that $T$ contains every $2$-simplex
$\sigma$ such that $\sigma\vert\Delta^{\{0,1\}}$ belongs to $M_{S}$.
Then the model structures on $\SS_{/\frak P}^{+}$ and $\SS_{/\frak P'}^{+}$
are identical, because they have the same class of cofibrations and
fibrant objects. This follows from Remark \ref{rem:mlfib}.
\end{rem}

\begin{rem}
\cite[Remark B.0.25]{HA}\label{rem:enrichment} Let $\frak P$ be
a categorical pattern on a simplicial set $S$. Given objects $\o X,\o Y\in\SS_{/\frak P}^{+}$,
we define a simplicial set $\Map_{S}^{\sharp}\pr{\o X,\o Y}$ as follows:
Its $n$-simplex is a morphism $\pr{\Delta^{n}}^{\sharp}\times\o X\to\o Y$
of marked simplicial sets over $\pr{S,M_{S}}$. This makes $\SS_{/\frak P}^{+}$
into a simplicial category. 
\end{rem}

\begin{rem}
Let $\frak P$ be a categorical pattern on a simplicial set $S$.
There is another simplicial enrichment $\Map_{S}^{\flat}\pr{-,-}$
of $\SS_{/\frak P}^{+}$, where an $n$-simplex of $\Map_{S}^{\flat}\pr{\o X,\o Y}$
is a map $\pr{\Delta^{n}}^{\flat}\times\o X\to\o Y$ of marked simplicial
sets over $\pr{S,M_{S}}$. However, we rarely use this enrichment
in this paper. Because of this, we will always understand that $\SS_{/\frak P}^{+}$
carries the simplicial enrichment of Remark \ref{rem:enrichment}.
\end{rem}

\begin{example}
\label{exa:pb_cat_pat}Let $\frak P=\pr{M_{S},T,\{p_{\alpha}:K_{\alpha}^{\lcone}\to S\}_{\alpha\in A}}$
be a categorical pattern on a simplicial set $S$. Given an object
$f:\pr{X,M}\to\pr{S,M_{S}}\in\SS_{/\frak P}^{+}$, we will write $f^{*}\frak P$
for the categorical pattern on $X$ whose set of edges is $M$, whose
set of $2$-simplices is $f^{-1}\pr T$, and whose set of diagrams
consists of the diagrams $K_{\alpha}^{\lcone}\to X$ which lifts $p_{\alpha}$
for some $\alpha\in A$. More generally, if $g:Y\to S$ is a map of
simplicial sets, we will write $g^{*}\frak P$ for the categorical
pattern obtained by applying the above procedure to the object $\pr{Y,g^{-1}\pr{M_{S}}}\to\pr{S,M_{S}}\in\SS_{/\frak P}^{+}$.
\end{example}

The main result of \cite[Appendix B]{HA} asserts the following:
\begin{thm}
\label{thm:B.0.20}\cite[Theorem B.0.20]{HA} Let $\frak P$ be a
categorical pattern on a simplicial set $S$. There is a combinatorial
model structure on $\SS_{/\frak P}^{+}$, which is uniquely characterized
by the following properties:
\begin{enumerate}
\item A morphism is a cofibration if and only if it is a monomorphism.
\item An object $\overline{X}\in\SS_{/\frak P}^{+}$ is fibrant if and only
if it is $\frak P$-fibered.
\end{enumerate}
Moreover, the model structure is simplicial with respect to the simplicial
enrichment of Remark \ref{rem:enrichment}.
\end{thm}

From now on, we will understand that given a categorical pattern $\frak P$,
the category $\SS_{/\frak P}^{+}$ is equipped with the model structure
of Theorem \ref{thm:B.0.20}.

With Theorem \ref{thm:B.0.20} in mind, we make the following definitions.
\begin{defn}
\cite[Definition B.2.1]{HA} Let $\frak P$ be a categorical pattern
on a simplicial set $S$. We say that a morphism $f:\overline{X}\to\overline{Y}$
in $\SS_{/\frak P}^{+}$ is a \textbf{$\frak P$-equivalence} if it
is a weak equivalence of $\SS_{/\frak P}^{+}$, i.e., for each $\frak P$-fibered
object $\overline{Z}\in\SS_{/\frak P}^{+}$, the map
\[
\Map_{S}^{\sharp}\pr{\overline{Y},\overline{Z}}\to\Map_{S}^{\sharp}\pr{\overline{X},\overline{Z}}
\]
is a homotopy equivalence of Kan complexes.
\end{defn}

\begin{defn}
\label{def:P-fib}Let $\frak P$ be a categorical pattern on a simplicial
set $S$. We will write $\frak P\-\Fib$ for the homotopy coherent
nerve of the full simplicial subcategory of $\SS_{/\frak P}^{+}$
spanned by the fibrant--cofibrant objects. 
\end{defn}

We conclude this subsection with a certain stability property of marked
edges of $\frak P$-fibered objects.
\begin{prop}
\label{prop:A_0}Let $\frak P=\pr{M_{S},T,\{p_{\alpha}\}_{\alpha}}$
be a categorical pattern on a simplicial set $S$ and let $\overline{X}=\pr{X,M}\in\SS_{/\frak P}^{+}$
be a $\frak P$-fibered object. Suppose we are given a $2$-simplex
$\sigma$ of $X$, which we depict as
\[\begin{tikzcd}
	& y \\
	x && z.
	\arrow["f", from=2-1, to=1-2]
	\arrow["g", from=1-2, to=2-3]
	\arrow["h"', from=2-1, to=2-3]
\end{tikzcd}\]Let $p:X\to S$ denote the projection. Suppose that $p\pr{\sigma}$
belongs to $T$, that $p\pr f,p\pr g,p\pr h$ belong to $M_{S}$,
and that $f$ belongs to $M$. Then $g$ belongs to $M$ if and only
if $h$ belongs to $M$.
\end{prop}

\begin{proof}
Regard $\pr{\Delta^{2}}^{\sharp}$ as an object of $\SS_{/\frak P}^{+}$
using the map $\sigma$. It will suffice to show that the inclusions
$\pr{\Lambda_{1}^{2}}^{\sharp}\cup\pr{\Delta^{2}}^{\flat}\subset\pr{\Delta^{2}}^{\sharp}$
and $\pr{\Lambda_{0}^{2}}^{\sharp}\cup\pr{\Delta^{2}}^{\flat}\subset\pr{\Delta^{2}}^{\sharp}$
are trivial cofibrations of $\SS_{/\frak P}^{+}$. According to \cite[Proposition B.1.6]{HA}
and \cite[Lemma B.1.11]{HA}, these maps are $\frak P$-anodyne extensions,
which form a subclass of $\frak P$-equivalences (\cite[Example B.2.2]{HA}).
The claim follows.
\end{proof}

\subsection{\label{subsec:recognition}Recognizing Fibrations and Weak Equivalences
of \texorpdfstring{$\mathsf{sSet}^+_{/\mathfrak{P}}$}{}}

Let $\frak P=\pr{M_{\cal D},T,\{p_{\alpha}:K_{\alpha}^{\lcone}\to\cal D\}}$
be a categorical pattern on an $\infty$-category $\cal D$. Suppose
that $M_{\cal D}$ contains every equivalence of $\cal D$ and $T$
contains every $2$-simplex $\sigma$ such that $\sigma\vert\Delta^{\{0,1\}}$
is an equivalence. Then for every fibrant object $\overline{\cal X}=\pr{\cal X,M}\in\SS_{/\frak P}^{+}$,
the map $p:\cal X\to\cal D$ of simplicial sets is automatically a
categorical fibration (by ($\ast$) of Remark \ref{rem:rezk}), and
$\overline{\cal X}$ is weakly terminal if and only if $p$ is a trivial
fibration. This naturally leads to the following question: To what
extent can fibrations and weak equivalences of fibrant objects of
$\SS_{/\frak P}^{+}$ be detected by their underlying morphisms of
simplicial sets? In this subsection, we will introduce a class of
categorical patterns for which there is a complete answer to this
question.
\begin{defn}
\label{def:creative_commutative}Let $\frak P=\pr{M_{\cal D},T,\{p_{\alpha}:K_{\alpha}^{\lcone}\to\cal D\}_{\alpha\in A}}$
be a categorical pattern on an $\infty$-category $\cal D$. We say
that $\frak P$ is \textbf{creative} if the set $M_{\cal D}$ contains
every equivalence of $\cal D$ and the set $T$ contains every $2$-simplex
$\sigma$ such that $\sigma\vert\Delta^{\{0,1\}}$ is an equivalence.
If further $T$ contains \textit{all} $2$-simplices of $S$, we say
that $\frak P$ is \textbf{commutative}. If $\frak P$ is commutative,
we will omit $T$ from the notation and simply say that $\frak P=\pr{M_{\cal D},\{p_{\alpha}\}_{\alpha\in A}}$
is a commutative categorical pattern.
\end{defn}

\begin{rem}
Almost every categorical pattern which appears in nature is commutative. 
\end{rem}

The following result, which is the main result of this subsection,
asserts that for creative categorical patterns, fibrations and weak
equivalences are created (i.e., preserved and reflected) by the forgetful
functor $\SS_{/\frak P}^{+}\to\SS_{{\rm Joyal}}$, hence justifying
our terminology.
\begin{prop}
\label{prop:creative}Let $\frak P$ be a creative categorical pattern
on an $\infty$-category $\cal D$, and let $f:\overline{\cal X}\to\overline{\cal Y}$
be a morphism between fibrant objects of $\SS_{/\frak P}^{+}$. Then:
\begin{enumerate}
\item The map $f$ is a fibration of $\SS_{/\frak P}^{+}$ if and only if
its underlying morphism of simplicial sets is a categorical fibration.
\item The map $f$ is a weak equivalence of $\SS_{/\frak P}^{+}$ if and
only its underlying morphism of simplicial sets is a categorical equivalence.
\end{enumerate}
\end{prop}

\begin{rem}
The creativity of $\frak P$ is essential in Proposition \ref{prop:creative}.
For example, let $J$ denote the nerve of the groupoid with two objects
$0$ and $1$, and with exactly one morphisms between each pair of
objects. Consider the categorical pattern $\frak P$ on $J$ which
consists of degenerate edges, degenerate $2$-simplices, and no diagram.
Then any functor $\cal X\to J$ of $\infty$-categories is $\frak P$-fibered.
So a $\frak P$-fibered map is usually not a categorical fibration.
The inclusion $\{0\}\subset J$ is a categorical equivalence between
$\frak P$-fibered objects, but it is not a weak equivalence in $\SS_{/\frak P}^{+}$
because there is no morphism $J\to\{0\}$ over $J$.
\end{rem}

Assertion (1) of Proposition \ref{prop:creative} is proved in \cite[Proposition B.2.7]{HA},
so we shall focus on (2). For this, we will need the following lemma.
(Compare \cite[Remark 3.1.3.1]{HTT}.) 
\begin{lem}
\label{lem:Map_S^=00005Csharp_core}Let $\frak P=\pr{M_{S},T,\{p_{\alpha}\}_{\alpha\in A}}$
be a categorical pattern on a simplicial set $S$, and let $\overline{X}$
and $\overline{Y}$ be objects of $\SS_{/\frak P}^{+}$. Suppose that
$\overline{Y}$ is fibrant. Then the simplicial set $\Map_{S}^{\flat}\pr{\overline{X},\overline{Y}}$
is an $\infty$-category, and its core is the Kan complex $\Map_{S}^{\sharp}\pr{\overline{X},\overline{Y}}$.
\end{lem}

\begin{proof}
We will write $p:X\to S$ and $q:Y\to S$ for the projections. First
we show that $\Map_{S}^{\flat}\pr{\overline{X},\overline{Y}}$ is
an $\infty$-category. Since $\overline{Y}$ is fibrant, the map $q$
is an inner fibration. Therefore, the map $\Fun\pr{X,Y}\to\Fun\pr{X,S}$
is also an inner fibration. Hence the map $\Map^{\flat}\pr{\overline{X},\overline{Y}}\to\Map^{\flat}\pr{\overline{X},\overline{S}}$
is an inner fibration, so its fiber $\Map_{S}^{\flat}\pr{\overline{X},\overline{Y}}$
is an $\infty$-category. 

Next, we show that the core of $\Map_{S}^{\flat}\pr{\overline{X},\overline{Y}}$
is equal to $\Map_{S}^{\sharp}\pr{\overline{X},\overline{Y}}$. Since
$\SS_{/\frak P}^{+}$ is a simplicial model category, the simplicial
set $\Map_{S}^{\sharp}\pr{\overline{X},\overline{Y}}$ is a Kan complex.
Therefore, $\Map_{S}^{\sharp}\pr{\overline{X},\overline{Y}}$ is contained
in the core of $\Map_{S}^{\flat}\pr{\overline{X},\overline{Y}}$.
To prove the reverse inclusion, we must prove the following:
\begin{enumerate}
\item $\Map_{S}^{\sharp}\pr{\overline{X},\overline{Y}}$ is a subcategory
of $\Map_{S}^{\flat}\pr{\overline{X},\overline{Y}}$ in the sense
of \cite[1.2.11]{HTT}. In other words, the inclusion $\Map_{S}^{\sharp}\pr{\overline{X},\overline{Y}}\subset\Map_{S}^{\flat}\pr{\overline{X},\overline{Y}}$
is an inner fibration.
\item Every equivalence of $\Map_{S}^{\flat}\pr{\overline{X},\overline{Y}}$
belongs to $\Map_{S}^{\sharp}\pr{\overline{X},\overline{Y}}$.
\end{enumerate}
For assertion (1), it suffices to show that for every $0<i<n$, the
inclusion $\pr{\pr{\Lambda_{i}^{n}}^{\sharp}\cup\pr{\Delta^{n}}^{\flat}}\times\overline{X}\to\pr{\Delta^{n}}^{\sharp}\times\overline{X}$
is a $\frak P$-equivalence. This follows from \cite[Remark B.2.5]{HA}.

For assertion (2), let $h:\pr{\Delta^{1}}^{\flat}\times\overline{X}\to\overline{Y}$
be an equivalence of $\Map_{S}^{\flat}\pr{\overline{X},\overline{Y}}$.
We must show that $h$ determines a map $\pr{\Delta^{1}}^{\sharp}\times\overline{X}\to\overline{Y}$
of marked simplicial sets. Let $f:0\to1$ denote the unique nondegenerate
edge of $\Delta^{1}$ and let $g:x\to x'$ be a marked edge of $\overline{X}$.
We wish to show that the edge $h\pr{f,g}$ is marked in $\overline{Y}$.
Consider the $2$-simplex $\sigma$ of $Y$, depicted as 
\[\begin{tikzcd}
	& {h(0,x')} \\
	{h(0,x)} && {h(1,x').}
	\arrow["{h(0,g)}", from=2-1, to=1-2]
	\arrow["{h(f,x')}", from=1-2, to=2-3]
	\arrow["{h(f,g)}"', from=2-1, to=2-3]
\end{tikzcd}\]By hypothesis, the edge $h\pr{0,g}$ is marked in $Y$. The simplex
$q\pr{\sigma}$ is a degeneration of $p\pr g$, so $q\pr{\sigma}$
belongs to $T$ and its boundary consists of the edges in $M_{S}$.
Therefore, by Proposition \ref{prop:A_0}, it suffices to show that
$h\pr{f,x'}$ is marked in $Y$. Since $h$ is an equivalence, its
image in the $\infty$-category $\Map_{S}^{\flat}\pr{\{x'\}^{\sharp},\overline{Y}}\cong Y\times_{S}\{q(x')\}$
is an equivalence. In other words, $h\pr{f,x'}$ is an equivalence
in a fiber of $q$. In particular, it is a locally $q$-cocartesian
morphism lying over a marked edge of $S$. Therefore, it is marked
in $\overline{Y}$, as required.
\end{proof}
\begin{cor}
\label{cor:P-equivalence}Let $\frak P$ be a categorical pattern
on a simplicial set $S$ and let $f:\overline{X}\to\overline{Y}$
be a morphism of $\frak P$. The following conditions are equivalent:
\begin{enumerate}
\item For each fibrant object $\overline{Z}\in\SS_{/\frak P}^{+}$, the
map
\[
\Map_{S}^{\flat}\pr{\overline{Y},\overline{Z}}\to\Map_{S}^{\flat}\pr{\overline{X},\overline{Z}}
\]
is a categorical equivalence.
\item For each fibrant object $\overline{Z}\in\SS_{/\frak P}^{+}$, the
map
\[
\Map_{S}^{\sharp}\pr{\overline{Y},\overline{Z}}\to\Map_{S}^{\sharp}\pr{\overline{X},\overline{Z}}
\]
is a homotopy equivalence.
\end{enumerate}
\end{cor}

\begin{proof}
The implication (1)$\implies$(2) follows from Proposition \ref{lem:Map_S^=00005Csharp_core}.
For the converse, we must show that for each fibrant object $\overline{Z}\in\SS_{/\frak P}^{+}$,
the functor $\Map_{S}^{\flat}\pr{-,\overline{Z}}:\pr{\SS_{/\frak P}^{+}}^{\op}\to\SS_{{\rm Joyal}}$
preserves weak equivalences. By Ken Brown's lemma \cite[Lemma 1.1.12]{Hovey},
it suffices to show that the functor $\Map_{S}^{\flat}\pr{-,\overline{Z}}$
carries trivial cofibrations of $\SS_{/\frak P}^{+}$ to categorical
equivalences. So take an arbitrary trivial cofibration $\overline{A}\to\overline{B}$
$\SS_{/\frak P}^{+}$. We claim that the map
\[
\Map_{S}^{\flat}\pr{\overline{B},\overline{Z}}\to\Map_{S}^{\flat}\pr{\overline{A},\overline{Z}}
\]
is a trivial fibration. Unwinding the definitions, we must show that,
for each $n\geq0$, the map
\[
\pr{\Delta^{n}}^{\flat}\times\overline{A}\cup\pr{\partial\Delta^{n}}^{\flat}\times\overline{B}\to\pr{\Delta^{n}}^{\flat}\times\overline{B}
\]
is a $\frak P$-equivalence. This follows from \cite[Remark B.2.5]{HA}.
\end{proof}
As a consequence of Corollary \ref{cor:P-equivalence}, we obtain
a stronger version of part (2) of Proposition \ref{prop:creative}:
\begin{prop}
\label{prop:creative_precise}Let $\frak P=\pr{M_{\cal D},T,\{p_{\alpha}\}_{\alpha\in A}}$
be a categorical pattern on an $\infty$-category $\cal D$, and let
$f:\overline{\cal X}\to\overline{\cal Y}$ be a morphism between fibrant
objects of $\SS_{/\frak P}^{+}$. Suppose that the set $M_{\cal D}$
contains every equivalence of $\cal D$. Then $f$ is a $\frak P$-equivalence
if and only if its underlying map of simplicial sets is a categorical
equivalence.
\end{prop}

\begin{proof}
Suppose first that $f$ is a weak equivalence. We must show that $f$
is a categorical equivalence. Since $\cal D$ is an $\infty$-category,
the underlying simplicial sets of $\frak P$-fibered object are $\infty$-categories.
Therefore, the forgetful functor $\frak P\-\Fib\to\Cat_{\infty}$
is well-defined. Since $f$ is an equivalence of $\frak P\-\Fib$,
it follows that its image in $\Cat_{\infty}$ is also an equivalence,
as claimed.

Conversely, suppose that $f$ is a categorical equivalence. According
to Corollary \ref{cor:P-equivalence}, it suffices to show that, for
every fibrant object $\overline{\cal Z}\in\SS_{/\frak P}^{+}$, the
functor 
\begin{equation}
f^{*}:\Map_{\cal D}^{\flat}\pr{\overline{\cal Y},\overline{\cal Z}}\to\Map_{\cal D}^{\flat}\pr{\overline{\cal X},\overline{\cal Z}}\label{eq:cateq-peq_flat}
\end{equation}
is a categorical equivalence. Since $\frak P$ is a creative categorical
pattern, part (1) of Proposition \ref{prop:creative} shows that the
map $\cal Z\to\cal D$ is a categorical fibration. Therefore, the
map
\begin{equation}
f^{*}:\Fun_{\cal D}\pr{\cal Y,\cal Z}\to\Fun_{\cal D}\pr{\cal X,\cal Z}\label{eq:cateq-peq}
\end{equation}
is a categorical equivalence. So the functor (\ref{eq:cateq-peq_flat})
is fully faithful. We complete the proof by showing that it is essentially
surjective.

Let $g\in\Map_{\cal D}^{\flat}\pr{\overline{\cal X},\overline{\cal Z}}$
be an arbitrary vertex. We must find a vertex $g'\in\Map_{\cal D}^{\flat}\pr{\overline{\cal Y},\overline{\cal Z}}$
such that $g'f$ is equivalent to $g$ as an object of $\Map_{\cal D}^{\flat}\pr{\overline{\cal X},\overline{\cal Z}}$.
Since the functor (\ref{eq:cateq-peq}) is essentially surjective,
there are a functor $g'\in\Fun_{\cal D}\pr{\cal Y,\cal Z}$ and an
equivalence $g\xrightarrow{\simeq}g'f$ in $\Fun_{\cal D}\pr{\cal X,\cal Z}$.
We claim that $g'$ belongs to $\Map_{\cal D}^{\flat}\pr{\overline{\cal Y},\overline{\cal Z}}$. 

Let $e:Y\to Y'$ be a marked edge of $\overline{\cal Y}$. We must
show that the edge $g'\pr e$ is marked in $\overline{\cal Z}$. Let
$p:\cal X\to\cal D$ and $q:\cal Y\to\cal D$ denote the projections.
Since $p$ and $q$ are categorical fibrations, the functor
\[
\Fun\pr{\Delta^{1},\cal X}\times_{\Fun\pr{\Delta^{1},\cal D}}\{q\pr e\}\to\Fun\pr{\Delta^{1},\cal Y}\times_{\Fun\pr{\Delta^{1},\cal D}}\{q\pr e\}
\]
is a categorical equivalence. Therefore, we can find a morphism $\widetilde{e}:X\to X'$
of $\cal X$ and a diagram $\Delta^{1}\times\Delta^{1}\to\cal Y$
depicted as 
\[\begin{tikzcd}
	{f(X)} & {f(X')} \\
	Y & {Y',}
	\arrow["{f(\widetilde{e})}", from=1-1, to=1-2]
	\arrow["\simeq", from=1-2, to=2-2]
	\arrow["e"', from=2-1, to=2-2]
	\arrow["\simeq"', from=1-1, to=2-1]
\end{tikzcd}\]such that the composite $\Delta^{1}\times\Delta^{1}\to\cal Y\to\cal D$
is equal to the map $\Delta^{1}\times\Delta^{1}\xrightarrow{\opn{pr}_{1}}\Delta^{1}\xrightarrow{q\pr e}\cal D$.
The morphism $f\pr{\widetilde{e}}$ is locally $q$-cocartesian because
$e$ is locally $q$-cocartesian. Since $f$ is a categorical equivalence,
this implies that $\widetilde{e}$ is locally $q$-cocartesian. Hence
$\widetilde{e}$ is marked in $\overline{\cal X}$. Therefore, the
edge $g\pr{\widetilde{e}}$ is marked. Now since $M_{\cal D}$ contains
every equivalence of $\cal D$, all equivalence of $\cal Z$ are marked.
Thus, applying Proposition \ref{prop:A_0} twice, we deduce that $g'f\pr{\widetilde{e}}$
is marked. Applying Proposition \ref{prop:A_0} twice again, we deduce
that $g'\pr e$ is marked. The proof is now complete.
\end{proof}
\begin{proof}
[Proof of Proposition \ref{prop:creative}]As we stated above, assertion
(1) is proved in \cite[Proposition B.2.7]{HA}. Part (2) is a consequence
of Proposition \ref{prop:creative_precise}.
\end{proof}

\subsection{\label{subsec:cat_inv}Categorical Invariance of Categorical Patterns}

Let $f:\cal C\to\cal D$ be a functor of $\infty$-categories. We
say that $f$ is \textbf{compatible} with categorical patterns $\frak P_{\cal C}$
and $\frak P_{\cal D}$ on $\cal C$ and $\cal D$ if $\frak P_{\cal C}\subset f^{*}\frak P_{\cal D}$.
In this case, the adjunction
\begin{equation}
f_{!}:\SS_{/\frak P_{\cal C}}^{+}\adj\SS_{/\frak P_{\cal D}}^{+}:f^{*}\label{eq:compatible}
\end{equation}
is a Quillen adjunction \cite[Proposition B.2.9]{HA}. In this subsection,
we give a sufficient condition for this Quillen adjunction to be a
Quillen equivalence.

To state the main result of this section, we introduce a bit of terminology.
\begin{defn}
Let $f:\cal C\to\cal D$ be a categorical equivalence of $\infty$-categories,
and let $\frak P_{\cal C}=\pr{M_{\cal C},\{q_{i}\}_{i\in I}}$ and
$\frak P_{\cal D}=\pr{M_{\cal D},\{p_{\alpha}\}_{\alpha\in A}}$ be
commutative categorical patterns on $\cal C$ and $\cal D$. We say
that $f$ is \textbf{strongly compatible} with $\frak P_{\cal C}$
and $\frak P_{\cal D}$ if it satisfies the following conditions:
\begin{enumerate}
\item The map $f$ is compatible with $\frak P_{\cal C}$ and $\frak P_{\cal D}$.
\item Every element of $M_{\cal D}$ is equivalent to an element of $f\pr{M_{\cal C}}$
as an object of $\Fun\pr{\Delta^{1},\cal D}$.
\item For each $\alpha\in A$, there is some $i\in I$ such that the diagram
$p_{\alpha}$ is naturally equivalent to $fq_{i}$.
\end{enumerate}
\end{defn}

Here is the main result of this subsection.
\begin{prop}
\label{prop:cat_inv}Let $f:\cal C\to\cal D$ be a categorical equivalence
of $\infty$-categories, and let $\frak P_{\cal C}=\pr{M_{\cal C},\{q_{i}\}_{i\in I}}$
and $\frak P_{\cal D}=\pr{M_{\cal D},\{p_{\alpha}\}_{\alpha\in A}}$
be commutative categorical patterns on $\cal C$ and $\cal D$. If
$f$ is strongly compatible with $\frak P_{\cal C}$ and $\frak P_{\cal D}$,
then the adjunction
\[
f_{!}:\SS_{/\frak P_{\cal C}}^{+}\adj\SS_{/\frak P_{\cal D}}^{+}:f^{*}
\]
is a Quillen equivalence.
\end{prop}

The proof of Proposition \ref{prop:cat_inv} relies on a lemma.
\begin{lem}
\label{lem:cat_inv}Suppose we are given a commutative diagram 
\[\begin{tikzcd}
	{\mathcal{X}} & {\mathcal{Y}} \\
	{\mathcal{C}} & {\mathcal{D}}
	\arrow["g", from=1-1, to=1-2]
	\arrow["q", from=1-2, to=2-2]
	\arrow["f"', from=2-1, to=2-2]
	\arrow["p"', from=1-1, to=2-1]
	\arrow["\simeq", from=2-1, to=2-2]
	\arrow["\simeq"', from=1-1, to=1-2]
\end{tikzcd}\]of $\infty$-categories, where $f$ and $g$ are categorical equivalences
and $p$ and $q$ are categorical fibrations. Let $\frak P_{\cal C}=\pr{M_{\cal C},\{q_{i}\}_{i\in I}}$
and $\frak P_{\cal D}=\pr{M_{\cal D},\{p_{\alpha}\}_{\alpha\in A}}$
be commutative categorical patterns on $\cal C$ and $\cal D$, and
suppose that $f$ is strongly compatible with $\frak P_{\cal C}$
and $\frak P_{\cal D}$. The following conditions are equivalent:
\begin{enumerate}
\item The map $p$ is $\frak P_{\cal C}$-fibered.
\item The map $q$ is $\frak P_{\cal D}$-fibered.
\end{enumerate}
\end{lem}

\begin{proof}
We first prove that (2)$\implies$(1), assuming (1)$\implies$(2).
Suppose that $q$ is $\frak P_{\cal D}$-fibered. Since the functor
$\SS_{/\frak P_{\cal D}}^{+}\to\SS_{/\frak P_{\cal C}}^{+}$ is right
Quillen, the map $q':\cal Y\times_{\cal D}\cal C\to\cal C$ is $\frak P_{\cal C}$-fibered.
Now since $p$ and $q'$ are both categorical fibrations, the functor
$\cal X\to\cal Y\times_{\cal D}\cal C$ has an inverse categorical
equivalence which commutes with the projection to $\cal C$. Applying
the implication (1)$\implies$(2) to the inverse equivalence, we deduce
that $p$ is $\frak P_{\cal C}$-fibered.

Next, we prove that (1)$\implies$(2). Suppose that $p$ is $\frak P_{\cal C}$-fibered.
We must show that $q$ is $\frak P_{\cal D}$-fibered. Let $\overline{\cal Y}$
denote the marked simplicial set obtained from $\cal Y$ by marking
the $q$-cocartesian edges lying over the edges in $M_{\cal D}$.
According to Remark \ref{rem:mlfib}, we must prove the following:
\begin{itemize}
\item [(A)]For each object $Y\in\cal Y$ and each morphism $\alpha:q\pr Y\to D'$
in $M_{\cal D}$, there is a $q$-cocartesian morphism $Y\to Y'$
which lifts $\alpha$.
\item [(B-1)]Every morphism $u:\pr{K_{\alpha}}^{\sharp}\to\overline{\cal Y}$
in $\SS_{/\frak P_{\cal D}}^{+}$ can be extended to $\pr{K_{\alpha}^{\lcone}}^{\sharp}$.
\item [(B-2)]Any morphism $h:\pr{K_{\alpha}^{\lcone}}^{\sharp}\to\overline{\cal Y}$
in $\SS_{/\frak P_{\cal D}}^{+}$ is a $q$-limit diagram.
\end{itemize}
We begin with (A). Since $f$ and $g$ are categorical equivalences,
we can find objects $X\in\cal X$ and $C\in\cal C$, equivalences
$\beta:g\pr X\xrightarrow{\simeq}Y$ and $\gamma:D'\xrightarrow{\simeq}f\pr C$,
and a $3$-simplex $\sigma:\Delta^{3}\to\cal D$ which we depict as
\[\begin{tikzcd}
	&& {D'} \\
	{q(g(X))} & {q(Y)} \\
	&& {f(C).}
	\arrow[curve={height=-6pt}, from=2-1, to=1-3]
	\arrow["\simeq"', from=2-1, to=2-2]
	\arrow["\simeq"', from=1-3, to=3-3]
	\arrow["\alpha", from=2-2, to=1-3]
	\arrow["\delta"', curve={height=6pt}, from=2-1, to=3-3]
	\arrow["{q(\beta)}", from=2-1, to=2-2]
	\arrow["\gamma", from=1-3, to=3-3]
	\arrow[from=2-2, to=3-3]
\end{tikzcd}\]It suffices to show that the morphism $\delta$ admits a $q$-cocartesian
lift with source $g\pr X$. Since $f$ is a categorical equivalence,
there is a morphism $\delta:p\pr X\to C$ such that $f\pr{\delta'}$
and $\delta$ are homotopic. Replacing $\sigma$ if necessary, we
may assume that $\delta=f\pr{\delta'}$. Since $f$ is strongly compatible
with $\frak P_{\cal C}$ and $\frak P_{\cal D}$, the morphism $\delta$
is equivalent to an edge in $M_{\cal C}$ as an object of $\Fun\pr{\Delta^{1},\cal C}$.
In particular, there is a $p$-cocartesian morphism $\widetilde{\delta'}:X\to X'$
over $\delta'$. Then $g\pr{\widetilde{\delta}'}$ is a $q$-cocartesian
morphism lying over $\delta$ with source $g\pr X$, as desired. 

Next, we verify condition (B-1). Since $p$ is a categorical fibration,
using \cite[Proposition A.2.3.1]{HTT}, we can find a commutative
diagram
\[\begin{tikzcd}
	{K_\alpha} & {\mathcal{Y}} & {\mathcal{X}} \\
	{K_\alpha^\triangleleft} & {\mathcal{D}} & {\mathcal{C}}
	\arrow["{f'}"', from=2-2, to=2-3]
	\arrow["p", from=1-3, to=2-3]
	\arrow["{g'}", from=1-2, to=1-3]
	\arrow["q", from=1-2, to=2-2]
	\arrow["{p_\alpha}"', from=2-1, to=2-2]
	\arrow["u", from=1-1, to=1-2]
	\arrow[from=1-1, to=2-1]
\end{tikzcd}\]such that the maps $f'$ and $g'$ are inverse equivalences of $f$
and $g$. By hypothesis, the composite $f'p_{\alpha}$ is naturally
equivalent to a diagram in $\{q_{i}\}_{i\in I}$. Therefore, by \cite[Proposition A.2.3.1]{HTT},
we can find a diagram $h':K_{\alpha}^{\lcone}\to\cal X$ which makes
the diagram 
\[\begin{tikzcd}
	{K_\alpha} && {\mathcal{X}} \\
	{K_\alpha^\triangleleft} && {\mathcal{C}}
	\arrow["p", from=1-3, to=2-3]
	\arrow[from=1-1, to=2-1]
	\arrow["{h'}"{description}, from=2-1, to=1-3]
	\arrow["{f'p_\alpha}"', from=2-1, to=2-3]
	\arrow["{g'u}", from=1-1, to=1-3]
\end{tikzcd}\]commutative, such that $h'$ carries each edge of $K_{\alpha}^{\lcone}$
to a $p$-cocartesian morphism. Using \cite[Proposition A.2.3.1]{HTT}
once again, we can find a diagram $h:K_{\alpha}^{\lcone}\to\cal Y$
which makes the diagram 
\[\begin{tikzcd}
	{K_\alpha} & {\mathcal{Y}} \\
	{K_\alpha^\triangleleft} & {\mathcal{D}}
	\arrow["q", from=1-2, to=2-2]
	\arrow["{p_\alpha}"', from=2-1, to=2-2]
	\arrow["u", from=1-1, to=1-2]
	\arrow[from=1-1, to=2-1]
	\arrow["h"{description}, from=2-1, to=1-2]
\end{tikzcd}\]commutative, and such that the composite $g'h$ is naturally equivalent
to $h'$. Since the composite $g'h$ carries each edge of $K_{\alpha}^{\lcone}$
to a $p$-cocartesian morphism, the map $h$ carries each edge of
$K_{\alpha}^{\lcone}$ to a $q$-cocartesian morphism. Thus $h$ gives
the desired extension.

Finally, we verify condition (B-2). Let $f',g'$ be as in the previous
paragraph. By hypothesis, the composite $f'p_{\alpha}$ is naturally
equivalent to a diagram in $\{q_{i}\}_{i\in I}$. It follows that
the composite $g'h$ is a $p$-limit diagram. Since $f'$ and $g'$
are categorical equivalences, we deduce that $h$ is a $q$-limit
diagram, and we are done.
\end{proof}
\begin{proof}
[Proof of Proposition \ref{prop:cat_inv}]We will show that the total
left derived functor $\bb Lf_{!}$ is fully faithful and that the
total right derived functor $\bb Rf^{*}$ is conservative. 

The conservativity of $\bb Rf^{*}$ follows from Proposition \ref{prop:creative}.
To show that $\bb L\pi_{!}$ is fully faithful, let $\pr{p:\overline{\cal X}\to\overline{\cal C}}\in\SS_{/\frak P_{\cal C}}^{+}$
be a fibrant object. Factor the map $fp:\cal X\to\cal D$ as 
\[
\cal X\xrightarrow{g}\cal Y\xrightarrow{q}\cal D,
\]
where $g$ is a categorical equivalence and $q$ is a categorical
fibration. Let $\overline{\cal Y}$ denote the marked simplicial set
obtained from $\cal Y$ by marking the $q$-cocartesian morphisms
lying over the morphisms in $M_{\cal D}$. According Lemma \ref{lem:cat_inv},
the object $\overline{\cal Y}\in\SS_{/\frak P}^{+}$ is fibrant. We
will show that the induced map $g:\overline{\cal X}\to\overline{\cal Y}$
is a $\frak P$-equivalence. It will then follow from Proposition
\ref{prop:creative} that the derived unit is an isomorphism, so that
$\bb Lf_{!}$ is fully faithful.

Let $\overline{\cal Z}\in\SS_{/\frak P_{\cal C}}^{+}$ be a fibrant
object. We must show that the map 
\[
\Map_{\cal D}^{\sharp}\pr{\overline{\cal Y},\overline{\cal Z}}\to\Map_{\cal D}^{\sharp}\pr{\overline{\cal X},\overline{\cal Z}}
\]
is a homotopy equivalence. By Lemma \ref{lem:Map_S^=00005Csharp_core},
it suffices to show that the functor
\[
\theta:\Map_{\cal D}^{\flat}\pr{\overline{\cal Y},\overline{\cal Z}}\to\Map_{\cal D}^{\flat}\pr{\overline{\cal X},\overline{\cal Z}}
\]
is a categorical equivalence. According to Proposition \ref{prop:creative},
the functor $r:\cal Z\to\cal D$ is a categorical fibration. Therefore,
the functor
\[
\theta':\Fun_{\cal D}\pr{\cal Y,\cal Z}\to\Fun_{\cal D}\pr{\cal X,\cal Z}
\]
is a categorical equivalence. It follows that $\theta$ is fully faithful,
so it will suffice to show that $\theta$ is essentially surjective.
Let $h\in\Map_{\cal D}^{\flat}\pr{\overline{\cal X},\overline{\cal Z}}$
be an arbitrary object. Since $\theta'$ is essentially surjective,
we can find a functor $h'\in\Fun_{\cal D}\pr{\cal Y,\cal Z}$ and
an equivalence $h'g\simeq h$ in $\Fun_{\cal D}\pr{\cal X,\cal Z}$.
To complete the proof, it suffice to show that $h'$ belongs to $\Map_{\cal D}^{\flat}\pr{\overline{\cal Y},\overline{\cal Z}}$.
Let $\beta$ be a marked edge of $\overline{\cal Y}$. We wish to
show that $h'\pr{\beta}$ is $r$-cocartesian. By hypothesis, there
is a morphism $\gamma$ in $M_{\cal C}$ such that $f\pr{\gamma}$
is equivalent to $q\pr{\beta}$ in $\Fun\pr{\Delta^{1},\cal Y}$.
Since $q$ is a categorical fibration, this means that $\beta$ is
equivalent to a morphism lying over $f\pr{\gamma}$ in $\Fun\pr{\Delta^{1},\cal Y}$.
Since $p$ and $q$ are categorical fibrations and $f$ and $g$ are
categorical equivalences, the functor
\[
\Fun\pr{\Delta^{1},\cal X}\times_{\Fun\pr{\Delta^{1},\cal C}}\{\gamma\}\to\Fun\pr{\Delta^{1},\cal Y}\times_{\Fun\pr{\Delta^{1},\cal D}}\{f\pr{\gamma}\}
\]
is a categorical equivalence. Thus $\beta$ is equivalent to a morphism
$f\pr{\alpha}$, where $\alpha$ is a morphism of $\cal X$ lying
over $\gamma$. Now $\alpha$ is necessarily $p$-cocartesian, so
$h\pr{\alpha}$ is $r$-cocartesian. Since $h\pr{\beta}$, $h'g\pr{\alpha}$,
and $h\pr{\alpha}$, are all equivalent objects of $\Fun\pr{\Delta^{1},\cal Z}$,
we deduce that $h\pr{\beta}$ is also $r$-cocartesian. The proof
is now complete.
\end{proof}

\section{\label{sec:P-bundles}\texorpdfstring{$\mathfrak{P}$}{P}-bundles}

In this section, we will introduce the notion of $\frak P$-bundles,
where $\frak P$ is a commutative categorical pattern on another $\infty$-category
$\cal D$ (Definition \ref{def:P-bundle}). Roughly speaking, $\frak P$-bundles
over a simplicial set $S$ are to functors $S\to\frak P\-\Bund$ what
cocartesian fibrations are to functors $S\to\Cat_{\infty}$. We will
then construct a model structure of $\frak P$-bundles over $\cal C$
(Proposition \ref{prop:recognition_of_fibrant_objects}), and establish
a fiberwise criterion for weak equivalences of this model structure
(Proposition \ref{prop:fiberwise_P_equiv}).
\begin{defn}
\label{def:P-bundle}Let $\frak P$ be a commutative categorical pattern
on an $\infty$-category $\cal D$. Let $S$ be a simplicial set.
A \textbf{$\frak P$-bundle }(\textbf{over }$S$) is a commutative
diagram 
\[\begin{tikzcd}
	X && {S\times \mathcal{D}} \\
	& S
	\arrow["p", from=1-1, to=1-3]
	\arrow["{\operatorname{pr}}", from=1-3, to=2-2]
	\arrow["q"', from=1-1, to=2-2]
\end{tikzcd}\]of simplicial sets which satisfies the following conditions:

\begin{enumerate}[label=(\alph*)]

\item The map $q:X\to S$ is a cocartesian fibration.

\item The map $p$ lifts to a fibration of fibrant objects of $\SS_{/S}^{+}$
with respect to the cocartesian model structure.

\item For each vertex $v\in S$, the map $X_{v}=X\times_{S}\{v\}\to\cal D$
is $\frak P$-fibered.

\item For each edge $f:v\to v'$ in $S$, the induced functor $f_{!}:X_{v}\to X_{v'}$
is a morphism of $\frak P$-fibered objects.

\end{enumerate}

Since the map $q$ can be recovered from $p$, we will often say that
the map $p$ is a $\frak P$-bundle over $\cal C$. Given a $\frak P$-bundle
$p:X\to S\times\cal D$, we will write $X_{\natural}$ for the marked
simplicial set obtained from $X$ by marking the $p$-cocartesian
edges lying over the morphisms in $M_{\cal D}$.
\end{defn}

\begin{rem}
\label{rem:P-bund_b}Recall that, given a simplicial set $K$ and
fibrant objects $X^{\natural},Y^{\natural}\in\SS_{/K}^{+}$ of the
cocartesian model structure, a map $X^{\natural}\to Y^{\natural}$
is a fibration if and only if it is a marked left fibration \cite[Proposition 1.1.7]{JacoMT}.
Therefore, condition (b) of Definition \ref{def:P-bundle} is equivalent
to the condition that the map $p:X\to S\times\cal D$ satisfy the
following conditions:
\begin{itemize}
\item The map $p$ is an inner fibration.
\item For each vertex $x\in X$ with image $\pr{v,D}\in S\times\cal D$,
and for each edge $\pr{f,g}:\pr{v,D}\to\pr{v',D'}$, if $g$ is an
equivalence of $\cal D$, there is a $p$-cocartesian edge $e:x\to x'$
such that $p\pr e=\pr{f,g}$.
\end{itemize}
\end{rem}

\begin{rem}
\label{rem:P-bundle}In condition (d) of Definition \ref{def:P-bundle},
we tacitly assumed that the functor $f_{!}$ is obtained from a cocartesian
natural transformation $X_{v}\times\Delta^{1}\to\cal X$ fitting into
the commutative diagram
\[\begin{tikzcd}
	{\{0\}\times X_v} && X \\
	{\Delta^1\times X_v} & {\Delta^1\times  \mathcal{D}} & {S\times \mathcal{D}}
	\arrow[from=1-1, to=1-3]
	\arrow["p", from=1-3, to=2-3]
	\arrow[from=2-1, to=2-2]
	\arrow["{f\times\operatorname{id}}"', from=2-2, to=2-3]
	\arrow[dashed, from=2-1, to=1-3]
	\arrow[from=1-1, to=2-1]
\end{tikzcd}\]Such a cocartesian natural transformation exists because for each
object $x\in X_{v}$, there is a $p$-cocartesian morphism $x\to x'$
lying over $f$ and an identity morphism of $\cal D$ (by conditions
(a) and (b)). The functor $f_{!}$ is well-defined up to natural equivalence
over $\cal D$.
\end{rem}

We now construct the model structure of $\frak P$-bundles over a
fixed base.
\begin{defn}
Let $\frak P=\pr{M_{\cal D},T,\{K_{\alpha}^{\lcone}\to\cal D\}_{\alpha\in A}}$
be a categorical pattern on an $\infty$-category $\cal D$. Let $\overline{S}=\pr{S,M_{S}}$
be a marked simplicial set. We let $\overline{S}\times\frak P$ denote
the categorical pattern 
\[
\pr{M_{S}\times M_{\cal D},S_{2}\times T,\{\{v\}\times K_{\alpha}^{\lcone}\to\cal C\times\cal D\}_{v\in S_{0},\alpha\in A}}.
\]
We will write $S\times\frak P=S^{\sharp}\times\frak P$.
\end{defn}

The goal of this section is to prove the following:
\begin{prop}
\label{prop:recognition_of_fibrant_objects}Let $\frak P=\pr{M_{\cal D},\{p_{\alpha}:K_{\alpha}^{\lcone}\to\cal D\}_{\alpha\in A}}$
be a commutative categorical pattern on an $\infty$-category $\cal D$
and let $S$ be a simplicial set. An object $\overline{X}=\pr{X,M}\in\SS_{/\cal C\times\frak P}^{+}$
is $S\times\frak P$-fibered if and only the map $X\to S\times\cal D$
is a $\frak P$-bundle over $S$ and $\overline{X}=X_{\natural}$
(Definition \ref{def:P-bundle}).
\end{prop}

In view of Theorem \ref{thm:B.0.20} and Proposition \ref{prop:creative},
the above Proposition immediately implies the existence of a model
category of $\frak P$-bundles:
\begin{cor}
\label{cor:model_structure}Let $\frak P$ be a commutative categorical
pattern on an $\infty$-category $\cal D$ and let $S$ be a simplicial
set. There is a combinatorial simplicial model structure on $\SS_{/S\times\frak P}^{+}$
which has the following properties:
\begin{itemize}
\item The simplicial enrichment is given by $\Map_{S\times\cal D}^{\sharp}\pr{-,-}$.
\item Cofibrations are the monomorphisms.
\item Fibrant objects are the objects of the form $X_{\natural}$, where
$X$ is a $\frak P$-bundle over $S$.
\end{itemize}
If further $S$ is an $\infty$-category, then this model structure
enjoys the following additional property:
\begin{itemize}
\item A morphism $\cal X_{\natural}\to\cal Y_{\natural}$ between fibrant
objects is a fibration (resp. weak equivalence) if and only if the
underlying map $\cal X\to\cal Y$ of simplicial sets is a categorical
fibration (resp. categorical equivalence).
\end{itemize}
\end{cor}

With Corollary \ref{cor:model_structure} in mind, we make the following
definition:
\begin{defn}
\label{def:eq_P-bund}Let $\frak P$ be a commutative categorical
pattern on an $\infty$-category $\cal D$. Given a simplicial set
$S$, we will write $\frak P\-\Bund\pr S$ for the homotopy coherent
nerve of the full simplicial subcategory of $\SS_{/S\times\frak P}^{+}$
spanned by the fibrant--cofibrant objects. If $X\to S\times\cal D$
and $Y\to S\times\cal D$ are $\frak P$-bundles, then a map $X\to Y$
of simplicial sets over $S\times\cal D$ is called a \textbf{morphism}
(resp. \textbf{equivalence}) \textbf{of $\frak P$-bundles }(\textbf{over
}$S$)\textbf{ }if it determines a morphism (resp. equivalence) in
$\frak P\-\Bund\pr S$.
\end{defn}

We now turn to the proof of Proposition \ref{prop:recognition_of_fibrant_objects}.
We need a few preliminaries.

\begin{lem}
\label{lem:2.2.2.7}Let $q:X\to S$ be an inner fibration of simplicial
sets and let $i:A\to B$ be a monomorphism of simplicial sets. Suppose
we are given a map $\overline{f}:B\times K^{\rcone}\to X$ of simplicial
sets such that, for each vertex $b\in B$, the diagram $\overline{f}\vert\{b\}\times K^{\rcone}$
is a $q$-colimit diagram. Set $\overline{f}_{A}=\overline{f}\vert A\times K^{\rcone}\cup B\times K$.
The map 
\[
\theta_{i,\overline{f}}:X_{\overline{f}/}\to X_{\overline{f}_{A}/}\times_{S_{q\overline{f}_{A}/}}S_{q\overline{f}/}
\]
is a trivial fibration.
\end{lem}

\begin{proof}
Let $\scr M$ denote the class of monomorphisms $A\to B$ for which
the lemma holds. We wish to show that $\scr M$ contains all monomorphisms.
Since $\scr M$ is closed under pushouts and transfinite compositions,
it will suffice to show that it contains the inclusion $\partial\Delta^{n}\subset\Delta^{n}$
for every $n\geq0$.

Before we proceed, we remark that $\scr M$ has the following right
cancellation property: Given a sequence of monomorphisms $A\xrightarrow{f}B\xrightarrow{g}C$
of simplicial sets with $f\in\scr M$, we have $gf\in\scr M$ if and
only if $g\in\scr M$. This follows from the observation that the
map $\theta_{i,\overline{f}}$ is a left fibration, so it is a trivial
fibration if and only if it is a weak categorical equivalence \cite[Proposition 3.3.1.7]{HTT}.

We now show that the inclusions $\{\partial\Delta^{n}\subset\Delta^{n}\}_{n\geq0}$
belong to $\scr M$. We prove this by induction on $n$. When $n=0$,
the claim follows from the definition of relative colimits. For the
inductive step, suppose that we have proved the assertion up to $n-1$.
Since $\scr M$ is closed under pushouts and compositions, the inductive
hypothesis ensures that the inclusion $\emptyset\subset\partial\Delta^{n}$
is an element of $\scr M$. Therefore, by the right cancellation property
of $\scr M$, we are reduced to showing that the inclusion $\emptyset\subset\Delta^{n}$
belongs to $\scr M$. Since $\scr M$ is closed under composition
and the inclusion $\emptyset\subset\{n\}$ belongs to $\scr M$, it
suffices to show that the inclusion $\{n\}\subset\Delta^{n}$ is an
element of $\scr M$. But this is obvious, because $\scr M$ contains
every right anodyne extension.
\end{proof}
The following lemma shows that relative colimits can be formed and
detected fiberwise. (Compare \cite[Porposition 4.3.1.10]{HTT}.)
\begin{lem}
\label{lem:fiberwise_colimit}Let 
\[\begin{tikzcd}
	X && Y \\
	& Z
	\arrow["f", from=1-1, to=1-3]
	\arrow["q", from=1-3, to=2-2]
	\arrow["p"', from=1-1, to=2-2]
\end{tikzcd}\]be a commutative diagram of simplicial sets, let $K$ be a simplicial
set, let $z\in Z$ be a vertex, and let $\overline{\sigma}:K^{\rcone}\to X_{z}$
be a diagram. Assume that the following conditions are satisfied:
\begin{itemize}
\item [(a)]The functors $p$ and $q$ are cocartesian fibrations.
\item [(b)]The functor $f$ induces a marked left fibration $X^{\natural}\to Y^{\natural}$,
where $X^{\natural}$ and $Y^{\natural}$ denote the marked simplicial
set obtained from $X$ and $Y$ by marking the cocartesian edges over
$Z$.
\end{itemize}
Then the following conditions are equivalent:
\begin{enumerate}
\item The diagram $\overline{\sigma}$ is an $f$-colimit diagram.
\item For each edge $\alpha:z\to z'$ in $Z$, the induced diagram $\alpha_{!}\overline{\sigma}$
is an $f_{E'}$-colimit diagram, where $f_{E'}:X_{z'}\to Y_{z'}$
is the restriction of $f$.
\end{enumerate}
\end{lem}

\begin{proof}
We will write $\delta_{K}\pr z:K\to Z$ for the constant map at the
vertex $z$. If $\alpha:z\to z'$ is an edge of $Z$, we will write
$\alpha_{K}:K^{\rcone}\to Z$ for the composite
\[
K^{\rcone}\to\pr{\Delta^{0}}^{\rcone}\xrightarrow{\alpha}Z.
\]
We will also write $\sigma=\overline{\sigma}\vert K$.

Condition (1) is equivalent to the requirement that, for any pair
of vertices $z'\in Z$ and $x'\in X_{z'}$, the top square of the
diagram 
\[\begin{tikzcd}
	{X_{\overline{\sigma}/}\times _{X}\{x'\}} & {X_{\sigma/}\times _{X}\{x'\}} \\
	{Y_{f\overline{\sigma}/}\times _{Y}\{f(x')\}} & {Y_{f\sigma/}\times _{Y}\{f(x')\}} \\
	{Z_{\delta_{K^\triangleright}(z)/}\times _{Z}\{z'\}} & {Z_{\delta_K(z)/}\times _{Z}\{z'\}}
	\arrow[from=1-1, to=1-2]
	\arrow[from=1-1, to=2-1]
	\arrow[from=1-2, to=2-2]
	\arrow[from=2-1, to=2-2]
	\arrow[from=2-1, to=3-1]
	\arrow[from=2-2, to=3-2]
	\arrow[from=3-1, to=3-2]
\end{tikzcd}\]of Kan complexes be homotopy cartesian. Note that the vertical arrows
of the above diagram are all Kan fibrations; for instance, the map
$X_{\sigma/}\times_{X}\{x'\}\to Y_{f\sigma/}\times_{Y}\{f(x')\}$
is a pullback of the left fibration $X_{\sigma/}\to Y_{f\sigma/}\times_{Y}X$,
so it is a left fibration whose codomain is a Kan complex and hence
is a Kan fibration. Moreover, since the inclusion $\{\infty\}\subset K^{\rcone}$
is final, the projection $K^{\rcone}\to\{\infty\}$ induces a covariant
equivalence $Z_{z/}\to Z_{\delta_{K^{\rcone}}\pr z/}$ over $Z$.
Thus, every vertex of $Z_{\delta_{K^{\rcone}}\pr z/}\times_{Z}\{z'\}$
belongs to the same components as a vertex of the form $\alpha_{K^{\rcone}}$,
where $\alpha:z\to z'$ is an edge of $Z$. Therefore, we can reformulate
condition (1) as follows:
\begin{itemize}
\item [(1$'$)]For each morphism $\alpha:z\to z'$ of $Z$ and for each
$x'\in X_{z'}$, the diagram 
\begin{equation}\label{d:fiberwise_colimit}
\begin{tikzcd}
	{(X_{\overline{\sigma}/}\times _{X}\{x'\})_{\alpha_{K^\triangleright}}} & {(X_{\sigma/}\times _{X}\{x'\})_{\alpha_K}} \\
	{(Y_{f\overline{\sigma}/}\times _{Y}\{f(x')\})_{\alpha_{K^\triangleright}}} & {(Y_{f\sigma/}\times _{Y}\{f(x')\})_{\alpha_K}}
	\arrow[from=1-1, to=1-2]
	\arrow[from=1-1, to=2-1]
	\arrow[from=1-2, to=2-2]
	\arrow[from=2-1, to=2-2]
\end{tikzcd}
\end{equation}of Kan complexes is homotopy cartesian, where the subscripts on the
left and right indicate fibers over the vertices $\alpha_{K^{\rcone}}$
and $\alpha_{K}$ of $Z_{\delta_{K^{\rcone}}\pr z/}\times_{Z}\{z'\}$
and $Z_{\delta_{K}\pr z/}\times_{Z}\{z'\}$, respectively.
\end{itemize}

Now let $\alpha:z\to z'$ be an edge of $Z$, and let $x'\in X_{z'}$
be a vertex. We will show that the Kan complex $\pr{X_{\sigma/}\times_{X}\{x'\}}_{\alpha_{K}}$
is homotopy equivalent to the Kan complex $\pr{X_{z'}}_{\alpha_{!}\sigma/}\times_{X_{z'}}\{x'\}$.
We will write $\beta:\pr{K\times\Delta^{1}}^{\rcone}\to Z$ for the
composite
\[
\pr{K\times\Delta^{1}}^{\rcone}\to\pr{\Delta^{1}}^{\rcone}\cong\Delta^{2}\xrightarrow{u}\Delta^{1}\xrightarrow{\alpha}Z,
\]
where the map $u$ is induced by the surjective poset map $[2]\to[1]$
which hits $1\in[1]$ twice. Choose a $q$-cocartesian natural transformation
$\psi:Y_{z}\times\Delta^{1}\to Y$ covering $\alpha$, and choose
a $p$-cocartesian natural transformation $\phi:X_{z}\times\Delta^{1}\to X$
rendering the diagram 
\[\begin{tikzcd}
	{\{0\}\times X_z} && X \\
	{\Delta^1\times X_z} & {\Delta^1\times Y_z} & Y
	\arrow[from=1-1, to=1-3]
	\arrow["\psi"', from=2-2, to=2-3]
	\arrow["f", from=1-3, to=2-3]
	\arrow[from=2-1, to=2-2]
	\arrow["\phi"{description}, dashed, from=2-1, to=1-3]
	\arrow[from=1-1, to=2-1]
\end{tikzcd}\]commutative. We consider the following commutative diagram:
\[\begin{tikzcd}
	{X_{\sigma/}\times _{Z_{\delta_K(z)/}}Z_{\alpha\pi /}} & {X_{\phi(\sigma\times \operatorname{id}_{\Delta^1})/}} & {X_{\alpha_!\sigma/}\times _{Z_{\delta_K(z')/}}Z_{\alpha\pi /}} \\
	& {X\times _{Z}Z_{\alpha\pi/}.}
	\arrow[from=1-2, to=1-3]
	\arrow[from=1-2, to=1-1]
	\arrow[from=1-1, to=2-2]
	\arrow[from=1-3, to=2-2]
	\arrow[from=1-2, to=2-2]
\end{tikzcd}\]Here $\pi:K\times\Delta^{1}\to\Delta^{1}$ denotes the projection.
Since the inclusion $\{1\}\subset\Delta^{1}$ is right anodyne, the
right horizontal arrow is a trivial fibration. According to Lemma
\ref{lem:2.2.2.7}, the left horizontal arrow is also a trivial fibration.
Thus, by passing to the fiber over $\pr{x',\beta}\in X\times_{Z}Z_{\alpha\pi/}$,
we obtain trivial fibrations
\[
\pr{X_{\sigma/}\times_{X}\{x'\}}_{\alpha_{K}}\stackrel{\simeq}{\twoheadleftarrow}X_{\phi\pr{\sigma\times\id_{\Delta^{1}}}/}\times_{X\times_{Z}Z_{\alpha\pi/}}\{\pr{x',\beta}\}\xepi{\simeq}{}\pr{X_{z'}}_{\alpha_{!}\sigma/}\times_{X_{z'}}\{x'\}.
\]
This gives the desired identification of the homotopy type of $\pr{X_{\sigma/}\times_{X}\{x'\}}_{\alpha_{K}}$. 

Carrying out similar analyses for the vertices of the diagram (\ref{d:fiberwise_colimit}),
we obtain the following reformulation of (1$'$):
\begin{itemize}
\item [(1$''$)]For each edge $\alpha:z\to z'$ of $Z$ and each vertex
$x'\in X_{z'}$, the diagram 
\[\begin{tikzcd}
	{(X_{z'})_{\alpha_!\overline\sigma/}\times _{X_{z'}}\{x'\}} & {(X_{z'})_{\alpha_!\sigma/}\times _{X_{z'}}\{x'\}} \\
	{(Y_{z'})_{\alpha_!f\overline\sigma/}\times _{Y_{z'}}\{f(x')\}} & {(Y_{z'})_{\alpha_!f\sigma/}\times _{Y_{z'}}\{f(x')\}}
	\arrow[from=1-1, to=1-2]
	\arrow[from=1-2, to=2-2]
	\arrow[from=2-1, to=2-2]
	\arrow[from=1-1, to=2-1]
\end{tikzcd}\]of Kan complexes is homotopy cartesian.
\end{itemize}
Condition (1$''$) is equivalent to condition (2), and we are done.
\end{proof}
We are now ready to prove Proposition \ref{prop:recognition_of_fibrant_objects}:
\begin{proof}
[Proof of Proposition \ref{prop:recognition_of_fibrant_objects}]

We will write $p:X\to S\times\cal D$ for the structure map and $q:X\to S$
for the composition of $p$ and the projection $S\times\cal D\to S$.
We must show that $\overline{X}$ is $S\times\frak P$-fibered if
and only if $X$ is a $\frak P$-bundle over $S$ and the set $M$
consists of the $p$-cocartesian morphisms lifting morphisms in $M_{\cal D}$. 

Suppose first that the object $\overline{X}$ is $S\times\frak P$-fibered.
Then the marked edges of $\overline{X}$ are the $p$-cocartesian
morphisms lifting morphisms in $M_{\cal D}$, so it will suffice to
show that $X$ is a $\frak P$-bundle over $S$. For this, we will
check that $X$ satisfies conditions (a) through (d) of Definition
\ref{def:P-bundle}.

\begin{enumerate}[label=(\alph*)]

\item The map $q:X\to S$ is a cocartesian fibration. Let $x\in X$
be a vertex with image $\pr{v,D}\in S\times\cal D$, and let $f:v\to v'$
be an edge of $S$. We wish to find a $q$-cocartesian lift of $f$.
It suffices to show that there is a $p$-cocartesian lift of the edge
$\pr{f,\id_{D}}$ with source $x$. Since $\pr{f,\id_{D}}$ is a marked
edge of $S^{\sharp}\times\overline{\cal D}$, there is an edge $e:x\to x'$
of $M$ lying over $\pr{f,\id_{D}}$. Since $S\times\overline{\frak P}$
contains all $2$-simplices $\sigma$ such that $\sigma\vert\Delta^{\{0,1\}}=\pr{f,\id_{D}}$,
it follows from Remark \ref{rem:rezk} that $e$ is $p$-cocartesian.
Hence $e$ is the desired lift of $\pr{f,\id_{D}}$.

\item By Remark \ref{rem:P-bund_b}, it will suffice to show that
the map $p:\overline{X}\to S^{\sharp}\times\overline{\cal D}$ is
a marked left fibration. This follows from Remark \ref{rem:mlfib}.

\item For each vertex $v\in S$, the map $X_{v}\to\cal D$ is $\frak P$-fibered.
This follows from the fact that the inclusion $\{v\}\hookrightarrow S$
induces a left Quillen functor $\SS_{/\frak P}^{+}\to\SS_{/S\times\frak P}^{+}$
\cite[Proposition B.2.9]{HA}.

\item Let $f:v\to v'$ be an edge of $S$, and let $e$ be a $p_{v}$-cocartesian
edge lying over an edge in $M_{\cal D}$, where $p_{v}:X_{v}\to\{v\}\times\cal D$
denotes the pullback of $p$. We must show that the functor $f_{!}:X_{v}\to X_{v'}$
carries the morphism $e$ to a cocartesian morphism over $\cal D$.
Since $e$ is locally $p$-cocartesian, Remark \ref{rem:rezk} shows
that $e$ is $p$-cocartesian. So the claim follows from Lemma \ref{lem:fiberwise_colimit}.

\end{enumerate}

This completes the verification of the ``only if'' part.

Conversely, suppose that $X$ is a $\frak P$-bundle over $S$, so
that it satisfies conditions (a) through (d) of Definition \ref{def:P-bundle},
and suppose moreover that $\overline{X}=X_{\natural}$. We must show
that $\overline{X}$ is $S\times\frak P$-fibered. For this, we will
verify conditions (A) and (B) of Remark \ref{rem:mlfib}.

\begin{enumerate}[label=(\Alph*)]

\item The map $\overline{X}\to S^{\sharp}\times\overline{\cal D}$
is a marked left fibration. Indeed, let $x\in X$ be a vertex with
image $\pr{v,D}\in S\times\cal D$, and let $\pr{f,g}:\pr{v,D}\to\pr{v',D'}$
be a marked edge of $S^{\sharp}\times\overline{\cal D}$. We must
find a $p$-cocartesian lift $x\to x''$ lying over $\pr{f,g}$. Using
condition (b), we can find a $p$-cocartesian edge $\widetilde{f}:x\to x'$
lying over $\pr{f,\id_{D}}$. By condition (c), there is a $p_{v'}$-cocartesian
edge $\widetilde{g}:x'\to x''$ lying over $g$, where $p_{v'}:X_{v'}\to\{v'\}\times\cal D$
denotes the pullback of $p$. By conditions (b) and (d) and Lemma
\ref{lem:fiberwise_colimit}, the edge $\widetilde{g}$ is $p$-cocartesian.
Since $p$ is an inner fibration by (b), we can find a $2$-simplex
of $X$ whose boundary can be depicted as 
\[\begin{tikzcd}
	& {x'} \\
	x && {x'',}
	\arrow["{\widetilde{g}}", from=1-2, to=2-3]
	\arrow["{\widetilde{f}}", from=2-1, to=1-2]
	\arrow["{\widetilde{h}}"', from=2-1, to=2-3]
\end{tikzcd}\]such that $p\pr{\widetilde{h}}=\pr{f,g}$. The edge $\widetilde{h}$
is the desired $p$-cocartesian lift of $\pr{f,g}$.

\item Any lifting problem of the form
\[\begin{tikzcd}
	{\{v\}^\sharp\times (K_\alpha) ^\sharp} && {\overline{X}} \\
	{\{v\}^\sharp\times (K_\alpha^\triangleleft) ^\sharp} & {\{v\}^\sharp\times \overline{\mathcal{D}}} & {S^\sharp \times \overline{\mathcal{D}}}
	\arrow[from=1-1, to=1-3]
	\arrow[from=1-3, to=2-3]
	\arrow[from=1-1, to=2-1]
	\arrow[dashed, from=2-1, to=1-3]
	\arrow["{\operatorname{id}\times p_\alpha}"', from=2-1, to=2-2]
	\arrow[hook, from=2-2, to=2-3]
\end{tikzcd}\]admits a solution, and any such lift is a $p$-limit diagram. A solution
to this lifting problem exists because of condition (c). The assertion
on relative limits follows from conditions (a), (b), and (c), together
with \cite[Corollary 4.3.1.15]{HTT}.

\end{enumerate}

This completes the verification of the ``if'' part. 
\end{proof}
We conclude this section with a fiberwise criterion for equivalences
of $\frak P$-bundles (Definition \ref{def:eq_P-bund}):
\begin{prop}
\label{prop:fiberwise_P_equiv}Let $\frak P$ be a commutative categorical
pattern on an $\infty$-category $\cal D$ and let $S$ be a simplicial
set. Let $f:X_{\natural}\to Y_{\natural}$ be a morphism of fibrant
objects of $\SS_{/S\times\frak P}^{+}$. The following conditions
are equivalent:
\begin{enumerate}
\item The map $f$ is a $S\times\frak P$-equivalence.
\item For each vertex $v\in S$, the map
\[
X_{\natural}\times_{S}\{v\}\to Y_{\natural}\times_{S}\{v\}
\]
is a $\frak P$-equivalence.
\item For each vertex $v\in S$, the map $X\times_{S}\{v\}\to Y\times_{S}\{v\}$
is a categorical equivalence.
\end{enumerate}
\end{prop}

The proof of Proposition \ref{prop:fiberwise_P_equiv} needs a lemma.
\begin{lem}
\label{lem:cc_P-bundle}Let $\frak P$ be a commutative categorical
pattern on an $\infty$-category $\cal D$ and let $S$ be a simplicial
set. Let $p:X\to S\times\cal D$ be a $\frak P$-bundle over $S$,
and let $e:x\to x'$ be an edge of $X$. Let $q:X\to S$ denote the
composite $X\xrightarrow{p}S\times\cal D\xrightarrow{\opn{pr}}S$.
The following conditions are equivalent:
\begin{enumerate}
\item The edge $e$ is $q$-cocartesian.
\item The edge $e$ is $p$-cocartesian and lies over an equivalence of
$\cal D$.
\end{enumerate}
\end{lem}

\begin{proof}
Condition (2) implies that $e$ is $p$-cocartesian and $p\pr e$
is $\opn{pr}$-cocartesian, so clearly (2)$\implies$(1). Conversely,
suppose that $e$ is $q$-cocartesian. Write $p\pr e=\pr{u,v}:\pr{s,D}\to\pr{s',D'}$.
The edge $\pr{u,\id_{D}}:\pr{s,D}\to\pr{s',D}$ is marked in $S\times\frak P$,
so it has a $p$-cocartesian lift $\hat e:x\to x''$. Using (2)$\implies$(1),
we find that $\hat e$ is $q$-cocartesian. Thus there is a $2$-simplex
$\sigma$ of $X$ depicted as 
\[\begin{tikzcd}
	& {x''} \\
	x && {x',}
	\arrow["{\hat{e}}", from=2-1, to=1-2]
	\arrow["e"', from=2-1, to=2-3]
	\arrow["{e'}", from=1-2, to=2-3]
\end{tikzcd}\]which lifts the degenerate $2$-simplex $s_{1}\pr u$. The edge $e'$
is necessarily $q$-cocartesian, so it is an equivalence of the $\infty$-category
$X\times_{S}\{s'\}$. Hence its image in $\cal D$ is an equivalence.
Thus $p\pr{e'}$ is $\opn{pr}$-cocartesian. It follows that $e'$
is $p$-cocartesian. Since $\hat e$ is $p$-cocartesian, we deduce
that $e$ is also $p$-cocartesian. Also, since the images of $\hat e$
and $e'$ in $\cal D$ are equivalences, so must the image of $e$.
Hence (1)$\implies$(2), and the proof is complete.
\end{proof}
\begin{proof}
[Proof of Proposition \ref{prop:fiberwise_P_equiv}]The equivalence
(2)$\iff$(3) follows from Proposition \ref{prop:creative}. The implication
(1)$\implies$(2) is clear, for the functor $\SS_{/S\times\frak P}^{+}\to\SS_{/\{v\}\times\frak P}^{+}$
is right Quillen. For (2)$\implies$(1), suppose that condition (2)
is satisfied. By factoring the map $f$ as a trivial cofibration followed
by a fibration, we may assume that $f$ is a fibration. We must show
that $f$ is a trivial fibration. In other words, we must prove the
following:

\begin{enumerate}[label=(\roman*)]

\item The map $f:X\to Y$ is a trivial fibration of simplicial sets.

\item If $e:x\to x'$ is an edge of $X$ such that $f\pr e$ is marked
in $Y_{\natural}$, then $e$ is marked in $X_{\natural}$.

\end{enumerate}

Assertion (ii) follows from assertion (i) and Proposition \ref{prop:recognition_of_fibrant_objects},
because (i) would imply that every edge of $X$ is $f$-cocartesian.
So we will focus on (i).

Let $p:X\to S\times\cal D$, $q:Y\to S\times\cal D$, and $r:S\times\cal D\to S$
denote the projections. Let $\overline{X}$ and $\overline{Y}$ denote
the marked simplicial sets obtained from $X$ and $Y$ by marking
the $rp$-cocartesian edges and $rq$-cocartesian edges, respectively.
By Lemma \ref{lem:cc_P-bundle}, the map $f$ lifts to a map $\overline{X}\to\overline{Y}$
of marked simplicial sets. To prove (i), it suffices to show that
the latter map is a trivial fibration of $\SS_{/S}^{+}$, equipped
with the cocartesian model structure. By Proposition \ref{prop:creative}
and \cite[Proposition 3.1.3.5]{HTT}, the map $f:\overline{X}\to\overline{Y}$
is a cocartesian equivalence over $S$. It will therefore suffice
to show that it is a fibration in the cocartesian model structure
over $S$. Since the object $\overline{Y}\in\SS_{/S}^{+}$ is fibrant
in the cocartesian model structure, it will suffice to show that $f$
is a marked left fibration \cite[Proposition 1.1.7]{JacoMT}. In other
words, we must prove the following:

\begin{enumerate}[label=(\alph*)]

\item The map $X\to Y$ is an inner fibration.

\item The marked edges of $\overline{X}$ are precisely the $f$-cocartesian
edges over the marked edges of $\overline{Y}$.

\item For each vertex $x\in X$ and each marked edge $e:f\pr x\to y$
of $\overline{Y}$, there is a marked edge $x\to x'$ of $\overline{X}$
lying over $e$.

\end{enumerate}

According to \cite[Example B.2.2]{HA}, the map $X_{\natural}\to Y_{\natural}$
has the right lifting property for the maps of marked simplicial sets
of the following forms:
\begin{itemize}
\item [($B_0$)]The inclusion $\{0\}^{\sharp}\subset\pr{\Delta^{1}}^{\sharp}$.
\item [($C_1$)]The inclusion $\pr{\Lambda_{i}^{n}}^{\flat}\subset\pr{\Delta^{n}}^{\flat}$,
for every $0<i<n$.
\end{itemize}
Using the lifting property for morphisms of type ($C_{1}$), we deduce
that $f$ satisfies condition (a). Condition (b) is obvious. For condition
(c), use the lifting property for the morphisms of type ($B_{0}$)
to find a marked edge $\widetilde{e}:x\to x'$ of $X_{\natural}$
lying over $e$. Using Lemma \ref{lem:cc_P-bundle}, we find that
$\widetilde{e}$ is marked in $\overline{X}$. The proof is now complete.
\end{proof}
In Section \ref{sec:Rectification}, we will establish a categorical
equivalence
\[
\Fun\pr{S,\frak P\-\Fib}\simeq\frak P\-\Bund\pr S.
\]
For the proof of the equivalence, we need to review Lurie's straightening
and unstraightening constructions. This is the content of the next
section (Section \ref{sec:St_Un}).

\section{\label{sec:St_Un}Review of Straightening and Unstraightening}

In this section, we review Lurie's straightening and unstraightening
functors and related constructions. We begin with the definition and
motivation of Lurie's unstraightening functor in Subsection \ref{subsec:Definition-and-Intuition}.
In Subsection \ref{subsec:relnerve}, we recall a toy version of the
unstraightening functor, called the relative nerve functor. In Subsection
\ref{subsec:unstraightening_process}, we take a closer look at the
functors between the fibers of the cocartesian fibrations obtained
by the unstraightening construction. Since our definition of unstraightening
looks different from the one in \cite{HTT}, we give a proof of the
equivalence of the two functors in Subsection \ref{subsec:comparison_with_Lurie}.

\subsection{\label{subsec:Definition-and-Intuition}Definitions and Intuition}

Given a small category $\cal C$ and a functor $F:\cal C\to\sf{Cat}$,
where $\sf{Cat}$ denotes the category of small categories, we define
a category $\int_{\cal C}F=\int F$, called the \textbf{Grothendieck
construction} of $F$, as follows:
\begin{itemize}
\item The objects are the pairs $\pr{C,X}$, where $C\in\cal C$ and $X\in FC$.
\item A morphism $\pr{C,X}\to\pr{C',X'}$ is a pair $\pr{f,g}$, where $f:C\to C'$
is a morphism of $\cal C$ and $g:Ff\pr X\to X'$ is a morphism of
$FC'$.
\item The composition is defined in the obvious manner.
\end{itemize}
Lurie's \textbf{unstraightening construction}, which we now define,
is a homotopy coherent version of the Grothendieck construction.
\begin{defn}
\label{def:St_Un}For each $n\geq0$, let $\Psi_{n}:\widetilde{\fr C}[\Delta^{n}]\to\SS$
denote the simplicial functor defined by $\Psi_{n}\pr i=N\pr{P_{\max i}}^{\op}$,
where $P_{\max i}$ denotes the poset of subsets of $[n]$ with maximal
element $i$. (See \nameref{subsec:not_term} for the definition of
$\widetilde{\fr C}[-]$.) The structure map
\[
N\pr{P_{i,j}}^{\op}\times N\pr{P_{\max i}}^{\op}\to N\pr{P_{\max j}}^{\op}
\]
is induced by the operation of taking unions. 

Now let $S$ be a simplicial set and let $\phi:\widetilde{\fr C}[S]\to\cal C_{\Delta}$
be a simplicial functor. Given a simplicial functor $F:\cal C_{\Delta}\to\SS$,
we define the (\textbf{covariant}) \textbf{unstraightening} $\widetilde{\Un}_{\phi}\pr F\in\SS_{/S}$
of $F$ as follows:
\begin{itemize}
\item The $n$-simplices of $\widetilde{\Un}_{\phi}\pr F$ are the pairs
$\pr{\sigma,\alpha}$, where $\sigma$ is an $n$-simplex of $S$
and $\alpha$ is a simplicial natural transformation $\Psi_{n}\to F\circ\phi\circ\widetilde{\fr C}[\sigma]$
of simplicial functors $\widetilde{\fr C}[\Delta^{n}]\to\SS$. 
\item Given a poset map $u:[m]\to[n]$, the map $u^{*}:\widetilde{\Un}_{\phi}\pr F_{n}\to\widetilde{\Un}_{\phi}\pr F_{m}$
is induced by the natural transformation
\[
\Psi_{m}\to\Psi_{n}\circ\widetilde{\fr C}[u]
\]
determined by the direct image map $u\pr -:P_{\max i}\to P_{\max u\pr i}$.
\end{itemize}
The functor $\cun_{\phi}:\Fun\pr{\cal C_{\Delta},\SS}\to\SS_{/S}$
has a left adjoint (see Subsection \ref{subsec:comparison_with_Lurie}),
denoted by $\widetilde{\St}_{\phi}$. This is the (\textbf{covariant})
\textbf{straightening functor}.

If $G:\cal C_{\Delta}\to\SS^{+}$ is a simplicial functor, we define
the (\textbf{marked}) (\textbf{covariant}) \textbf{unstraightening}
$\widetilde{\Un}_{\phi}^{+}\pr G\in\SS_{/S}^{+}$ of $G$ by marking
the edges $\pr{x\to y,\alpha}$ of $\cun_{\phi}\pr{G_{\flat}}$ such
that the map $\alpha:\Psi_{1}\pr 1\to G_{\flat}\pr{\phi\pr y}$ classifies
a marked edge of $G\pr{\phi\pr y}$. The functor $\cun_{\phi}^{+}:\Fun\pr{\cal C_{\Delta},\SS^{+}}\to\SS_{/S}^{+}$
has a left adjoint (see Subsection \ref{subsec:comparison_with_Lurie}),
denoted by $\widetilde{\St}_{\phi}^{+}$. This is the (\textbf{marked})
(\textbf{covariant}) \textbf{straightening functor}.
\end{defn}

\begin{rem}
Lurie's definition of the unstraightening functor does not resemble
the one in Definition \ref{def:St_Un}. Nevertheless, our definition
yields the same functor as Lurie's. The verification is somewhat tedious
and only distracts us from the main discussion, so we will defer it
to Subsection \ref{subsec:comparison_with_Lurie}.
\end{rem}

\begin{rem}
Let us see why the unstraightening functor may be regarded as a generalization
of the Grothendieck construction. Let $S$ be a simplicial set, let
$\phi:\widetilde{\fr C}[S]\to\cal C_{\Delta}$ be a simplicial functor,
and let $F:\cal C_{\Delta}\to\SS$ be another simplicial functor.
Simplices of $\widetilde{\Un}_{\phi}\pr F$ of low dimensions can
be described as follows. We will write $F'=F\phi$. Also, if $f$
is an edge of $S$, we will denote the corresponding morphism of $\widetilde{\fr C}[S]$
by $f$. 
\begin{enumerate}
\item A vertex is a pair $\pr{s,v}$, where $s$ is a vertex of $S$ and
$v$ is a vertex of $F'\pr s$. 
\item An edge $\pr{s,v}\to\pr{s',v'}$ is a pair $\pr{f,g}$, where $f:s\to s'$
is an edge of $S$ and $g:F'f\pr v\to v'$ is an edge of $F'\pr{s'}$.
\item Suppose we are given edges $\pr{f_{ij},g_{ij}}:\pr{s_{i},v_{i}}\to\pr{s_{j},v_{j}}$
of $\cun_{\phi}\pr F$ for $0\leq i<j\leq2$. Then a $2$-simplex
in $\cun_{\phi}\pr F$ whose boundary is depicted as 
\[\begin{tikzcd}
	& {(s_1,v_1)} \\
	{(s_0,v_0)} && {(s_2,v_2)}
	\arrow["{(f_{01},g_{01})}", from=2-1, to=1-2]
	\arrow["{(f_{12},g_{12})}", from=1-2, to=2-3]
	\arrow["{(f_{02},g_{02})}"', from=2-1, to=2-3]
\end{tikzcd}\]consists of the following data:
\begin{itemize}
\item A $2$-simplex $\sigma$ of $S$ whose boundary is depicted as
\[\begin{tikzcd}
	& {s_1} \\
	{s_0} && {s_2.}
	\arrow["{f_{01}}", from=2-1, to=1-2]
	\arrow["{f_{12}}", from=1-2, to=2-3]
	\arrow["{f_{02}}"', from=2-1, to=2-3]
\end{tikzcd}\]
\item A square $\Delta^{1}\times\Delta^{1}\to F's_{2}$ depicted as
\[\begin{tikzcd}
	{F'f_{12}F'f_{01}(v_0)} & {F'f_{12}(v_1)} \\
	{F'f_{02}(v_0)} & {v_2}
	\arrow["\alpha"', from=1-1, to=2-1]
	\arrow["{F'f_{12}(g_{01})}", from=1-1, to=1-2]
	\arrow["{g_{12}}", from=1-2, to=2-2]
	\arrow["{g_{02}}"', from=2-1, to=2-2]
\end{tikzcd}\]where the map $\alpha$ is determined by the $2$-simplex $\sigma$.
\end{itemize}
\end{enumerate}
This explains the sense in which the unstraightening construction
is a homotopy coherent version of the Grothendieck construction: Given
an $n$-simplex $\sigma$ of $S$, a simplex of $\cun_{\phi}\pr F$
lying over $\sigma$ carries the data of vertices $v_{i}\in F'\pr{\sigma\pr i}$
and coherent homotopy filling the $n$-dimensional cube arising from
$\sigma$ and $v_{i}$ which expresses all possible ways the edges
of $\sigma$ can be composed.
\end{rem}

The following result, which is an $\infty$-categorical version of
the equivalence of Grothendieck fibrations and presheaves, is one
of the monumental achievements of \cite{HTT}. Recall that a simplicial
functor is said to be a \textbf{weak equivalence} if it induces a
categorical equivalence between the homotopy categories and weak homotopy
equivalences between the hom-simplicial sets.
\begin{thm}
\cite[Theorem 2.2.1.2, 3.2.0.1]{HTT}\label{thm:st_un}Let $S$ be
a simplicial set, $\cal C$ a simplicial category, and $\phi:\fr C[S]\to\cal C$
a weak equivalence of simplicial categories. The adjunctions
\[
\St_{\phi}:\pr{\SS_{/S}}_{\mathrm{contra}}\to\Fun^{s}\pr{\cal C_{\Delta},\SS}_{\mathrm{proj}}:\Un_{\phi}
\]
and
\[
\St_{\phi}^{+}:\pr{\SS_{/S}^{+}}_{\mathrm{cart}}\adj\Fun^{s}\pr{\cal C_{\Delta},\u{\SS^{+}}}_{\mathrm{proj}}:\Un_{\phi}^{+}
\]
are Quillen equivalences.
\end{thm}

In light of Remark \ref{def:St_Un}, it is not hard to imagine that
the unstraightening construction over the nerve of an ordinary category
is close to the ordinary Grothendieck construction. This is true,
as the following Proposition shows:
\begin{prop}
\label{prop:unstraightening_vs_Gr_const}Let $\cal C$ be a small
category and let $\varepsilon:\widetilde{\fr C}[N\pr{\cal C}]\to\cal C$
denote the counit map. Suppose we are given a commutative diagram
\[\begin{tikzcd}
	{\mathcal{C}} & {\mathsf{sSet}^+} \\
	{\mathsf{Cat}} & {\mathsf{sSet}}
	\arrow["F"', from=1-1, to=2-1]
	\arrow["N"', from=2-1, to=2-2]
	\arrow["{F^+}", from=1-1, to=1-2]
	\arrow["{\text{forget}}", from=1-2, to=2-2]
\end{tikzcd}\]of functors. There is an isomorphism 
\[
\widetilde{\Un}_{\varepsilon}^{+}\pr{F^{+}}\cong\pr{N\pr{\int F},M}
\]
of marked simplicial sets over $N\pr{\cal C}^{\sharp}$, where $M$
is the set of edges of $N\pr{\int F}$ corresponding to a morphism
$\pr{f,g}:\pr{C,X}\to\pr{D,Y}$ such that the induced map $g:\pr{Ff}X\to Y$
is marked in $F^{+}\pr C$. The isomorphism is natural in $F^{+}$.
\end{prop}

\begin{proof}
Let us first construct an isomorphism
\[
\widetilde{\Un}_{\varepsilon}\pr{N\circ F}\cong N\pr{\int F}
\]
of simplicial sets over $N\pr{\cal C}$. Let $\sigma=\pr{C_{0}\to\cdots\to C_{n}}$
be an $n$-simplex of $N\pr{\cal C}$. By definition, an $n$-simplex
of $\widetilde{\Un}_{\varepsilon}\pr{N\circ F}$ which lies over $\sigma$
is a simplicial natural transformation $\alpha:\Psi_{n}\to F\circ\varepsilon\circ\widetilde{\fr C}[\sigma]$.
The data of $\alpha$ is equivalent a collection of functors $\{\alpha_{i}:P_{\max i}^{\op}\to FC_{i}\}_{0\leq i\leq n}$
such that the diagram 
\[\begin{tikzcd}
	{P_{\max i}^\mathrm{op}\times P_{ij}^\mathrm{op}} & {P^{\mathrm{op}}_{\max i}} & {FC_i} \\
	{P_{\max j}^\mathrm{op}} && {FC_j}
	\arrow["\cup"', from=1-1, to=2-1]
	\arrow["{\operatorname{pr}}", from=1-1, to=1-2]
	\arrow["{\alpha_j}"', from=2-1, to=2-3]
	\arrow["{\alpha_i}", from=1-2, to=1-3]
	\arrow["{Ff_{ij}}", from=1-3, to=2-3]
\end{tikzcd}\]commutes for each pair of integers $0\leq i\le j\leq n$, where $f_{ij}:C_{i}\to C_{j}$
is the map determined by $\sigma$. On the other hand, an $n$-simplex
of $N\pr{\int F}$ lying over $\sigma$ is a collection of functors
$\{\beta_{i}:[i]\to FC_{i}\}_{0\leq i\leq n}$ such that the diagram
\[\begin{tikzcd}
	{[i]} & {FC_i} \\
	{[j]} & {FC_{j}}
	\arrow["{\beta_i}", from=1-1, to=1-2]
	\arrow["{Ff_{ij}}", from=1-2, to=2-2]
	\arrow["{\beta_j}"', from=2-1, to=2-2]
	\arrow[hook, from=1-1, to=2-1]
\end{tikzcd}\]commutes for each pair of integers $0\leq i\le j\leq n$. With this
in mind, we define a pair of maps
\[
\varphi:\widetilde{\Un}_{\varepsilon}^{+}\pr{N\circ F}_{n}\times_{\cal C_{n}}\{\sigma\}\rl N\pr{\int F}\times_{\cal C_{n}}\{\sigma\}:\psi
\]
as follows. Given an element $\{\alpha_{i}\}_{0\leq i\leq n}$ of
$\widetilde{\Un}_{\varepsilon}^{+}\pr{N\circ F}_{n}\times_{\cal C_{n}}\{\sigma\}$,
we set $\varphi\pr{\{\alpha_{i}\}_{0\leq i\leq n}}=\{\beta_{i}\}_{0\leq i\leq n}$,
where $\beta_{i}:[i]\to FC_{i}$ is defined by
\[
\beta_{i}\pr{s\leq t}=\alpha_{i}\pr{\{s,\dots,i\}\supset\{t,\dots,i\}}.
\]
Conversely, given an element $\{\beta_{i}\}_{0\leq i\leq n}$ of $N\pr{\int F}\times_{\cal C_{n}}\{\sigma\}$,
we set $\psi\pr{\{\beta_{i}\}_{0\leq i\leq n}}=\{\alpha_{i}\}_{0\le i\leq n}$,
where $\alpha_{i}:P_{i}^{\op}\to FC_{i}$ is defined by
\[
\alpha_{i}\pr{S\supset T}=\beta_{i}\pr{\min S\leq\min T}.
\]

We claim that $\varphi$ and $\psi$ are inverses of each other. It
is clear that the composite $\varphi\psi$ is the identity map. It
thus suffices to show that $\psi\varphi$ is the identity map. We
must show that, given an element $\{\alpha_{i}\}_{0\leq i\leq n}$
of $\widetilde{\Un}_{\varepsilon}^{+}\pr{N\circ F}_{n}\times_{\cal C_{n}}\{\sigma\}$,
an integer $0\leq i\leq n$, and a morphism $S\supset T$ of $P_{\max i}^{\op}$,
we have
\[
\alpha_{i}\pr{S\supset T}=\alpha_{i}\pr{\{s,\dots,i\}\supset\{t,\dots,i\}},
\]
where $s=\min S$ and $t=\min T$. For this, it suffices to show that
$\alpha_{i}\pr{S\supset T}$ depends only on $s$ and $t$. This follows
from the computation
\begin{align*}
\alpha_{i}\pr{S\supset T} & =\alpha_{i}\pr{\{s\}\cup T\supset T}\circ\alpha_{i}\pr{S\supset\{s\}\cup T}\\
 & =\alpha_{i}\pr{\{s,t\}\cup T\supset\{t\}\cup T}\circ\alpha_{i}\pr{\{s\}\cup S\supset\{s\}\cup T}\\
 & =Ff_{ti}\pr{\{s,t\}\supset\{t\}}\circ Ff_{si}\pr{\alpha_{s}\pr{\{s\}\supset\{s\}}}\\
 & =Ff_{ti}\pr{\{s,t\}\supset\{t\}}.
\end{align*}

We have thus obtained a bijection $\widetilde{\Un}_{\varepsilon}\pr{N\circ F}_{n}\cong N\pr{\int F}_{n}$.
This bijection is natural in $n$ and commutes with the projections
to $N\pr{\cal C}_{n}$, so it gives rise to the desired isomorphism
of simplicial sets $\widetilde{\Un}_{\varepsilon}\pr{N\circ F}\cong N\pr{\int F}$
over $N\pr{\cal C}$. The assertion on the markings and the naturality
with respect to $F^{+}$ follows by inspection.
\end{proof}
We conclude this subsection with a remark on enrichment of the unstraightening
functor.
\begin{rem}
The unstraightening functor admits a simplicial enrichment. Let $S$
be a simplicial set and let $\phi:\widetilde{\fr C}[S]\to\cal C_{\Delta}$
be a simplicial functor. Let $\u{\Nat}^{s}\pr{-,-}$ denote the hom-simplicial
set of the simplicial category of simplicial functors $\cal C_{\Delta}\to\SS$.
Thus if $F,G:\cal C_{\Delta}\to\SS$ are simplicial functors, then
an $n$-simplex of $\u{\Nat}^{s}\pr{F,G}$ is a simplicial natural
transformation $\Delta^{n}\times F\to G$. We define a map
\[
\theta:\u{\Nat}^{s}\pr{F,G}\times\cun_{\phi}\pr F\to\cun_{\phi}\pr G
\]
as follows. The poset maps $\{\min:P_{\max i}^{\op}\to[n]\}_{0\leq i\leq n}$
induces a simplicial natural transformation $\varphi:\Psi_{n}\to\delta\pr{\Delta^{n}}$,
where $\delta\pr{\Delta^{n}}$ is the constant simplicial functor
at the object $\Delta^{n}$. The map $\theta$ maps an $n$-simplex
\[
\pr{\alpha:\Delta^{n}\times F\to G,\sigma:\Delta^{n}\to S,\beta:\Psi_{n}\to F\circ\phi\circ\widetilde{\fr C}[\sigma]}
\]
of $\u{\Nat}^{s}\pr{F,G}\times\cun_{\phi}\pr F$ to the $n$-simplex
\[
\pr{\sigma,\Psi_{n}\xrightarrow{\pr{\varphi,\beta}}\pr{\Delta^{n}\times F}\circ\phi\circ\widetilde{\fr C}[\sigma]\xrightarrow{\alpha\phi\widetilde{\fr C}[\sigma]}G\circ\phi\circ\widetilde{\fr C}[\sigma]}.
\]
The adjoint of $\theta$ determines a map 
\[
\u{\Nat}^{s}\pr{F,G}\to\Fun_{S}\pr{\cun_{\phi}\pr F,\cun_{\phi}\pr G},
\]
which endows $\cun_{\phi}$ with a simplicial enrichment. Similarly,
the marked unstraightening admits a simplicial enrichment.
\end{rem}

\subsection{\label{subsec:relnerve}Relative Nerve}

In Subsection \ref{subsec:Definition-and-Intuition}, we introduced
the unstraightening functor as a natural generalization of the Grothendieck
construction in the setting of $\infty$-categories. For unstraightening
over the nerve of ordinary categories, there is another natural generalization
of the Grothendieck construction, called the relative nerve construction.
In this subsection, we recall this construction and compare it with
the unstraightening functor.
\begin{defn}
\cite[$\S$3.2.5]{HTT} Let $\cal C$ be a small category and $F:\cal C\to\SS$
a functor. The \textbf{nerve of $\cal C$ relative to $F$} is the
simplicial set $\int F=\int_{\cal C}F$ (also denoted by $N_{F}\pr{\cal C}$
in \cite{HTT}) whose $n$-simplex is a pair $\pr{\sigma,\alpha}$,
where $\sigma:C_{0}\to\cdots\to C_{n}$ is an $n$-simplex of $N\pr{\cal C}$,
and $\alpha$ is a natural transformation $\{\Delta^{i}\to F\pr{C_{i}}\}_{0\leq i\leq n}$
of functors $[n]\to\SS$. If $u:[m]\to[n]$ is a poset map, then the
map $\pr{\int F}_{m}\to\pr{\int F}_{n}$ is given by $\pr{\sigma,\{\alpha_{i}\}_{0\leq i\leq n}}\mapsto\pr{u^{*}\sigma,\{\alpha_{u\pr j}\}_{0\leq j\leq m}}$.

Given a functor $G:\cal C\to\SS^{+}$, we let $\int^{+}G=\int_{\cal C}^{+}G\in\SS_{/\cal C}^{+}$
denote the marked simplicial set obtained from $\int G_{\flat}$ by
marking the edges $\pr{C_{0}\to C_{1},\{x_{I}\}_{\emptyset\neq I\subset[1]}}$
such that the map $x_{[1]}:\Delta^{1}\to G_{\flat}\pr{C_{1}}$ classifies
a marked edge of $G\pr{C_{1}}$.
\end{defn}

\begin{rem}
The relative nerve extends the familiar Grothendieck construction
in the following sense: If $\cal C$ is an ordinary category and $F:\cal C\to\sf{Cat}$
is a functor, there is an isomorphism of simplicial sets
\[
N\pr{\int F}\cong\int\pr{N\circ F}
\]
between the Grothendieck construction of $F$ and the relative nerve
of the composite $\cal C\xrightarrow{F}\sf{Cat}\xrightarrow{N}\SS$. 
\end{rem}

\begin{rem}
\label{rem:rel_nerve_leftadj}Let $\cal C$ be an ordinary category.
The relative nerve functor $\int:\Fun\pr{\cal C,\SS}\to\SS_{/N\pr{\cal C}}$
admits a left adjoint, given by 
\[
L\pr X=X\times_{N\pr{\cal C}}N\pr{\cal C_{/\bullet}}.
\]
To see why $L$ is a left adjoint of $\int$, let $\sigma:[n]\to\cal C$
be an $n$-simplex of $N\pr{\cal C}$. The maps $\{\Delta^{i}\to\Delta^{n}\times_{N\pr{\cal C}}N\pr{\cal C_{/C_{i}}}\}_{0\leq i\leq n}$
which classify the simplices of the form $\pr{0\to\cdots\to i,\,C_{0}\to\cdots\to C_{i}\xrightarrow{\opn{id}}C_{i}}$
determine a natural transformation $\eta:\Delta^{\bullet}\to L\pr{\sigma}\circ\sigma$
of functors $[n]\to\SS$, and $\eta$ exhibits $L\pr{\sigma}$ as
a left Kan extension of $\Delta^{\bullet}:[n]\to\SS$ along $\sigma:[n]\to\cal C$.
It follows that there is a bijection
\[
\Fun\pr{\cal C,\SS}\pr{L\pr{\sigma},F}\cong\Fun\pr{[n],\SS}\pr{\Delta^{\bullet},F\sigma}\cong\SS_{/N\pr{\cal C}}\pr{\sigma,\int F}.
\]
This proves that $L$ is a left adjoint of $\int$. Similarly, the
left adjoint $L^{+}:\SS_{/N\pr{\cal C}}^{+}\to\Fun\pr{\cal C,\SS^{+}}$
of $\int^{+}$ is given by $L^{+}\pr{\overline{X}}=\overline{X}\times_{N\pr{\cal C}^{\sharp}}N\pr{\cal C_{/\bullet}}^{\sharp}$.
\end{rem}

We now consider an extension the isomorphism of Proposition \ref{prop:unstraightening_vs_Gr_const}
by using the relative nerve functor.
\begin{defn}
Let $\cal C$ be an ordinary category. We define a natural transformation
\[
\int_{\cal C}\pr -\to\cun_{\varepsilon}\pr -,
\]
where $\varepsilon=\varepsilon_{\cal C}:\widetilde{\fr C}[N\pr{\cal C}]\to\cal C$
is the counit map, as follows. Let $F:\cal C\to\SS$ be a functor.
An $n$-simplex of $\int F$ is a pair $\pr{\sigma,\alpha}$, where
$\sigma:[n]\to\cal C$ is a functor and $\alpha:\Delta^{\bullet}\to F\sigma$
is a natural transformation of functors $[n]\to\SS.$ Precomposing
the simplicial functor $\varepsilon_{[n]}:\widetilde{\fr C}[\Delta^{n}]\to[n]$,
we obtain a natural transformation
\[
\alpha':\Delta^{\bullet}\varepsilon_{[n]}\to F\varepsilon_{\cal C}\widetilde{\fr C}[\sigma]:\fr{\widetilde{C}}[\Delta^{n}]\to\SS.
\]
Now there is a simplicial natural transformation $\chi_{n}:\Psi_{n}\to\Delta^{\bullet}\varepsilon_{[n]}$,
given by $\{\min:P_{\max i}^{\op}\to[i]\}_{0\leq i\leq n}$. Precomposing
$\chi_{n}$ to $\alpha'$, we obtain a natural transformation $\alpha'':\Psi_{n}\to F\varepsilon_{\cal C}\widetilde{\fr C}[\sigma]$.
We declare that the image of $\pr{\sigma,\alpha}$ is given by $\pr{\sigma,\alpha''}$. 

Note that the same map determines a natural transformation $\int_{\cal C}^{+}\pr -\to\cun_{\varepsilon}^{+}\pr -$.
\end{defn}

The following proposition appears as \cite[Lemma 3.2.5.17]{HTT}.
We give a quick proof (assuming some results from \cite{HTT}) for
the readers' convenience.
\begin{prop}
\label{prop:relative_nerve}Let $\cal C$ be an ordinary category. 
\begin{enumerate}
\item Let $F:\cal C\to\SS$ be a projectively fibrant functor. The map
\[
\int F\to\cun_{\varepsilon}\pr F
\]
is a covariant equivalence over $N\pr{\cal C}$.
\item Let $F:\cal C\to\SS^{+}$ be a projectively fibrant functor. The map
\[
\int^{+}F\to\cun_{\varepsilon}^{+}\pr F
\]
is a cocartesian equivalence over $N\pr{\cal C}$.
\end{enumerate}
\end{prop}

\begin{proof}
We will only prove part (2), for part (1) can be proved similarly.
Let $L^{+}:\SS_{/N\pr{\cal C}}^{+}\to\Fun\pr{\cal C,\SS^{+}}$ denote
the left adjoint of $\int_{\cal C}\pr -$. By \cite[Proposition 3.2.5.18]{HTT},
the functor $L^{+}$ is left Quillen. It will therefore suffice to
show that the natural transformation
\[
\alpha:\cst_{\varepsilon}^{+}\pr -\to L^{+}
\]
is a natural weak equivalence. Since $\SS_{/N\pr{\cal C}}^{+}$ is
generated under homotopy colimits by $\pr{\Delta^{0}}^{\sharp},\pr{\Delta^{1}}^{\sharp},\pr{\Delta^{n}}^{\flat}$,
it suffices to check that $\alpha$ is a weak equivalence at these
objects. Since the map $\int_{\cal C}^{+}\pr -\to\cun_{\varepsilon}^{+}\pr -$
is bijective on vertices, edges, and marked edges, $\alpha$ is an
isomorphism in these cases.
\end{proof}
In the case where $\cal C=[0]$, the relative nerve functor is naturally
isomorphic to the identity functor. Thus we obtain:
\begin{cor}
\cite[Proposition 3.2.1.14]{HTT}\label{cor:3.2.1.14} Let $\cal X$
be an $\infty$-category. The map
\[
\cal X^{\natural}\to\cun_{\Delta^{0}}^{+}\pr{\cal X^{\natural}}
\]
is a weak equivalence of marked simplicial sets. Consequently, the
map
\[
\cal X\to\widetilde{\Un}_{\Delta^{0}}\pr{\cal X}
\]
is an equivalence of $\infty$-categories.
\end{cor}

\subsection{\label{subsec:unstraightening_process}Understanding the Unstraightening
Process}

Given a cocartesian fibration $p:X\to S$ of simplicial sets and an
edge $f:s\to s'$ in $S$, there is a functor $f_{!}:X_{s}\to X_{s'}$,
which is well-defined up to natural equivalence. The functor $f_{!}$
is defined as follows: Choose a cocartesian natural transformation
$X_{s}\times\Delta^{1}\to X$ rendering the diagram 
\[\begin{tikzcd}
	{X_s\times \{0\}} && X \\
	{X_s\times\Delta ^1} & {\Delta^1} & S
	\arrow[from=1-1, to=1-3]
	\arrow[from=1-3, to=2-3]
	\arrow[from=1-1, to=2-1]
	\arrow[from=2-1, to=2-2]
	\arrow["f"', from=2-2, to=2-3]
	\arrow["h"{description}, dashed, from=2-1, to=1-3]
\end{tikzcd}\]commutative, and then define $f_{!}=h\vert X_{s}\times\{1\}$. In
general, we cannot hope to construct the functor $f_{!}$ ``by hand.''
But for some special class of cocartesian fibrations, there is a canonical
choice for the functor $f_{!}$. In this subsection, we will show
that cocartesian fibrations arising from the unstraightening functors
form one such class. 

Let $S$ be a simplicial set, let $\phi:\widetilde{\fr C}[S]\to\cal C_{\Delta}$
be a simplicial functor, let $f:s\to s'$ be an edge of $S$, and
let $F:\cal C_{\Delta}\to\SS^{+}$ be a simplicial functor. We define
a morphism of simplicial sets
\[
h:\cun_{\phi}\pr{F_{\flat}}_{s}\times\Delta^{1}\to\cun_{\phi}\pr{F_{\flat}}
\]
as follows. An $n$-simplex of $\cun_{\phi}\pr{F_{\flat}}_{s}\times\Delta^{1}$
is a pair $\pr{\{\alpha_{i}:N\pr{P_{\max i}^{\op}}\to F_{\flat}\pr{\phi\pr s}\}_{0\leq i\leq n},u}$,
where the collection $\{\alpha_{i}\}_{0\leq i\leq n}$ makes the diagram
\[\begin{tikzcd}
	{N(P_{\max i}^\mathrm{op}) \times N(P_{ij}^\mathrm{op})} & {N(P_{\max j}^\mathrm{op})} \\
	{N(P^{\mathrm{op}}_{\max i})} & {F_\flat(\phi(s))}
	\arrow["\cup", from=1-1, to=1-2]
	\arrow["{\operatorname{pr}}"', from=1-1, to=2-1]
	\arrow["{\alpha_j}", from=1-2, to=2-2]
	\arrow["{\alpha_i}"', from=2-1, to=2-2]
\end{tikzcd}\]commutative for each pair of integers $0\le i\leq j\leq n$, and $u:[n]\to[1]$
is a poset map. We set 
\[
h\pr{\{\alpha_{i}\}_{0\leq i\leq n},u}=\pr{u^{*}f,\{\beta_{i}\}_{0\leq i\leq n}},
\]
where 
\[
\beta_{i}=\begin{cases}
\alpha_{i} & \text{if }u\pr i=0,\\
F_{\flat}\pr{\phi\pr f}\circ\alpha_{i} & \text{if }u\pr i=1.
\end{cases}
\]
Here we wrote $f$ for the morphism of $\widetilde{\fr C}[S]$ determined
by the edge $f\in S_{1}$ by abusing notation. This defines an explicit
functor $\cun_{\phi}\pr{F_{\flat}}_{s}\times\Delta^{1}\to\cun_{\phi}\pr{F_{\flat}}$
which fits into the commutative diagram
\[\begin{tikzcd}
	{\widetilde{\operatorname{Un}}^+_\phi(F)_s\times (\{0\})^\sharp} & {\widetilde{\operatorname{Un}}^+_\phi(F)} \\
	{\widetilde{\operatorname{Un}}^+_\phi(F)_s\times (\Delta^1)^\sharp} & {S^\sharp.}
	\arrow[from=2-1, to=2-2]
	\arrow[from=1-1, to=2-1]
	\arrow[from=1-1, to=1-2]
	\arrow[from=1-2, to=2-2]
	\arrow[dashed, from=2-1, to=1-2]
\end{tikzcd}\]We always understand that the functor $f_{!}:\cun_{\phi}^{+}\pr F_{s}\to\cun_{\phi}^{+}\pr F_{s'}$
is obtained as the restriction $h\vert\cun_{\phi}^{+}\pr F_{s}\times\{1\}^{\sharp}$.
This will ensure the following:
\begin{prop}
\label{prop:rectification}Let $S$ be a simplicial set, let $\phi:\widetilde{\fr C}[S]\to\cal C_{\Delta}$
be a simplicial functor, let $f:s\to s'$ be an edge in $S$. The
assignment $F\mapsto\pr{f_{!}:\cun_{\phi}^{+}\pr F_{s}\to\cun_{\phi}^{+}\pr F_{s'}}$
defines a functor
\[
\varepsilon_{f}:\Fun^{s}\pr{\cal C_{\Delta},\SS^{+}}\to\Fun\pr{[1],\SS^{+}}.
\]
Moreover, the functor $\varepsilon_{f}$ admits a natural transformation
from the functor
\[
F\mapsto F\pr{\phi\pr f}
\]
whose components at projectively fibrant functors are weak equivalences.
\end{prop}

\begin{proof}
The first assertion is evident from the construction. The second assertion
follows from the isomorphism $\cun_{\phi}^{+}\pr F_{s}\cong\cun_{\Delta^{0}}^{+}\pr{F\pr{\phi\pr s}}$
and Corollary \ref{cor:3.2.1.14}. 
\end{proof}

\subsection{\label{subsec:comparison_with_Lurie}Equivalence with Lurie's Straightening
Functor}

In Subsection \ref{subsec:Definition-and-Intuition}, we defined the
unstraightening functor. In this subsection, we will justify our terminology
by showing that our unstraightening functor is right adjoint to Lurie's
straightening functor, defined in \cite[$\S$2.2.1, 3.2.1]{HTT}. Throughout
this subsection, we will work with a fixed simplicial set $S$ and
a simplicial functor $\phi:\widetilde{\fr C}[S]\to\cal C_{\Delta}$.

\subsubsection{Unmarked Case}

Let $p:X\to S$ be a morphism of simplicial sets. Lurie's model of
the straightening of $p$ with respect to $\phi$, which we temporarily
denote by $\cst'_{\phi}\pr p$, is defined\footnote{More precisely, the functor $\cst'_{\phi}\pr -$ is equal to the composite
\[
\SS_{/S}\xrightarrow[\cong]{\pr -^{\op}}\SS_{/S^{\op}}\xrightarrow{\St_{\widetilde{\phi}^{\op}}}\Fun^{s}\pr{\widetilde{\cal C}_{\Delta},\SS}\xrightarrow[\cong]{\pr -^{\circ}}\Fun^{s}\pr{\cal C_{\Delta},\SS},
\]
where $\St_{\widetilde{\phi}^{\op}}$ is defined as in {\cite[$\S$2.2.1]{HTT}},
and for $F\in\Fun^{s}\pr{\widetilde{\cal C}_{\Delta},\SS}$, the simplicial
functor $\pr F^{\circ}$ is defined by $\pr F^{\circ}\pr C=F\pr C^{\op}$.
A similar remark applies to the marked case (Subsubsection \ref{subsec:Marked-Case}).} as the composite
\[
\cal C_{\Delta}\to\widetilde{\fr C}[X^{\lcone}]\amalg_{\widetilde{\fr C}[X]}\cal C_{\Delta}\xrightarrow{\pr{\widetilde{\fr C}[X^{\lcone}]\amalg_{\widetilde{\fr C}[X]}\cal C_{\Delta}}\pr{\infty,-}}\SS,
\]
where $\infty\in X^{\lcone}$ denotes the cone point. We claim that
$\cst'_{\phi}$ is a left adjoint of $\cun_{\phi}$, so that we may
take $\cst_{\phi}=\cst'_{\phi}$. Given a simplicial functor $F:\cal C_{\Delta}\to\SS$,
let $\cal C_{\Delta}^{F}$ denote the simplicial category obtained
from $\cal C_{\Delta}$ by adjoining a single object $\infty$ to
$\cal C_{\Delta}$, with hom-simplicial sets given by
\begin{align*}
\cal C_{\Delta}^{F}\pr{\infty,C} & =F\pr C,\\
\cal C_{\Delta}^{F}\pr{C,\infty} & =\emptyset,\\
\cal C_{\Delta}^{F}\pr{\infty,\infty} & =\Delta^{0},
\end{align*}
for $C\in\cal C_{\Delta}$. This simplicial category has the following
universal property, which follows directly from the definitions:
\begin{prop}
\label{prop:precosheaf_by_cat}Let $F:\cal A\to\SS$ and $\varphi:\cal A\to\cal B$
be simplicial functors of simplicial categories, and let $B_{\infty}\in\cal B$
be an object. Suppose we are given a collection of maps $S=\{FA\to\cal B\pr{B_{\infty},\varphi\pr A}\}_{A\in\cal A}$
of simplicial sets. The following conditions are equivalent:
\begin{enumerate}
\item The set $S$ determines a simplicial functor $\cal A^{F}\to\cal B$
which extends $\varphi$.
\item The set $S$ determines a simplicial natural transformation $F\to\cal B\pr{B_{\infty},-}\circ\varphi$.
\end{enumerate}
\end{prop}

\begin{cor}
\label{cor:lan_pushout}Let $\cal A$ and $\cal B$ be small simplicial
categories and let $F:\cal A\to\SS$ and $\varphi:\cal A\to\cal B$
be simplicial functors. Let $\eta:F\to\pr{\varphi_{!}F}\circ\varphi$
be a simplicial natural transformation which exhibits $\varphi_{!}F$
as a left Kan extension of $F$ along $\varphi$, and by abusing notation
let $\eta:\cal A^{F}\to\cal B^{\varphi_{!}F}$ denote the induced
simplicial functor (Proposition \ref{prop:precosheaf_by_cat}). The
square 
\[\begin{tikzcd}
	{\mathcal{A}} & {\mathcal{A}^F} \\
	{\mathcal{B}} & {\mathcal{B}^{\varphi_!F}}
	\arrow["\varphi"', from=1-1, to=2-1]
	\arrow[from=1-1, to=1-2]
	\arrow["\eta", from=1-2, to=2-2]
	\arrow[from=2-1, to=2-2]
\end{tikzcd}\]of simplicial categories is a pushout.
\end{cor}

Now by \cite[Proposition 4.3]{DuggerNecklace}, we have $\widetilde{\fr C}[X^{\lcone}]=\widetilde{\fr C}[X]^{\Psi_{X}}$,
where the simplicial functor $\Psi_{X}:\widetilde{\fr C}[X]\to\SS$
is defined by $\Psi_{X}\pr x=\widetilde{\fr C}[X^{\lcone}]\pr{\infty,x}$.
Using Corollary \ref{cor:lan_pushout}, we deduce that the simplicial
natural transformation $\Psi_{X}\to\cst'_{\phi}\pr p\circ\phi\circ\widetilde{\fr C}[p]$
exhibits $\cst'_{\phi}\pr p$ as a left Kan extension of $\Psi_{X}$
along $\phi\circ\widetilde{\fr C}[p]$. Since $\Psi_{\Delta^{n}}=\Psi_{n}$,
we thus obtain, for each $n$-simplex $\sigma:\Delta^{n}\to S$ and
each simplicial functor $F\in\Fun^{s}\pr{\cal C_{\Delta},\SS}$, a
chain of bijections
\begin{align*}
\Fun^{s}\pr{\cal C_{\Delta},\SS}\pr{\cst'_{\phi}\pr{\sigma},F} & \cong\Fun^{s}\pr{\widetilde{\fr C}[\Delta^{n}],\SS}\pr{\Psi_{n},F\circ\phi\circ\widetilde{\fr C}[\sigma]}\\
 & \cong\SS_{/S}\pr{\sigma,\cun_{\phi}\pr F}.
\end{align*}
The resulting bijection 
\[
\theta_{\sigma}:\Fun^{s}\pr{\cal C_{\Delta},\SS}\pr{\cst'_{\phi}\pr{\sigma},F}\cong\SS_{/S}\pr{\sigma,\cun_{\phi}\pr F}
\]
is natural in the simplex $\sigma$, so there is a unique natural
bijection
\[
\Fun^{s}\pr{\cal C_{\Delta},\SS}\pr{\cst'_{\phi}\pr -,F}\cong\SS_{/S}\pr{-,\cun_{\phi}\pr F}
\]
of functors $\pr{\SS_{/S}}^{\op}\to\Set$ which extends the family
$\{\theta_{\sigma}\}_{\sigma:\Delta^{n}\to S}$. This natural bijection
is also natural in $F$. Thus we have proved that $\cun_{\phi}$ is
a right adjoint of $\cst'_{\phi}$. 

\subsubsection{\label{subsec:Marked-Case}Marked Case}

Let $p:\overline{X}=\pr{X,M}\to S^{\sharp}$ be a morphism of marked
simplicial sets. Lurie's model of the straightening of $p$ with respect
to $\phi$, which we temporarily denote by $\cst{}_{\phi}^{\p+}\pr p$,
is defined as follows: The composite
\[
\cal C_{\Delta}\xrightarrow{\cst_{\phi}^{\p+}\pr p}\SS^{+}\xrightarrow{\pr -_{\flat}}\SS
\]
is equal to the unmarked straightening $\cst_{\phi}\pr p$. For each
object $C\in\cal C$, the simplicial set $\cst_{\phi}\pr p\pr C=\pr{\widetilde{\fr C}[X^{\lcone}]\amalg_{\widetilde{\fr C}[X]}\cal C_{\Delta}}\pr{\infty,C}$
carries the marking which is minimal with respect to the following
requirements:
\begin{itemize}
\item $\cst_{\phi}^{\p+}\pr p$ is a simplicial functor from $\cal C_{\Delta}$
to $\SS^{+}$.
\item For each marked edge $f:x\to y$ of $\overline{X}$, the edge $\widetilde{f}$
of $\cst_{\phi}\pr p\pr{\phi\pr{p\pr y}}$ is marked, where $\widetilde{f}$
denotes the edge of $\cst_{\phi}\pr p\pr{\phi\pr{p\pr y}}$ determined
by the composite
\[
\widetilde{\fr C}[\pr{\Delta^{1}}^{\lcone}]\xrightarrow{\widetilde{\fr C}[f^{\lcone}]}\widetilde{\fr C}[X^{\lcone}]\to\widetilde{\fr C}[X^{\lcone}]\amalg_{\widetilde{\fr C}[X]}\cal C_{\Delta}.
\]
\end{itemize}
We claim that $\cst_{\phi}^{\p+}$ is a left adjoint of $\cun_{\phi}^{+}$,
so that we may take $\cst_{\phi}^{+}=\cst_{\phi}^{\p+}$. Define a
simplicial functor $\Psi_{\overline{X}}:\widetilde{\fr C}[X]\to\SS^{+}$
by giving $\Psi_{X}$ the minimal marking such that for each marked
edge $f:x\to y$ of $\overline{X}$, the edge of $\Psi_{X}\pr y$
determined by the map $\widetilde{\fr C}[\pr{\Delta^{1}}^{\lcone}]\xrightarrow{\widetilde{\fr C}[f^{\lcone}]}\widetilde{\fr C}[X^{\lcone}]$
is marked. (In other words, $\Psi_{\overline{X}}=\St_{\id_{\widetilde{\fr C}[X]}}^{\p+}\pr{\id_{\overline{X}}}$).
Using Proposition Corollary \ref{cor:lan_pushout}, we deduce that
the simplicial natural transformation $\Psi_{\overline{X}}\to\cst_{\phi}^{\p+}\pr p\circ\phi\circ\widetilde{\fr C}[p]$
exhibits $\cst_{\phi}^{\p+}\pr p$ as a left Kan extension of $\Psi_{\overline{X}}$.
We thus obtain, as in the unmarked case, a family of bijections $\{\Fun^{s}\pr{\cal C_{\Delta},\SS^{+}}\pr{\cst{}_{\phi}^{\p+}\pr{\sigma},F}\cong\SS_{/S}^{+}\pr{\sigma,\cun_{\phi}^{+}\pr F}\}_{\sigma}$,
where $\sigma$ ranges over all the maps of marked simplicial sets
of the form $\pr{\Delta^{n}}^{\flat}\to S^{\sharp}$ and $\pr{\Delta^{1}}^{\sharp}\to S^{\sharp}$.
This family extends to a natural bijection 
\[
\Fun^{s}\pr{\cal C_{\Delta},\SS^{+}}\pr{\cst_{\phi}^{\p+}\pr -,-}\cong\SS_{/S}^{+}\pr{-,\cun_{\phi}^{+}\pr -},
\]
so that $\cst_{\phi}^{\p+}$ is a left adjoint of $\cun_{\phi}^{+}$,
as claimed. 

\section{\label{sec:Rectification}Rectification of \texorpdfstring{$\mathfrak{P}$}{P}-Bundles}

Let $\frak P=\pr{M_{\cal D},\{p_{\alpha}\}_{\alpha\in A}}$ be a categorical
pattern on an $\infty$-category $\cal D$. In this section, we will
construct a categorical equivalence
\[
\Fun\pr{S,\frak P\-\Fib}\simeq\frak P\-\Bund\pr S
\]
for each (small) simplicial set $S$ (Corollary \ref{cor:main}).

Here is a sketch of our strategy. The equivalence will be realized
on the level of model categories, using the unstraightening functor.
Given a weak equivalence $\phi:\widetilde{\fr C}[S]\to\cal C_{\Delta}$
of simplicial categories, the unstraightening functor $\Fun^{s}\pr{\cal C_{\Delta},\SS^{+}}\to\SS_{/S}^{+}$
carries the constant diagram at $\overline{\cal D}=\pr{\cal D,M_{\cal D}}$
to the object $S^{\sharp}\times\cun_{\Delta^{0}}^{+}\pr{\overline{\cal D}}$.
Thus we obtain a functor
\[
\cun_{\phi}^{\frak P}:\Fun^{s}\pr{\cal C_{\Delta},\SS_{/\frak P}^{+}}\to\SS_{/S^{\sharp}\times\cun_{\Delta^{0}}^{+}\pr{\overline{\cal D}}}^{+}.
\]
With this in mind, we will construct a categorical pattern $\frak P_{\Un}$
on $\cun_{\Delta^{0}}\pr{\cal D}$ so that it has the following properties:
\begin{itemize}
\item The functor $\cun_{\phi}^{\frak P}$ is a right Quillen equivalence
with respect to $S\times\frak P_{\Un}$ (Theorem \ref{thm:unstraightening_P-fibered_obj}).
\item The comparison map $\cal D\to\cun_{\Delta^{0}}\pr{\cal D}$ induces
a right Quillen equivalence
\[
\SS_{/S\times\frak P_{\Un}}^{+}\to\SS_{/S\times\frak P}^{+}
\]
(Proposition \ref{prop:P_Un_good}). 
\end{itemize}
The desired Quillen equivalence will be obtained by composing these
two Quillen equivalences. The proof that these functors are right
Quillen equivalences will be carried out in two steps: We first prove,
in Subsection \ref{subsec:cat_inv_P-bundles}, a result which enables
us to reduce to the case where $S$ is an $\infty$-category. We then
prove the claim in the case where $S$ is an $\infty$-category.

In the case where $S$ is the nerve of an ordinary category, we can
also use relative nerve functor instead of the unstraightening functor.
We will follow this line of thought in Subsection \ref{subsec:rectification_1-cat}.

\subsection{\label{subsec:cat_inv_P-bundles}Weak Categorical Invariance of \texorpdfstring{$\mathfrak{P}$}{P}-Bundles}

The goal of this subsection is to prove the following weak categorical
invariance of $\frak P$-bundles.
\begin{prop}
\label{prop:weak_categorical_invariance}Let $\frak P$ be a commutative
categorical pattern on an $\infty$-category $\cal D$. Let $f:A\to B$
be a weak categorical equivalence of simplicial sets. The adjunction
\[
f_{!}:\SS_{/A\times\frak P}^{+}\adj\SS_{/B\times\frak P}^{+}:f^{*}
\]
is a Quillen equivalence.
\end{prop}

The proof of Proposition \ref{prop:weak_categorical_invariance} relies
on another proposition.

\begin{prop}
\label{prop:mla_are_weak_equiv}Let $\frak P=\pr{M_{S},T,\{p_{\alpha}\}_{\alpha\in A}}$
be a categorical pattern on a simplicial set $S$, and let $f:\overline{X}\to\overline{Y}$
be a morphism in $\SS_{/S}^{+}$. If the image of $f$ in $\SS^{+}$
is marked left anodyne and $T$ contains all $2$-simplices $\sigma$
such that $\sigma\vert\Delta^{\{0,1\}}\in M_{S}$, then $f$ is an
$\frak P$-equivalence.
\end{prop}

\begin{proof}
Let $Z_{\natural}\in\SS_{\frak P}^{+}$ be a fibrant object. We wish
to show that the map
\[
\Map_{S}^{\sharp}\pr{\overline{Y},Z_{\natural}}\to\Map_{S}^{\sharp}\pr{\overline{X},Z_{\natural}}
\]
is a trivial fibration. Let $A\to B$ be a monomorphism of simplicial
sets. In the following diagram, the lifting problem on the left is
equivalent to the one on the right:
\[\begin{tikzcd}
	A & {\operatorname{Map}_{S}^\sharp (\overline{Y},Z_\natural)} & {(A^\sharp\times \overline{Y})\amalg_{A^{\sharp}\times \overline{X}}(B^\sharp \times \overline{X})} & {Z_\natural} \\
	B & {\operatorname{Map}_{S}^\sharp (\overline{X},Z_\natural)} & {B^\sharp\times \overline{Y}} & {(S,M_S)}
	\arrow[from=1-1, to=1-2]
	\arrow[from=1-2, to=2-2]
	\arrow[from=1-1, to=2-1]
	\arrow[from=2-1, to=2-2]
	\arrow[from=1-3, to=2-3]
	\arrow[from=1-3, to=1-4]
	\arrow[from=1-4, to=2-4]
	\arrow[from=2-3, to=2-4]
	\arrow[dashed, from=2-3, to=1-4]
	\arrow[dashed, from=2-1, to=1-2]
\end{tikzcd}\]The right hand lifting problem admits a solution, for the map $\pr{A^{\sharp}\times\overline{Y}}\amalg_{A^{\sharp}\times\overline{X}}\pr{B^{\sharp}\times\overline{X}}\to B^{\sharp}\times\overline{Y}$
is marked left anodyne by \cite[Proposition 3.1.2.2]{HTT} and the
map $Z_{\natural}\to\pr{S,M_{S}}$ is a marked left fibration by Remark
\ref{rem:mlfib}. The proof is now complete.
\end{proof}
\begin{proof}
[Proof of Proposition \ref{prop:weak_categorical_invariance}]We will
write $M_{\cal D}$ for the set of edges of $\cal D$ specified by
the categorical pattern $\frak P$.

Using the small object argument, find a commutative diagram
\[\begin{tikzcd}
	A & B \\
	{\mathcal{A}} & {\mathcal{B},}
	\arrow["i"', from=1-1, to=2-1]
	\arrow["g"', from=2-1, to=2-2]
	\arrow["f", from=1-1, to=1-2]
	\arrow["j", from=1-2, to=2-2]
\end{tikzcd}\]where $\cal A$ and $\cal B$ are $\infty$-categories and $i$ and
$j$ are countable compositions of pushouts of coproducts of inclusions
of inner horns. The functor $g$ is a categorical equivalence, so
Proposition \ref{prop:cat_inv} shows that the adjunction $g_{!}\dashv g^{*}$
is a Quillen equivalence. It will therefore suffice to show that the
adjunctions $i_{!}\dashv i^{*}$ and $j_{!}\dashv j^{*}$ are Quillen
equivalences. Thus we are reduced to the case where $B$ is an $\infty$-category
and $f$ is a countable composition of pushouts of coproducts of inclusions
of inner horns. 

Since inner anodyne extensions induce bijections between the set of
vertices, Proposition \ref{prop:fiberwise_P_equiv} shows that the
total right derived functor of $f^{*}$ is conservative. It will therefore
suffice to show that the derived unit of the adjunction $f_{!}\dashv f^{*}$
is an isomorphism.

Let $p:X_{\natural}\to A^{\sharp}\times\overline{\cal D}$ be a fibrant
object of $\SS_{/A\times\frak P}^{+}$. Let $\overline{X}$ denote
the marked simplicial set obtained from $X$ by marking the cocartesian
edges over $A$; equivalently, $\overline{X}=X_{\natural}\times_{\overline{\cal D}}\cal D^{\natural}$
by Lemma \ref{lem:cc_P-bundle}. Find a commutative diagram 
\[\begin{tikzcd}
	{\overline{X}} & {\overline{Y}} \\
	{A^\sharp\times \mathcal{D}^\natural} & {B^\sharp\times \mathcal{D}^\natural}
	\arrow["{f\times \operatorname{id}_{\mathcal{D}}}"', from=2-1, to=2-2]
	\arrow["q", from=1-2, to=2-2]
	\arrow["p"', from=1-1, to=2-1]
	\arrow["i", from=1-1, to=1-2]
\end{tikzcd}\]where $i$ is marked left anodyne, $q$ is a marked left fibration,
and $\cal D^{\natural}$ denotes the marked simplicial set obtained
from $\cal D$ by marking all equivalences. We will prove that the
map $q$ has the following properties:
\begin{enumerate}
\item For each vertex $a\in A$, the induced map $X_{a}\to Y_{f\pr a}$
of fibers is a categorical equivalence.
\item The map $q:Y\to B\times\cal D$ is a $\frak P$-bundle over $B$.
\item The map $i$ induces a map $X_{\natural}\to Y_{\natural}$ of marked
simplicial sets.
\item The map $X_{\natural}\to Y_{\natural}$ is a $B\times\frak P$-equivalence.
\end{enumerate}
It will then follow from Proposition \ref{prop:fiberwise_P_equiv}
that the derived unit of the adjunction $f_{!}\dashv f^{*}$ is an
isomorphism.

We begin with the verification of (1). Since the adjunction $\SS_{/A}^{+}\adj\SS_{/B}^{+}$
is a Quillen equivalence (\cite[Proposition 3.3.1.1]{HTT}), we deduce
that the derived unit $\overline{X}\to A^{\sharp}\times_{B^{\sharp}}\overline{Y}$
is a cocartesian equivalence over $A$. In particular, for each vertex
$a\in A$, the map $X_{a}\to Y_{f\pr a}$ is a categorical equivalence,
as claimed.

Next we check condition (2). We must check that $q$ satisfies conditions
(a) through (d) of Definition \ref{def:P-bundle}. 

\begin{enumerate}[label=(\alph*)]

\item The map $Y\to B$ is a cocartesian fibration. This follows
from the construction.

\item The map $q$ lifts to a map of fibrant objects of $\SS_{/B}^{+}$.
This follows from the construction.

\item For each object $b\in B$, the map $Y_{b}\to\cal D$ is $\frak P$-fibered.
Since $f$ is inner anodyne, it is bijective on vertices. Therefore,
there is some vertex $a\in A$ such that $f\pr a=b$. According to
claim (1), the map $X_{a}\to Y_{f\pr a}$ is a categorical equivalence.
Moreover, the map $Y_{b}\to\cal D$ is a categorical fibration by
Proposition \ref{prop:creative}. Therefore, Lemma \ref{lem:cat_inv}
show that the map $Y_{b}\to\cal D$ is $\frak P$-fibered.

\item For each morphism $\beta:b\to b'$ of $B$, the induced functor
$\alpha_{!}:Y_{b}\to Y_{b'}$ preserves cocartesian edges over the
edges in $M_{\cal D}$. By the reduction we made in the first paragraph,
the morphism $\alpha$ can be written as a finite composition of morphisms
in the image of $f$. It will therefore suffice to consider the case
where $\beta=f\pr{\alpha}$ for some edge $\alpha:a\to a'$ of $A$.
Consider the diagram 
\[\begin{tikzcd}
	{X_a} & {Y_{f(a)}} \\
	{X_{a'}} & {Y_{f(a')}}
	\arrow["i", from=1-1, to=1-2]
	\arrow["{\beta_!}", from=1-2, to=2-2]
	\arrow["i"', from=2-1, to=2-2]
	\arrow["{\alpha_!}"', from=1-1, to=2-1]
\end{tikzcd}\]of $\infty$-categories, which commutes up to natural equivalence
over $\cal D$. The horizontal arrows are categorical equivalences
by claim (1), so it suffices to show that the functor $\alpha_{!}$
preserves cocartesian edges lying over the edges in $M_{\cal D}$.
This follows from the fact that $X$ is a $\frak P$-bundle.

\end{enumerate}

Next we verify claim (3). Let $\gamma:x\to x''$ be a marked edge
of $X_{\natural}$. We wish to show that the edge $f\pr{\gamma}$
is $q$-cocartesian. Let $p\pr{\gamma}=\pr{\alpha,\delta}:\pr{a,D}\to\pr{a'',D''}$
denote the image of $\gamma$. By Remark \ref{rem:mlfib}, the map
$X_{\natural}\to A^{\sharp}\times\overline{\cal D}$ is a marked left
fibration. Therefore, we can find a $p$-cocartesian edge $\gamma':x\to x'$
lying over $\pr{\alpha,\id_{D}}$. Since $\gamma'$ is $p$-cocartesian,
there is a $2$-simplex 
\[\begin{tikzcd}
	& {x'} \\
	x && {x''}
	\arrow["{\gamma'}", from=2-1, to=1-2]
	\arrow["{\gamma''}", from=1-2, to=2-3]
	\arrow["\gamma"', from=2-1, to=2-3]
\end{tikzcd}\]of $X$, where $\gamma''$ lies over the edge $\pr{\id_{a''},\delta}$.
Note that the edge $\gamma''$ is $p$-cocartesian because $\gamma$
and $\gamma'$ are $p$-cocartesian. Now the edge $\gamma'$ is marked
in $\overline{X}$, so its image $i\pr{\gamma'}$ is marked in $\overline{Y}$
and hence is $q$-cocartesian. Also, the edge $\gamma''$ is $p_{a''}$-cocartesian,
where $p_{a''}:X_{a''}\to\cal D$ denotes the pullback of $p$, so
claim (1) implies that its image in $Y_{a''}$ is $q_{f\pr{a''}}$-cocartesian.
It follows from (2) and Lemma \ref{lem:fiberwise_colimit} that the
edge $i\pr{\gamma''}$ is also $q$-cocartesian. In conclusion, the
edges $i\pr{\gamma'}$ and $i\pr{\gamma''}$ are $q$-cocartesian.
Hence $i\pr{\gamma}$ is $q$-cocartesian, as claimed.

We complete the proof by verifying condition (4). Since the model
structure on $\SS_{/B\times\frak P}^{+}$ is left proper, Proposition
\ref{prop:mla_are_weak_equiv} shows that the map $X_{\natural}\to X_{\natural}\amalg_{\overline{X}}\overline{Y}$
is a $B\times\frak P$-equivalence. Therefore, it suffices to show
that the inclusion $X_{\natural}\amalg_{\overline{X}}\overline{Y}\to Y_{\natural}$
is a $B\times\frak P$-equivalence. For each $n\geq0$, let $M\pr n$
denote the set of morphisms $\beta$ of $B$ such that $\beta$ is
either an equivalence of $B$ or a composition of at most $\pr{n+1}$
morphisms in $A$. Set $\overline{Y}\pr n=Y_{\natural}\times_{B^{\sharp}}\pr{B,M\pr n}$
for $n\geq0$, and set $\overline{Y}\pr{-1}=X_{\natural}\amalg_{\overline{X}}\overline{Y}$.
We then have a nested sequence
\[
X_{\natural}\amalg_{\overline{X}}\overline{Y}=\overline{Y}\pr{-1}\subset\overline{Y}\pr 0\subset\overline{Y}\pr 1\subset\cdots\subset Y_{\natural}
\]
with $Y_{\natural}=\bigcup_{n\geq-1}\overline{Y}\pr n$. To complete
the proof, it suffices to show that the inclusion $\overline{Y}\pr{n-1}\subset\overline{Y}\pr n$
is a $B\times\frak P$-equivalence for every $n\geq1$. But this inclusion
is a pushout of coproducts of the inclusions of the form $\pr{\Lambda_{1}^{2}}^{\sharp}\cup\pr{\Delta^{2}}^{\flat}\subset\pr{\Delta^{2}}^{\sharp}$,
so the claim follows from \cite[Example B.2.2]{HA}.
\end{proof}

\subsection{\label{subsec:P_Un}The Categorical Pattern \texorpdfstring{$\mathfrak{P}_{\operatorname{Un}}$}{P}}

In this subsection, we construct an auxiliary categorical pattern
$\frak P_{\Un}$ from a given commutative categorical pattern $\frak P$,
and show that it behaves well with respect to considerations of bundles
(Proposition \ref{prop:P_Un_good}).
\begin{defn}
Let $\frak P=\pr{M_{\cal D},\{p_{\alpha}\}_{\alpha\in A}}$ be a commutative
categorical pattern on an $\infty$-category $\cal D$. We let $\frak P_{\Un}$
denote the smallest commutative categorical pattern on $\cun_{\Delta^{0}}\pr{\cal D}$
which contains the image of $\frak P$ under the comparison map $f:\cal D\to\cun_{\Delta^{0}}\pr{\cal D}$
of Corollary \ref{cor:3.2.1.14}. Explicitly, we have
\[
\frak P_{\Un}=\pr{f\pr{M_{\cal D}},\{fp_{\alpha}\}_{\alpha\in A}}.
\]
\end{defn}

\begin{defn}
Let $\frak P$ be a commutative categorical pattern on an $\infty$-category
$\cal D$. Given a simplicial set $S$ and a simplicial functor $\phi:\widetilde{\fr C}[S]\to\cal C_{\Delta}$,
the functor 
\[
\cun_{\phi}^{+}:\Fun^{s}\pr{\cal C_{\Delta},\SS^{+}}\to\SS_{/S}^{+}
\]
carries the constant diagram $\delta\pr{\overline{\cal D}}$ at $\overline{\cal D}$
to the object $S^{\sharp}\times\cun_{\Delta^{0}}^{+}\pr{\overline{\cal D}}$.
We let $\cun_{\phi}^{\frak P}$ denote the induced functor
\[
\Fun^{s}\pr{\cal C_{\Delta},\SS_{/\frak P}^{+}}\to\SS_{/S\times\frak P_{\Un}}^{+}.
\]
The left adjoint of $\cun_{\phi}^{\frak P}$ will be denoted by $\cst_{\phi}^{\frak P}$.
We will write $\cun_{\cal C}^{\frak P}=\cun_{\id_{\widetilde{\frak C}[\cal C]}}^{\frak P}$
and $\cst_{\cal C}^{\frak P}=\cst_{\id_{\widetilde{\frak C}[\cal C]}}^{\frak P}$.
\end{defn}

Observe that the functor $\cun_{\phi}^{\frak P}$ takes values in
$\SS_{/S\times\frak P_{\Un}}^{+}$, not in $\SS_{/S\times\frak P}^{+}$.
But this is not a problem for us, because of the following proposition:
\begin{prop}
\label{prop:P_Un_good}Let $\frak P$ be a commutative categorical
pattern on an $\infty$-category $\cal D$ and let $S$ be a simplicial
set. The functor
\[
\SS_{/S\times\frak P_{\Un}}^{+}\to\SS_{/S\times\frak P}^{+}
\]
is a right Quillen equivalence.
\end{prop}

\begin{proof}
By Proposition \ref{prop:weak_categorical_invariance}, we may assume
that $S$ is an $\infty$-category. In this case, the claim follows
from Proposition \ref{prop:cat_inv}.
\end{proof}

\subsection{\label{subsec:main}Main Result}

In this subsection, we show that $\cun_{\phi}^{\frak P}$ is a right
Quillen equivalence (Theorem \ref{thm:unstraightening_P-fibered_obj}),
and then use it to prove the main result of this paper (Corollary
\ref{cor:main}). 

We begin by showing that $\cun_{\phi}^{\frak P}$ is right Quillen.
\begin{prop}
\label{prop:Un_=00005Cphi^=00007B=00005CfrakP=00007D_right_Quillen}Let
$S$ be a simplicial set and let $\phi:\widetilde{\fr C}[S]\to\cal C_{\Delta}$
be a simplicial functor. Let $\frak P$ be a commutative categorical
pattern on an $\infty$-category $\cal D$. The functor
\[
\cun_{\phi}^{\frak P}:\Fun^{s}\pr{\cal C_{\Delta},\SS_{/\frak P}^{+}}_{{\rm proj}}\to\SS_{/S\times\frak P_{\Un}}^{+}
\]
is right Quillen.
\end{prop}

\begin{proof}
First we make a reduction to the case where $S$ is an $\infty$-category
and $\phi$ is the identity simplicial functor. Suppose that we have
proved that, for each $\infty$-category $\cal C$, the functor $\cun_{\cal C}^{\frak P}$
is right Quillen. We then complete the proof as follows. The functor
$\cun_{\phi}^{\frak P}$ factors as
\[
\Fun^{s}\pr{\cal C_{\Delta},\SS_{/\frak P}^{+}}\xrightarrow{\phi^{*}}\Fun^{s}\pr{\widetilde{\fr C}[S],\SS_{/\frak P}^{+}}\xrightarrow{\cun_{S}^{\frak P}}\SS_{/S\times\frak P_{\Un}}^{+}.
\]
Since $\phi^{*}$ is right Quillen, it suffices to show that the functor
$\cun_{S}^{\frak P}$ is right Quillen. We will show that $\cun_{S}^{\frak P}$
preserves trivial fibrations and that $\cst_{S}^{\frak P}$ preserves
weak equivalences. The functor $\cun_{S}^{\frak P}$ preserves trivial
fibrations because the forgetful functors $\Fun^{s}\pr{\widetilde{\fr C}[S],\SS_{/\frak P}^{+}}\to\Fun^{s}\pr{\widetilde{\fr C}[S],\SS^{+}}$
and $\SS_{/S\times\frak P_{\Un}}^{+}\to\SS_{/S}^{+}$ preserve and
reflect trivial fibrations, and because the functor $\cun_{\phi}^{+}$
is right Quillen. To show that $\cst_{S}^{\frak P}$ preserves weak
equivalences, choose a weak categorical equivalence $i:S\to\cal C$,
where $\cal C$ is an $\infty$-category. The diagram 
\[\begin{tikzcd}
	{\mathsf{sSet}^+_{/S\times\mathfrak{P}}} & {\operatorname{Fun}^{s}(\widetilde{\mathfrak{C}}[S],\mathsf{sSet}^+_{/\mathfrak{P}})} \\
	{\mathsf{sSet}^+_{/\mathcal{C}\times\mathfrak{P}}} & {\operatorname{Fun}^{s}(\widetilde{\mathfrak{C}}[\mathcal{C}],\mathsf{sSet}^+_{/\mathfrak{P}})}
	\arrow["{\widetilde{\operatorname{St}}_{\mathcal{C}}^{\mathfrak{P}}}"', from=2-1, to=2-2]
	\arrow["{i_!}"', from=1-1, to=2-1]
	\arrow["{\widetilde{\operatorname{St}}_{\widetilde{\mathfrak{C}}[i]}^{\mathfrak{P}}}"{description}, from=1-1, to=2-2]
	\arrow["{\widetilde{\mathfrak{C}}[i]_!}", from=1-2, to=2-2]
	\arrow["{\widetilde{\operatorname{St}}_{S}^{\mathfrak{P}}}", from=1-1, to=1-2]
\end{tikzcd}\]commutes up to natural isomorphism. Since the functor $\widetilde{\fr C}[i]_{!}$
is a left Quillen equivalence \cite[Proposition A.3.3.8]{HTT}, it
reflects weak equivalences of cofibrant objects. Therefore, it suffices
to show that the functor $\cst_{\widetilde{\fr C}[i]}^{\frak P}$
preserves weak equivalences. This follows from our assumption that
the functor $\cst_{\cal C}^{\frak P}$ is left Quillen and the fact
that the functor $i_{!}$ is also left Quillen.

We are therefore reduced to the case where $S=\cal C$ is an $\infty$-category
and $\phi$ is the identity simplicial functor. For the remainder
of the proof, we will write $\pi:\SS_{/\fr P}^{+}\to\SS^{+}$ for
the forgetful functor.

By \cite[Proposition 7.15]{JT07}, it suffices to show that the functor
$\cun_{\phi}^{\frak P}$ preserves fibrations between fibrant objects
and that $\cst_{\phi}^{\frak P}$ preserves cofibrations. As in the
first paragraph, the functor $\cun_{\phi}^{\frak P}$ preserves trivial
fibrations, so the functor $\cst_{\phi}^{\frak P}$ preserves cofibrations.
It will therefore suffice to show that the functor $\cun_{\phi}^{\frak P}$
preserves fibrations between fibrant objects.

To prove that the functor $\cun_{\phi}^{\frak P}$ preserves fibrations
between fibrant objects, it will suffice to show that it preserves
fibrant objects. Indeed, suppose we have shown that $\cun_{\phi}^{\frak P}$
preserves fibrant objects. Let $\alpha:F\to G$ be a projective fibration
between projectively fibrant functors $F,G\in\Fun^{s}\pr{\widetilde{\fr C}[\cal C],\SS_{/\frak P}^{+}}$.
Since the categorical pattern $\cal C\times\frak P_{\Un}$ is creative
(in fact, commutative) and $\cun_{\phi}^{\frak P}$ preserves fibrant
objects, Proposition \ref{prop:creative} shows that $\cun_{\phi}^{\frak P}\pr{\alpha}$
is a fibration if and only if the map $\cun_{\phi}\pr{\pr{\pi\circ F}_{\flat}}\to\cun_{\phi}\pr{\pr{\pi\circ G}_{\flat}}$
is a categorical fibration. But this is the underlying map of the
map $\cun_{\phi}^{+}\pr{\pr{\pr{\pi\circ F}_{\flat}}^{\natural}}\to\cun_{\phi}^{+}\pr{\pr{\pr{\pi\circ G}_{\flat}}^{\natural}}$,
so the claim follows from Proposition \ref{prop:creative} and the
fact that the functor $\cun_{\phi}^{+}$ is right Quillen.

We are thus reduced to showing that $\cun_{\phi}^{\frak P}$ preserves
fibrant objects. Let $F:\cal C_{\Delta}\to\SS_{/\frak P}^{+}$ be
a projectively fibrant functor. We must show that the object $\cun_{\phi}^{\frak P}\pr F\in\SS_{/\cal C\times\frak P_{\Un}}^{+}$
is fibrant. Set $F'=\pi\circ F$. According to Proposition \ref{prop:recognition_of_fibrant_objects},
it will suffice to show that the map $p:\cun_{\phi}\pr{F'_{\flat}}\to\cal C\times\cun_{\Delta^{0}}\pr{\cal D}$
is a $\frak P_{\Un}$-bundle over $\cal C$ and that the marked edges
of $\cun_{\phi}^{+}\pr{F'}$ are precisely the $p$-cocartesian morphisms
whose images in $\cun_{\Delta^{0}}\pr{\cal D}$ are marked. 

We first verify that $p$ is a $\frak P_{\Un}$-bundle over $\cal C$.
We will check conditions (a) through (d) of Definition \ref{def:P-bundle}.

\begin{enumerate}[label=(\alph*)]

\item The map $q:\cun_{\phi}\pr{F'_{\flat}}\to\cal C$ is a cocartesian
fibration. This follows from the fact that $q$ is the underlying
map of the fibrant object $\cun_{\phi}^{+}\pr{\pr{F'_{\flat}}^{\natural}}\in\SS_{/\cal C}^{+}$.

\item The map $p$ lifts to a fibration in $\SS_{/\cal C}^{+}$ equipped
with the cocartesian model structure. Since $p$ is the underlying
map of the map $\cun_{\phi}^{+}\pr{\pr{F_{\flat}^{\p}}^{\natural}}\to\cun_{\phi}^{+}\pr{\delta\pr{\cal D^{\natural}}}$,
the claim is a consequence of Proposition \ref{prop:creative} and
the fact that the functor $\cun_{\phi}^{+}$ is right Quillen.

\item For each object $C\in\cal C$, the map
\[
p_{C}:\cun_{\phi}\pr{F'_{\flat}}\times_{\cal C}\{C\}\to\cun_{\Delta^{0}}\pr{\cal D}
\]
is $\frak P_{\Un}$-fibered. According to Corollary \ref{cor:3.2.1.14},
there is a commutative diagram 
\[\begin{tikzcd}
	{F'_\flat(C)} & {\widetilde{\operatorname{Un}}_{\Delta^0} (F'_\flat(C))} \\
	{\mathcal{D}} & {\widetilde{\operatorname{Un}}_{\Delta^0} (\mathcal{D})}
	\arrow["\simeq", from=2-1, to=2-2]
	\arrow[from=1-2, to=2-2]
	\arrow["\simeq", from=1-1, to=1-2]
	\arrow[from=1-1, to=2-1]
\end{tikzcd}\]whose horizontal arrows are categorical equivalences. The vertical
maps are categorical fibrations by Proposition \ref{prop:creative}.
It follows from Lemma \ref{lem:cat_inv} and Corollary \ref{cor:3.2.1.14}
that $p_{C}$ is $\frak P_{\Un}$-fibered.

\item Let $f:C\to C'$ be a morphism of $\cal C$. We must show that
the induced functor 
\[
\cun_{\phi}\pr{F'_{\flat}}\times_{\cal C}\{C\}\to\cun_{\phi}\pr{F'_{\flat}}\times_{\cal C}\{C'\}
\]
preserves cocartesian edges over the marked edges of $\cun_{\Delta^{0}}^{+}\pr{\overline{\cal D}}$.
According to Proposition \ref{prop:rectification}, there is a commutative
diagram 
\[\begin{tikzcd}
	& {F'_\flat(C')} && {\widetilde{\operatorname{Un}}_{\Delta^0}(F'_\flat(C'))} \\
	{F'_\flat(C)} && {\widetilde{\operatorname{Un}}_{\Delta^0}(F'_\flat(C))} \\
	& {\mathcal{D}} && {\widetilde{\operatorname{Un}}_{\Delta^0}(\mathcal{D})}
	\arrow[from=2-1, to=1-2]
	\arrow[from=1-2, to=3-2]
	\arrow[from=2-1, to=3-2]
	\arrow["\simeq"{pos=0.7}, from=2-1, to=2-3]
	\arrow["\simeq", from=1-2, to=1-4]
	\arrow[from=2-3, to=1-4]
	\arrow[from=2-3, to=3-4]
	\arrow["\simeq", from=3-2, to=3-4]
	\arrow[from=1-4, to=3-4]
\end{tikzcd}\]whose horizontal arrows are categorical equivalences. Since $F'_{\flat}\pr f$
preserves cocartesian edges over the marked edges of $\cal D$, we
are done.

\end{enumerate}

Next, we check that the marked edges of $\cun_{\phi}^{+}\pr F$ are
precisely the $p$-cocartesian morphisms whose image in $\cun_{\Delta^{0}}^{+}\pr{\overline{\cal D}}$
is marked. Take an arbitrary morphism $\pr{f,g}:\pr{C,X}\to\pr{C',X'}$
of $\cun_{\phi}\pr{F'_{\flat}}$ whose image in $\cun_{\Delta^{0}}^{+}\pr{\cal D}$
is marked. We wish to show that $\pr{f,g}$ is marked in $\cun_{\phi}^{+}\pr{F'}$
if and only if it is $p$-cocartesian. By the definition of marked
edges of $\cun_{\phi}^{+}\pr{F'}$, the following conditions are equivalent:
\begin{itemize}
\item [(1)]The morphism $\pr{f,g}$ is marked in $\cun_{\phi}^{+}\pr{F'}$.
\item [(2)]The morphism $g:\pr{F'_{\flat}\pr f}X\to X'$ is marked in $F'\pr{C'}$.
\end{itemize}
Since the map $F'_{\flat}\pr{C'}\to\cal D$ is $\frak P$-fibered,
condition (2) is equivalent to the following condition:
\begin{itemize}
\item [(3)]The morphism $g$ is cocartesian with respect to the projection
$F'_{\flat}\pr{C'}\to\cal D$.
\end{itemize}
Using Corollary \ref{cor:3.2.1.14}, we see that condition (3) is
equivalent to the following condition:
\begin{itemize}
\item [(4)]The morphism $g$ is cocartesian with respect to the projection
$\cun_{\Delta^{0}}\pr{F'_{\flat}\pr{C'}}\to\cun_{\Delta^{0}}\pr{\cal D}$.
\end{itemize}
By Lemma \ref{lem:fiberwise_colimit}, this is equivalent to the following
condition:
\begin{itemize}
\item [(5)]The morphism $\pr{\id_{C'},g}:\pr{C',\pr{F'_{\flat}\pr f}X}\to\pr{C',X'}$
is $p$-cocartesian.
\end{itemize}
Now the morphism $\pr{f,g}$ can be factored as 
\[
\pr{f,g}=\pr{\id_{C'},g}\circ\pr{f,\id_{\pr{F'_{\flat}\pr f}X}}.
\]
The morphism $\pr{f,\id_{\pr{F'_{\flat}\pr f}X}}$ is $p$-cocartesian
because it is cocartesian over $\cal C$ and its image in $\cal C\times\cun_{\Delta^{0}}\pr{\cal D}$
is also cocartesian over $\cal C$. Therefore, condition (5) holds
if and only if the morphism $\pr{f,g}$ is $p$-cocartesian, as required.
\end{proof}
We next show that $\cun_{\phi}^{\frak P}$ is a right Quillen equivalence.
\begin{thm}
\label{thm:unstraightening_P-fibered_obj}Let $S$ be a simplicial
set and let $\phi:\widetilde{\fr C}[S]\to\cal C_{\Delta}$ be a weak
equivalence of simplicial categories. Let $\frak P$ be a commutative
categorical pattern on an $\infty$-category $\cal D$. The functor
\[
\cun_{\phi}^{\frak P}:\Fun^{s}\pr{\cal C_{\Delta},\SS_{/\frak P}^{+}}\to\SS_{/S\times\frak P_{\Un}}^{+}
\]
is a right Quillen equivalence.
\end{thm}

The proof of Theorem \ref{thm:unstraightening_P-fibered_obj} requires
some preliminaries.
\begin{lem}
\label{lem:weak_+_equiv->P-equiv}Let $\frak P$ be a commutative
categorical pattern on an $\infty$-category $\cal D$. Let $f:\overline{X}\to\overline{Y}$
be a morphism in $\SS_{/\frak P}^{+}$. If $f$ is a weak equivalence
of marked simplicial sets and the marked simplicial set $\overline{Y}\in\SS^{+}$
is fibrant, then $f$ is a $\frak P$-equivalence.
\end{lem}

\begin{proof}
Factor the map $f$ as $\overline{X}\xrightarrow{i}\overline{X'}\xrightarrow{p}\overline{Y}$,
where $i$ is marked left anodyne and $p$ is a marked left fibration.
Since $\overline{Y}\in\SS^{+}$ is fibrant, the map $p$ is a fibration
of $\SS^{+}$. Since $f$ and $i$ are weak equivalences of marked
simplicial sets, the map $p$ is a trivial fibration of marked simplicial
sets. Therefore, the map $p$ is a $\frak P$-equivalence. We also
know from Lemma \ref{prop:mla_are_weak_equiv} that the map $i$ is
a $\frak P$-equivalence. Being the composite of $\frak P$-equivalences,
the map $f$ is also a $\frak P$-equivalence, and the proof is complete.
\end{proof}
\begin{lem}
\label{lem:St_Un_counit}Let $S$ be a simplicial set and let $\phi:\widetilde{\fr C}[S]\to\cal C_{\Delta}$
be a simplicial functor. For each simplicial functor $F:\cal C_{\Delta}\to\SS^{+}$
and each vertex $x\in S$, the map
\[
\varepsilon:\cst_{\phi}^{+}\cun_{\phi}^{+}\pr F\pr{\phi\pr x}\to F\pr{\phi\pr x}
\]
induces a surjection on the set of marked edges.
\end{lem}

\begin{proof}
Let $f:\pr{\Delta^{1}}^{\sharp}\to F\pr{\phi\pr x}$ be a marked edge.
We must show that $f$ factors through $\varepsilon$. Regard $\pr{\Delta^{1}}^{\sharp}$
as an object of $\SS_{/S}^{+}$ with respect to the projection $\pr{\Delta^{1}}^{\sharp}\to\{x\}^{\sharp}\hookrightarrow S^{\sharp}$,
and consider the simplicial natural transformation $\alpha:\cst_{\phi}^{+}\pr{\pr{\Delta^{1}}^{\sharp}}\to F$
which is adjoint to the composite
\[
\pr{\Delta^{1}}^{\sharp}\to F\pr{\phi\pr x}\to\cun_{\Delta^{0}}^{+}\pr{F\pr{\phi\pr x}}\cong\cun_{\phi}^{+}\pr F\times_{S}\{x\}\hookrightarrow\cun_{\phi}^{+}\pr F.
\]
According to the description of the adjunction $\cst_{\phi}^{+}\dashv\cun_{\phi}^{+}$
given in subsection \ref{subsec:comparison_with_Lurie}, the composite
\[
\pr{\Delta^{1}}^{\sharp}\cong\Psi_{\pr{\Delta^{1}}^{\sharp}}\pr 1\to\St_{\phi}^{+}\pr{\pr{\Delta^{1}}^{\sharp}}\pr{\phi\pr x}\xrightarrow{\alpha}F\pr{\phi\pr x}
\]
classifies the edge $f$. Hence $f$ is the image of a marked edge
of $\St_{\phi}^{+}\pr{\pr{\Delta^{1}}^{\sharp}}\pr{\phi\pr x}$ under
$\alpha$. Since $\alpha$ factors through the counit map $\cst_{\phi}^{+}\cun_{\phi}^{+}\pr F\to F$,
we are done.
\end{proof}
We now arrive at the proof of Theorem \ref{thm:unstraightening_P-fibered_obj}.
\begin{proof}
[Proof of Theorem \ref{thm:unstraightening_P-fibered_obj}]Notice
first that the validity of the statement of the theorem depends only
on the simplicial set $S$ and not on the simplicial functor $\phi$.
Indeed, the functor $\cun_{\phi}^{\frak P}$ factors as
\[
\Fun^{s}\pr{\cal C_{\Delta},\SS_{/\frak P}^{+}}_{{\rm proj}}\xrightarrow{\phi^{*}}\Fun^{s}\pr{\widetilde{\fr C}[S],\SS_{/\frak P}^{+}}\xrightarrow{\cun_{S}^{\frak P}}\SS_{/S\times\frak P_{\Un}}^{+},
\]
and by \cite[Proposition A.3.3.8]{HTT}, the functor $\phi^{*}$ is
a right Quillen equivalence. Therefore, $\cun_{\phi}^{\frak P}$ is
a right equivalence if and only if $\cun_{S}^{\frak P}$ is a right
Quillen equivalence. The latter functor clearly does not depend on
$\phi$.

We next remark that the validity of the theorem depends only on the
weak categorical type of $S$. Indeed, for each map of simplicial
sets $f:S\to T$, there is a diagram 
\[\begin{tikzcd}
	{\operatorname{Fun}^{s}(\widetilde{\mathfrak{C}}[T],\mathsf{sSet}^+_{/\mathfrak{P}})} & {\mathsf{sSet}^+_{/T\times\mathfrak{P}}} \\
	{\operatorname{Fun}^{s}(\widetilde{\mathfrak{C}}[S],\mathsf{sSet}^+_{/\mathfrak{P}})} & {\mathsf{sSet}^+_{/S\times\mathfrak{P}}}
	\arrow["{\widetilde{\operatorname{Un}}_{T}^{\mathfrak{P}}}", from=1-1, to=1-2]
	\arrow["{f^*}", from=1-2, to=2-2]
	\arrow["{\widetilde{\mathfrak{C}}[f]^*}"', from=1-1, to=2-1]
	\arrow["{\widetilde{\operatorname{Un}}_{S}^{\mathfrak{P}}}"', from=2-1, to=2-2]
\end{tikzcd}\]of right Quillen functors, which commutes up to natural isomorphisms.
If $f$ is a weak categorical equivalence, then the functor $f^{*}$
is a right Quillen equivalence by Proposition \ref{prop:weak_categorical_invariance},
and the functor $\widetilde{\fr C}[f]^{*}$ is a right Quillen equivalence
by \cite[Proposition A.3.3.8]{HTT}. Thus $\cun_{S}^{\frak P}$ is
a right Quillen equivalence if and only if $\cun_{T}^{\frak P}$ is
a right Quillen equivalence.

Combining the observations in the previous paragraphs, we may assume
that $\cal C_{\Delta}$ is fibrant in the Bergner model structure
(i.e., the hom-simplicial sets of $\cal C_{\Delta}$ are Kan complexes),
that $S=N\pr{\cal C_{\Delta}}$ is its homotopy coherent nerve, and
that $\phi$ is given by the counit map $\widetilde{\fr C}[N\pr{\cal C_{\Delta}}]\to\cal C_{\Delta}$.
For the remainder of the proof, we let $\pi:\SS_{/\fr P}^{+}\to\SS^{+}$
denote the forgetful functor. It suffices to verify that the total
right derived functor $\bb R\cun_{\phi}^{\frak P}$ of $\cun_{\phi}^{\frak P}$
has the following properties:
\begin{enumerate}
\item The functor $\bb R\cun_{\phi}^{\frak P}$ is essentially surjective.
\item The functor $\bb R\cun_{\phi}^{\frak P}$ is fully faithful.
\end{enumerate}
We first prove (1). Let $p:\cal X\to S\times\cun_{\Delta^{0}}\pr{\cal D}$
be a $\frak P_{\Un}$-bundle over $S$. We wish to find a projectively
fibrant simplicial functor $F:\cal C_{\Delta}\to\SS_{/\frak P}^{+}$
and a weak equivalence $\cal X_{\natural}\to\cun_{\phi}^{\frak P}\pr F$
of $\SS_{/S\times\frak P_{\Un}}^{+}$. By Proposition \ref{prop:creative},
it suffices to find a projectively fibrant simplicial functor $F:\cal C_{\Delta}\to\SS_{/\frak P}^{+}$
and a categorical equivalence $\cal X\to\cun_{\phi}\pr{\pr{\pi\circ F}_{\flat}}$
over $S\times\cun_{\Delta^{0}}\pr{\cal D}$.

Let $\overline{\cal X}$ denote the marked simplicial set obtained
from $\cal X$ by marking the cocartesian morphisms over $S$. Let
$p':\cst_{\phi}^{+}\pr{\overline{\cal X}}\to\delta\pr{\cal D^{\natural}}$
denote the adjoint of the functor $p$. We factor the map $p'$ as
\[\begin{tikzcd}
	{\widetilde{\operatorname{St}}^+_\phi(\overline{\mathcal{X}})} && G \\
	& {\delta(\mathcal{D}^\natural),}
	\arrow["{i'}", from=1-1, to=1-3]
	\arrow["{q'}", from=1-3, to=2-2]
	\arrow["{p'}"', from=1-1, to=2-2]
\end{tikzcd}\]where $i'$ is a trivial cofibration and $q'$ is a fibration in $\Fun^{s}\pr{\cal C_{\Delta},\SS^{+}}_{{\rm proj}}$.
We will show that the map $q'$ satisfies the following conditions:
\begin{itemize}
\item [(1-a)]For each $C\in\cal C_{\Delta}$, the map $G_{\flat}\pr C\to\cal D$
is $\frak P$-fibered.
\item [(1-b)]For each morphism $f:C\to C'$ in the underlying category
of $\cal C_{\Delta}$, the induced map $G_{\flat}\pr f:G_{\flat}\pr C\to G_{\flat}\pr{C'}$
is a morphism of $\frak P$-fibered objects.
\end{itemize}
Assuming that conditions (1-a) and (1-b) have been verified for now,
we may complete the proof of (1) as follows. Using Lemma \ref{lem:Map_S^=00005Csharp_core}
and conditions (1-a) and (1-b), we can define a projectively fibrant
simplicial functor $F:\cal C_{\Delta}\to\SS_{/\frak P}^{+}$ by
\[
F\pr C=\pr{G_{\flat}\pr C}_{\natural}.
\]
We then consider the commutative diagram
\[\begin{tikzcd}
	{\overline{\mathcal{X}}} && {\widetilde{\operatorname{Un}}_\phi ^+(G)} \\
	& {S^\sharp\times \widetilde{\operatorname{Un}}_{\Delta^0}^+(\mathcal{D}^\natural),}
	\arrow["i", from=1-1, to=1-3]
	\arrow["q", from=1-3, to=2-2]
	\arrow["p"', from=1-1, to=2-2]
\end{tikzcd}\]where $i$ is adjoint to $i'$ and $q=\cun_{\phi}^{+}\pr G$. By Theorem
\ref{thm:st_un}, the map $i$ is a weak equivalence of fibrant objects
in $\SS_{/S}^{+}$, so its underlying map of simplicial sets 
\[
\cal X\to\cun_{\phi}\pr{G_{\flat}}=\cun_{\phi}\pr{\pr{\pi\circ F}_{\flat}}
\]
is a categorical equivalence over $S\times\cun_{\Delta^{0}}\pr{\cal D}$.
This is the categorical equivalence we were looking for.

We now turn to the verification of conditions (1-a) and (1-b). We
begin with (1-a). According to Proposition \ref{prop:creative}, the
image of the map $q$ in $\SS$ is a categorical fibration. Since
$\cal X$ is $S\times\frak P_{\Un}$-fibered, Lemma \ref{lem:cat_inv}
tells us that so is $\widetilde{\Un}_{\phi}\pr{G_{\flat}}$. It follows
that for each $C\in\cal C$, the map
\[
\cun_{\phi}\pr{G_{\flat}\pr C}\to\cun_{\Delta^{0}}\pr{\cal D}
\]
is $\frak P_{\Un}$-fibered. Thus, by Lemma \ref{lem:cat_inv} and
Corollary \ref{cor:3.2.1.14}, the map $G_{\flat}\pr C\to\cal D$
is $\frak P$-fibered, as required.

Next, we prove (1-b). By Proposition \ref{prop:rectification}, there
is a commutative diagram
\[\begin{tikzcd}
	& {G_\flat(C')} && {\widetilde{\operatorname{Un}}_{\Delta^0}(G_\flat(C'))} \\
	{G_\flat(C)} && {\widetilde{\operatorname{Un}}_{\Delta^0}(G_\flat(C))} \\
	& {\mathcal{D}} && {\widetilde{\operatorname{Un}}_{\Delta^0}(\mathcal{D})}
	\arrow[from=2-1, to=1-2]
	\arrow[from=1-2, to=3-2]
	\arrow[from=2-1, to=3-2]
	\arrow["\simeq", from=1-2, to=1-4]
	\arrow["\simeq"'{pos=0.6}, from=2-1, to=2-3]
	\arrow[from=2-3, to=1-4]
	\arrow[from=2-3, to=3-4]
	\arrow["\simeq"', from=3-2, to=3-4]
	\arrow[from=1-4, to=3-4]
\end{tikzcd}\]whose horizontal arrows are categorical equivalences. Since $\cun_{\phi}\pr{G_{\flat}}$
is a $\frak P_{\Un}$-bundle over $\cal C$, the map $\cun_{\Delta^{0}}\pr{G_{\flat}\pr C}\to\cun_{\Delta^{0}}\pr{G_{\flat}\pr{C'}}$
preserves cocartesian edges over the marked edges of $\cun_{\Delta^{0}}\pr{\cal D}$.
Therefore, the functor $G_{\flat}\pr C\to G_{\flat}\pr{C'}$ preserves
cocartesian edges lying over marked edges of $\cal D$, as claimed.

Next, we turn to the proof of (2). Let $F:\cal C_{\Delta}\to\SS_{/\frak P}^{+}$
be a projectively fibrant functor. We must show that, for each object
$C\in\cal C_{\Delta}$, the counit map
\[
\varepsilon:\pr{\cst_{\phi}^{\frak P}\cun_{\phi}^{\frak P}F}\pr C\to F\pr C
\]
is a $\frak P$-equivalence. Set $F'=\pi\circ F$ and consider the
commutative diagram 
\[\begin{tikzcd}
	{(\widetilde{\operatorname{St}}^+_{\phi}\widetilde{\operatorname{Un}}^+_{\phi}(F'_\flat)^\natural)(C)} & {(F'_\flat)^\natural(C)} \\
	{(\widetilde{\operatorname{St}}^+_{\phi}\widetilde{\operatorname{Un}}^+_{\phi}F')(C)} & {F'(C)}
	\arrow["\varepsilon"', from=2-1, to=2-2]
	\arrow[from=1-1, to=2-1]
	\arrow["{\varepsilon'}", from=1-1, to=1-2]
	\arrow[from=1-2, to=2-2]
\end{tikzcd}\]in $\SS_{/\frak P}^{+}$. According to Lemma \ref{lem:St_Un_counit},
the square is cocartesian. Moreover, Theorem \ref{thm:st_un} and
Lemma \ref{lem:weak_+_equiv->P-equiv} show that the map $\varepsilon'$
is a $\frak P$-equivalence. Since the model structure on $\SS_{/\frak P}^{+}$
is left proper, we conclude that $\varepsilon$ is also a $\frak P$-equivalence.
\end{proof}
\begin{cor}
\label{cor:main}Let $S$ be a simplicial set and let $\frak P$ be
a commutative categorical pattern on an $\infty$-category $\cal D$.
There is a categorical equivalence
\[
\Fun\pr{S,\frak P\-\Fib}\simeq\frak P\-\Bund\pr S.
\]
\end{cor}

\begin{proof}
Let $\Fun^{s}\pr{\widetilde{\fr C}[S],\SS_{/\frak P}^{+}}^{\circ}\subset\Fun^{s}\pr{\widetilde{\fr C}[S],\SS_{/\frak P}^{+}}$
denote the full simplicial subcategory spanned by the projectively
fibrant-cofibrant objects. By \cite[Proposition 4.2.4.4]{HTT}, the
functor
\[
N\pr{\Fun^{s}\pr{\widetilde{\fr C}[S],\SS_{/\frak P}^{+}}^{\circ}}\to\Fun\pr{S,\frak P\-\Fib}
\]
is a categorical equivalence. Since the functor $\Un_{\phi}^{\frak P}:\Fun^{s}\pr{\widetilde{\fr C}[S],\SS_{/\frak P}^{+}}\to\SS_{/S\times\frak P_{\Un}}^{+}$
admits a simplicial enrichment (Remark \ref{rem:enrichment}), Theorem
\ref{thm:unstraightening_P-fibered_obj} and \cite[Corollary A.3.1.12]{HTT}
give us a categorical equivalence 
\[
N\pr{\Fun^{s}\pr{\widetilde{\fr C}[S],\SS_{/\frak P}^{+}}^{\circ}}\xrightarrow{\simeq}\frak P_{\Un}\-\Bund\pr S.
\]
Similarly, using Proposition \ref{prop:P_Un_good}, we find that the
simplicial functor $\SS_{/S\times\frak P_{\Un}}^{+}\to\SS_{/S\times\frak P}^{+}$
induces a categorical equivalence
\[
\frak P_{\Un}\-\Bund\pr S\xrightarrow{\simeq}\frak P\-\Bund\pr S.
\]
In conclusion, there are categorical equivalences
\[
\frak P\-\Bund\pr S\xleftarrow{\simeq}N\pr{\Fun^{s}\pr{\widetilde{\fr C}[S],\SS_{/\frak P}^{+}}^{\circ}}\xrightarrow{\simeq}\Fun\pr{S,\frak P\-\Fib},
\]
and the proof is complete.
\end{proof}

\subsection{\label{subsec:rectification_1-cat}Rectifications of \texorpdfstring{$\mathfrak{P}$}{P}-Bundles
over Ordinary Categories}

Let $\cal C$ be an ordinary category and let $\frak P$ be a commutative
categorical pattern on an $\infty$-category $\cal D$. In this subsection,
we will construct yet another right Quillen equivalence
\[
\Fun\pr{\cal C,\SS_{/\frak P}^{+}}\to\SS_{/N\pr{\cal C}\times\frak P}^{+}
\]
using the relative nerve functor. As we will see, the construction
is equivalent to the one in Theorem \ref{thm:unstraightening_P-fibered_obj}.
The construction in this subsection has the advantage that it is far
simpler than the previous construction.
\begin{defn}
Let $\cal C$ be an ordinary category and let $\frak P$ be a commutative
categorical pattern on an $\infty$-category $\cal D$. The relative
nerve functor
\[
\int^{+}:\Fun\pr{\cal C,\SS^{+}}\to\SS_{/N\pr{\cal C}}^{+}
\]
carries the constant functor $\delta\pr{\overline{\cal D}}:\cal C\to\SS^{+}$
at the marked simplicial set $\overline{\cal D}\in\SS^{+}$ to the
object $N\pr{\cal C}^{\sharp}\times\overline{\cal D}$. We let $\int^{\frak P}$
denote the induced functor
\[
\int^{\frak P}:\Fun\pr{\cal C,\SS_{/\frak P}^{+}}\to\SS_{/N\pr{\cal C}\times\frak P}^{+}.
\]
\end{defn}

\begin{prop}
Let $\frak P$ be a commutative categorical pattern on an $\infty$-category
$\cal D$, and let $\cal C$ be an ordinary category. The functor
\[
\int^{\frak P}:\Fun\pr{\cal C,\SS_{/\frak P}^{+}}_{\mathrm{proj}}\to\SS_{/N\pr{\cal C}\times\frak P}^{+}
\]
is a right Quillen equivalence. Moreover, there is a natural transformation
$\Theta\to\int^{\frak P}$ whose components at fibrant objects are
weak equivalences, where $\Theta$ denotes the composite
\[
\Theta:\Fun\pr{\cal C,\SS_{/\frak P}^{+}}\xrightarrow{\cun_{\varepsilon}^{\frak P}}\SS_{/N\pr{\cal C}\times\frak P_{\Un}}^{+}\to\SS_{/N\pr{\cal C}\times\frak P}^{+}.
\]
Here $\varepsilon:\widetilde{\fr C}[N\pr{\cal C}]\to\cal C$ denotes
the counit map.
\end{prop}

\begin{proof}
According to Theorem \ref{thm:unstraightening_P-fibered_obj} and
Proposition \ref{prop:P_Un_good}, the functor $\Theta$ is a right
Quillen equivalence. It will therefore suffice to prove the following:
\begin{enumerate}
\item The functor $\int^{\frak P}$ is right Quillen.
\item There is a natural transformation $\Theta\to\int^{\frak P}$ whose
components at fibrant objects are weak equivalences.
\end{enumerate}
The proof of (1) is entirely analogous to (and easier than) the proof
of Proposition \ref{prop:Un_=00005Cphi^=00007B=00005CfrakP=00007D_right_Quillen},
so we leave it to the reader.

We next turn to the proof of (2). Let $\pi:\SS_{/\frak P}^{+}\to\SS^{+}$
denote the forgetful functor. Given a functor $F:\cal C\to\SS_{/\frak P}^{+}$,
we define a map $\alpha_{F}:\int^{\frak P}F\to\Theta\pr F$ as follows:
Set $F'=\pi\circ F$. The comparison map of Proposition \ref{prop:relative_nerve}
gives rise to a commutative diagram 
\[\begin{tikzcd}
	{\int ^+F'} & {\widetilde{\operatorname{Un}}_{\varepsilon}^+(F')} \\
	{\mathcal{C}^{\sharp}\times \overline{\mathcal{D}}} & {\mathcal{C}^{\sharp}\times \widetilde{\operatorname{Un}}_{\Delta^0}^+(\overline{\mathcal{D}})}
	\arrow[from=1-1, to=1-2]
	\arrow["{\widetilde{\operatorname{Un}}_{\varepsilon}^\mathfrak{P}(F)}", from=1-2, to=2-2]
	\arrow["{\int^{\mathfrak{P}}F}"', from=1-1, to=2-1]
	\arrow[from=2-1, to=2-2]
\end{tikzcd}\]of marked simplicial sets. We define $\alpha_{F}$ to be the induced
map $\int^{+}F'\to\overline{\cal D}\times_{\cun_{\Delta^{0}}^{+}\pr{\overline{\cal D}}}\cun_{\varepsilon}^{+}\pr{F'}$.
Clearly $\alpha_{F}$ is natural in $F$. If $F$ is projectively
fibrant, then the underlying map of simplicial sets of $\alpha_{F}$
is a categorical equivalence by Proposition \ref{prop:creative},
so again by Proposition \ref{prop:creative} the map $\alpha_{F}$
is a weak equivalence. So the maps $\{\alpha_{F}\}_{F\in\Fun\pr{\cal C,\SS_{/\frak P}^{+}}}$
form the desired natural transformation $\Theta\to\int^{\frak P}$.
\end{proof}

\section{\label{sec:applications}Applications}

In this section, we look at some applications of the results which
we established in Section \ref{sec:Rectification}. In Subsection
\ref{subsec:Classification}, we will show that every $\frak P$-bundle
is classified by a universal $\frak P$-bundle. In Subsections \ref{subsec:Limits}
and \ref{subsec:Colimits}, we will provide formulas for limits and
colimits in $\frak P\-\Fib$ in terms of the associated $\frak P$-bundles
(Corollaries \ref{cor:limit_formula} and \ref{cor:colimit_formula}),
and establish a criterion for a diagram in $\frak P\-\Fib$ to be
a limit or a colimit diagram (Propositions \ref{prop:3.3.3.1} and
\ref{prop:3.3.4.2}). 

We remark that the underlying theme of this section is heavily influenced
by \cite[$\S$3.3]{HTT}. 

\subsection{\label{subsec:Classification}Classification of \texorpdfstring{$\mathfrak{P}$}{P}-Bundles}

Let $\frak P$ be a commutative categorical pattern on an $\infty$-category
$\cal D$. The counit map $\varepsilon:\widetilde{\fr C}[\frak P\-\Fib]\to\pr{\SS_{/\frak P}^{+}}^{\circ}$
gives rise to a $\frak P$-bundle
\[
\cal Z\pr{\frak P}=\cun_{\frak P\-\Fib}\pr{\varepsilon}\times_{\cun_{\Delta^{0}}\pr{\cal D}}\cal D\to\frak P\-\Fib\times\cal D
\]
which we shall refer to as the \textbf{universal $\frak P$-bundle}. 

Given a (small) simplicial set $K$ and a $\frak P$-bundle $p:X\to K\times\cal D$
over $K$, we will say that a diagram $f:K\to\frak P\-\Fib$ \textbf{classifies}
$p$ if there is an equivalence $X_{\natural}\simeq\cal Z\pr{\frak P}\times_{\pr{\frak P\-\Fib}}K$
of $\frak P$-bundles over $K$. Equivalently, $f$ classifies $p$
if it corresponds to $X$ under the categorical equivalence $\Fun\pr{K,\frak P\-\Fib}\simeq\frak P\-\Bund\pr K$
of Corollary \ref{cor:main}. In particular, every $\frak P$-bundle
admits an essentially unique classifying map, justifying the usage
of the adjective ``universal.''

\subsection{\label{subsec:Limits}Limits of \texorpdfstring{$\mathfrak{P}$}{P}-Fibered
Objects}

Recall that if $K$ is a simplicial set and $f:K\to\Cat_{\infty}$
is a functor which classifies a cocartesian fibration $q:X\to K$,
then the limit of $f$ is given by the $\infty$-category of cocartesian
sections of $q$ (\cite[Proposition 3.3.3.2]{HTT}). In this subsection,
we will generalize this result to the case where $\Cat_{\infty}$
is replaced by $\frak P\-\Fib$ (Corollary \ref{cor:limit_formula}).
We will in fact prove a more precise statement (Proposition \ref{prop:3.3.3.1}),
using the diffraction map (Definition \ref{def:diffraction}). 
\begin{rem}
Let $\frak P$ be a commutative categorical pattern on an $\infty$-category
$\cal D$. Then small limits in $\frak P\-\Fib$ can be computed in
$\pr{\Cat_{\infty}}_{/\cal D}$. More precisely, let $\frak{Eq}_{\cal D}=\pr{\{\text{equivalences of }\cal D\},\emptyset}$
denote the smallest commutative categorical pattern on $\cal D$.
Note that the functor $\fr{Eq}_{\cal D}\-\Fib\xrightarrow{\simeq}\pr{\Cat_{\infty}}_{/\cal D}$
is a categorical equivalence \cite[\href{https://kerodon.net/tag/01ZS}{Tag 01ZS}]{kerodon}.
The composite
\[
\frak P\-\Fib\to\fr{Eq}_{\cal D}\-\Fib\xrightarrow{\simeq}\pr{\Cat_{\infty}}_{/\cal D}
\]
preserves and reflects small limits, for it is a conservative right
adjoint by Proposition \ref{prop:creative} and \cite[Proposition B.2.9]{HA}.
\end{rem}

To define the diffraction map, we need a few preliminaries.
\begin{defn}
Let $\frak P$ be a categorical pattern on an $\infty$-category $\cal D$,
and let $S$ be a simplicial set. We will write
\[
\Gamma_{S}^{\frak P}:\SS_{/S\times\frak P}^{+}\to\SS_{/\frak P}^{+}
\]
for the right adjoint of the functor $S^{\sharp}\times-:\SS_{/\frak P}^{+}\to\SS_{/S\times\frak P}^{+}$.
Note that $\Gamma_{S}^{\frak P}$ is right Quillen by \cite[Remark B.2.5]{HA}.
\end{defn}

\begin{prop}
\label{prop:section_equivalence}Let $\frak P$ be a commutative categorical
pattern on an $\infty$-category $\cal D$ and let $f:A\to B$ be
an initial map of simplicial sets. For each fibrant object $X_{\natural}\in\SS_{/B\times\frak P}^{+}$,
the map
\[
\Gamma_{B}^{\frak P}\pr{X_{\natural}}\to\Gamma_{A}^{\frak P}\pr{f^{*}X_{\natural}}
\]
is a $\frak P$-equivalence, which is a trivial fibration if $f$
is a monomorphism.
\end{prop}

\begin{proof}
It suffices to show that the natural transformation $f_{!}\circ\pr{A^{\sharp}\times-}\to B^{\sharp}\times-$
of left Quillen functors $\SS_{/\frak P}^{+}\to\SS_{/B\times\frak P}^{+}$
is a natural weak equivalence. By \cite[Remark B.2.5]{HA}, it will
suffice to show that the map $A^{\sharp}\to B^{\sharp}$ is a weak
equivalence of $\SS^{+}$. Since every initial map is a composition
of a left anodyne extension followed by a trivial fibration \cite[Corollary 4.1.1.12]{HTT},
we are reduced to the case where $f$ is left anodyne. In this case,
the claim is a consequence of \cite[Lemma 3.2.17]{Landoo-cat}.
\end{proof}
Using Proposition \ref{prop:section_equivalence}, we can now define
the diffraction map. (Compare \cite[\href{https://kerodon.net/tag/02TD}{Tag 02TD}]{kerodon}.)
\begin{defn}
\label{def:diffraction}Let $\frak P$ be a commutative categorical
pattern on an $\infty$-category $\cal D$ and let $X'\to K^{\lcone}\times\cal D$
be a $\frak P$-bundle. Set $X=X'\times_{K^{\lcone}}K$. We define
the \textbf{diffraction map}
\[
\opn{Df}:X'_{\natural}\times_{\pr{K^{\lcone}}^{\sharp}}\pr{\{\infty\}}^{\sharp}\to\Gamma_{K}^{\frak P}\pr{X_{\natural}}
\]
to be the composite
\[
X'_{\natural}\times_{\pr{K^{\lcone}}^{\sharp}}\pr{\{\infty\}}^{\sharp}\xrightarrow{\theta}\Gamma_{K^{\lcone}}^{\frak P}\pr{X'_{\natural}}\to\Gamma_{K}^{\frak P}\pr{X_{\natural}},
\]
where $\theta$ is a section of the trivial fibration
\[
\Gamma_{K^{\lcone}}^{\frak P}\pr{X'_{\natural}}\to\Gamma_{\{\infty\}}^{\frak P}\pr{X'_{\natural}\times_{\pr{K^{\lcone}}^{\sharp}}\pr{\{\infty\}}^{\sharp}}\cong X'_{\natural}\times_{\pr{K^{\lcone}}^{\sharp}}\pr{\{\infty\}}^{\sharp}.
\]
\end{defn}

We now arrive at the main result of this subsection.
\begin{prop}
\label{prop:3.3.3.1}Let $\frak P$ be a commutative categorical pattern
on an $\infty$-category $\cal D$. Let $K$ be a small simplicial
set, let $\overline{p}:K^{\lcone}\to\frak P\-\Fib$ be a diagram,
let $X'\to K^{\lcone}\times\cal D$ be a $\frak P$-bundle classified
by $\overline{p}$. Set $X=X'\times_{K^{\lcone}}K$. The following
conditions are equivalent:
\begin{enumerate}
\item The diagram $\overline{p}$ is a limit diagram.
\item The restriction map
\[
\Gamma_{K^{\lcone}}^{\frak P}\pr{X'_{\natural}}\to\Gamma_{K}^{\frak P}\pr{X_{\natural}}
\]
is a $\frak P$-equivalence.
\item The diffraction map
\[
\opn{Df}:X'_{\natural}\times_{\pr{K^{\lcone}}^{\sharp}}\pr{\{\infty\}}^{\sharp}\to\Gamma_{K}^{\frak P}\pr{X_{\natural}}
\]
is a $\frak P$-equivalence.
\end{enumerate}
\end{prop}

\begin{proof}
The equivalence of conditions (2) and (3) is obvious, so we will focus
on the equivalence (1)$\iff$(2). Given a simplicial set $S$, we
will write $\Gamma_{S}^{\frak P}=\Gamma_{S}$. By \cite[Proposition 4.2.3.14]{HTT},
there is an ordinary category $\cal A$ and an initial map $f:N\pr{\cal A}\to K$.
By Proposition \ref{prop:section_equivalence}, the maps
\[
\Gamma_{K^{\lcone}}\pr{X'_{\natural}}\to\Gamma_{N\pr{\cal A}^{\lcone}}\pr{\pr{f^{\lcone}}^{*}X'_{\natural}},\,\Gamma_{K}\pr{X_{\natural}}\to\Gamma_{N\pr{\cal A}}\pr{f^{*}X_{\natural}}
\]
are $\frak P$-equivalences. Therefore, condition (2) is equivalent
to the condition that the map
\[
\Gamma_{N\pr{\cal A}^{\lcone}}\pr{\pr{f^{\lcone}}^{*}X'_{\natural}}\to\Gamma_{N\pr{\cal A}}\pr{f^{*}X_{\natural}}
\]
be a $\frak P$-equivalence. Thus, replacing $K$ by $N\pr{\cal A}$
if necessary, we may assume that $K=N\pr{\cal A}$ is the nerve of
an ordinary category $\cal A$. 

Replacing $X'_{\natural}\in\frak P\-\Bund\pr{N\pr{\cal A}^{\lcone}}$
by an equivalent object, we may assume that $X'_{\natural}=\int_{\cal A^{\lcone}}^{\frak P}F'$
for some \textit{injectively} fibrant (hence projectively fibrant)
functor $F':\cal A^{\lcone}\to\SS_{/K\times\frak P}^{+}$ and that
$\overline{p}$ is the nerve of $F'$. Note that the restriction $F=F'\vert\cal A$
is injectively fibrant, for the left Kan extension functor
\[
\Fun\pr{\cal A,\SS_{/K\times\frak P}^{+}}\to\Fun\pr{\cal A^{\lcone},\SS_{/K\times\frak P}^{+}}
\]
is left Quillen with respect to the injective model structures (because
it simply assigns to the cone point the initial object). Therefore,
by \cite[Theorem 4.2.4.1]{HTT}, condition (1) is equivalent to the
condition that the map $F'\pr{\infty}\to\lim_{\cal A}F$ be a $\frak P$-equivalence.

Now let $L_{\cal A}^{\frak P}:\SS_{/N\pr{\cal A}\times\frak P}^{+}\to\Fun\pr{\cal A,\SS_{/\frak P}^{+}}$
denote the left adjoint of the functor $\int_{\cal A}^{\frak P}$.
Explicitly, $L_{\cal A}^{\frak P}$ is given by $L_{\cal A}^{\frak P}\pr{\overline{X}}=\overline{X}\times_{N\pr{\cal A}^{\sharp}}N\pr{\cal A_{/\bullet}}^{\sharp}$
(Remark \ref{rem:rel_nerve_leftadj}). We consider the diagram 
\[\begin{tikzcd}[ampersand replacement=\&]
	\& {\mathsf{sSet}^+_{/N(\mathcal{A})\times \mathfrak{P}}} \\
	{\mathsf{sSet}^+_{/\mathfrak{P}}} \&\& {\operatorname{Fun}(N(\mathcal{A}),\mathsf{sSet}^+_{/\mathfrak{P}})_{\mathrm{inj}}}
	\arrow["\delta"', from=2-1, to=2-3]
	\arrow["{N(\mathcal{A})^\sharp\times}"{pos=0.3}, from=2-1, to=1-2]
	\arrow["{L^\mathfrak{P}_{\mathcal{A}}}", from=1-2, to=2-3]
\end{tikzcd}\]of left Quillen functors, where $\delta$ denotes the diagonal functor.
By \cite[Remark B.2.5]{HA} and \cite[Lemma 3.2.17]{Landoo-cat},
for each object $\overline{X}\in\SS_{/\frak P}^{+}$, the projection
\[
L_{\cal A}^{\frak P}\pr{N\pr{\cal A}^{\sharp}\times\overline{X}}=N\pr{\cal A_{/\bullet}}^{\sharp}\times\overline{X}\to\delta\pr{\overline{X}}
\]
is a weak equivalence of $\Fun\pr{N\pr{\cal A},\SS_{/\frak P}^{+}}$.
We thus obtain a $\frak P$-equivalence
\[
\lim_{\cal A}F\to\Gamma_{\cal A}X_{\natural}
\]
of right adjoints. Likewise, there is a $\frak P$-equivalence
\[
F'\pr{\infty}=\lim_{\cal A^{\lcone}}F'\to\Gamma_{\cal A^{\lcone}}X_{\natural}'.
\]
So the map $F'\pr{\infty}\to\lim_{\cal A}F$ is a $\frak P$-equivalence
if and only if condition (2) holds. The proof is now complete.
\end{proof}
\begin{cor}
\label{cor:limit_formula}Let $\frak P$ be a commutative categorical
pattern on an $\infty$-category $\cal D$. Let $K$ be a small simplicial
set, let $p:K\to\frak P\-\Fib$ be a diagram, and let $X\to K\times\cal D$
be a $\frak P$-bundle classified by $p$. Then $\Gamma_{K}^{\frak P}\pr{X_{\natural}}$
is a limit of $p$.
\end{cor}

\begin{proof}
Extend $p$ to a limit diagram $\overline{p}:K^{\lcone}\to\frak P\-\Fib$
and let $Y\to K^{\lcone}\times\cal D$ denote a $\frak P$-bundle
classified by $\overline{p}$. The $\frak P$-bundle $Y\times_{K^{\lcone}}K\to K\times\cal D$
is classified by $p$, so it is equivalent to $X$ as a $\frak P$-bundle
over $K$. Thus, by Proposition \ref{prop:3.3.3.1}, there is a $\frak P$-equivalence
$\Gamma_{K^{\lcone}}^{\frak P}\pr{Y_{\natural}}\simeq\Gamma_{K}^{\frak P}\pr{X_{\natural}}.$
By Proposition \ref{prop:section_equivalence}, there is also a $\frak P$-equivalence
\[
\Gamma_{K^{\lcone}}^{\frak P}\pr{Y_{\natural}}\xrightarrow{\simeq}\Gamma_{\{\infty\}}^{\frak P}\pr{Y_{\natural}\times_{K^{\lcone}}\{\infty\}}\cong Y_{\natural}\times_{K^{\lcone}}\{\infty\}\simeq\overline{p}\pr{\infty}.
\]
Hence $\Gamma_{K}^{\frak P}\pr{X_{\natural}}$ is $\frak P$-equivalent
to $\overline{p}\pr{\infty}$, as desired.
\end{proof}

\subsection{\label{subsec:Colimits}Colimits of \texorpdfstring{$\mathfrak{P}$}{P}-Fibered
Objects}

In Subsection \ref{subsec:Limits}, we showed how to compute limits
of $\frak P$-fibered objects. In this subsection, we compute colimits.
Again, the situation is quite similar to the case of colimits in $\Cat_{\infty}$.
Recall that, given a diagram $f:K\to\Cat_{\infty}$ classifying a
cocartesian fibration $q:X\to K$, the colimit of $f$ is weakly equivalent
in $\SS^{+}$ to $X^{\natural}$, the marked simplicial set obtained
from $X$ by marking the $q$-cocartesian edges (\cite[Proposition 3.3.4.2]{HTT}).
We will see that an analogous statement holds when $\Cat_{\infty}$
is replaced by $\frak P\-\Fib$ (Corollary \ref{cor:colimit_formula}).
In fact, we will prove a more precise statement (Proposition \ref{prop:3.3.4.2}),
using the refraction map (Definition \ref{def:refraction}).

We start by introducing the refraction map. (Compare \cite[\href{https://kerodon.net/tag/02UP}{Tag 02UP}]{kerodon}.)
\begin{defn}
\label{def:refraction}Let $\frak P$ be a commutative categorical
pattern on an $\infty$-category $\cal D$, let $K$ be a simplicial
set, and let $p':X'\to K^{\rcone}\times\cal D$ be a $\frak P$-bundle
over $K^{\rcone}$. Set $X=X'\times_{K^{\rcone}}K$. A map $\opn{Rf}:X_{\natural}\to X'_{\natural}\times_{\pr{K^{\rcone}}^{\sharp}}\{\infty\}^{\sharp}$
of $\SS_{/\frak P}^{+}$ is called a \textbf{refraction map} if there
is a morphism $H:\pr{\Delta^{1}}^{\sharp}\times X_{\natural}\to X'_{\natural}$
in $\SS_{/K^{\rcone}\times\frak P}^{+}$ satisfying the following
conditions:
\begin{enumerate}
\item The diagram
\[\begin{tikzcd}
	{\{0\}^\sharp\times X_\natural} && {X'_{\natural}} \\
	{(\Delta^1)^\sharp \times X_\natural} & {(\Delta^1)^{\sharp}\times (K)^\sharp\times \overline{\mathcal{D}}} & {(K^\triangleright)^\sharp\times \overline{\mathcal{D}}}
	\arrow[hook, from=1-1, to=2-1]
	\arrow["{\operatorname{id}\times p}"', from=2-1, to=2-2]
	\arrow["{h\times \operatorname{id}}"', from=2-2, to=2-3]
	\arrow["p", from=1-3, to=2-3]
	\arrow[from=1-1, to=1-3]
	\arrow["H"{description}, from=2-1, to=1-3]
\end{tikzcd}\]is commutative, where $h:\Delta^{1}\times K\to K$ is the map determined
by the inclusion $K\times\{0\}\hookrightarrow K^{\rcone}$ and the
projection $K\times\{1\}\to\{\infty\}$. 
\item The restriction $H\vert\{1\}^{\sharp}\times X_{\natural}$ is equal
to $\opn{Rf}$.
\end{enumerate}
Note that, by Proposition \ref{prop:mla_are_weak_equiv}, refraction
maps exist and are well-defined up to equivalence as objects of $\Map_{\cal D}^{\sharp}\pr{X_{\natural},X_{\natural}'\times_{\pr{K^{\rcone}}^{\sharp}}\{\infty\}}$.
\end{defn}

We can now state the main result of this subsection.
\begin{prop}
\label{prop:3.3.4.2}Let $\frak P$ be a commutative categorical pattern
on an $\infty$-category $\cal D$, let $K$ be a small simplicial
set, let $\overline{f}:K^{\rcone}\to\frak P\-\Fib$ be a diagram which
classifies a $\frak P$-bundle $X'\to K^{\rcone}\times\cal D$. Set
$X=X'\times_{K^{\rcone}}K$. The following conditions are equivalent:
\begin{enumerate}
\item The diagram $\overline{f}$ is a colimit diagram.
\item The inclusion $X_{\natural}\subset X'_{\natural}$ is a $\frak P$-equivalence.
\item The inclusion $X_{\natural}\subset X'_{\natural}$ is a $K^{\rcone}\times\frak P$-equivalence.
\item The refraction map $X_{\natural}\to X'_{\natural}\times_{\pr{K^{\rcone}}^{\sharp}}\{\infty\}$
is a $\frak P$-equivalence.
\end{enumerate}
\end{prop}

To prove Proposition \ref{prop:3.3.4.2}, we need a certain result
on the interaction of $\frak P$-bundles and final maps, stated as
Lemma \ref{lem:pullback_finality}. We will prove this lemma by using
the notion of deformation retracts (in a rather ad hoc manner), which
we now recall.
\begin{defn}
\cite{Nguyen2019}Let $\bf A$ be a simplicial model category. A morphism
$i:A\to B$ of $\bf A$ is called a \textbf{right deformation retract}
if the map
\[
\Delta^{1}\otimes A\amalg_{\{0\}\otimes A}\{0\}\otimes B\to B
\]
obtained as the amalgamation of the maps $\Delta^{1}\otimes A\to\Delta^{0}\otimes A\cong A\xrightarrow{i}B$
and $\{0\}\otimes B\xrightarrow{\cong}B$ extends to a map $h:\Delta^{1}\otimes B\to B$
such that the restriction $h\vert\{1\}\otimes B$ factors through
$A$.
\end{defn}

\begin{rem}
Every right deformation retract of a simplicial model category is
a weak equivalence.
\end{rem}

\begin{example}
\label{exa:right_deformation_retracts}Right deformation retracts
abounds in nature. Here are some examples.
\begin{itemize}
\item The inclusion $\{1\}\subset\Delta^{1}$ is a right deformation retract
of $\SS$. 
\item Let $\bf A$ be a simplicial model category. For any object $A\in\bf A$,
and each right deformation retract $K\to L$ in $\SS$ the map $K\otimes A\to L\otimes A$
is a right deformation retract. 
\item Any pushout of a right deformation retract is again a right deformation
retract.
\end{itemize}
\end{example}

\begin{defn}
Let $\bf A$ be a simplicial model category. A morphism $p:X\to Y$
of $\bf A$ is said to have the \textbf{left path lifting property}
if the map
\[
X^{\Delta^{1}}\to Y^{\Delta^{1}}\times_{Y^{\{0\}}}X^{\{0\}}
\]
is a trivial fibration.
\end{defn}

\begin{example}
Every fibration in a simplicial model category has the left path lifting
property. Maps which has the left path lifting property is stable
under pullback.
\end{example}

\begin{example}
A map $p:X\to Y$ of simplicial sets has the left path lifting property
if and only if it is a left fibration. This follows from Proposition
\cite[Proposition 2.1.2.6]{HTT}.
\end{example}

\begin{prop}
\label{prop:rdr_base_change}Let $\bf A$ be a simplicial model category
and let 
\[\begin{tikzcd}[ampersand replacement=\&]
	{A'} \& {B'} \\
	A \& B
	\arrow["{i'}", from=1-1, to=1-2]
	\arrow[from=1-1, to=2-1]
	\arrow["i"', from=2-1, to=2-2]
	\arrow["p", from=1-2, to=2-2]
\end{tikzcd}\]be a pullback square in $\bf A$. If the map $i$ is a right deformation
retract and $p$ has the left path lifting property, then $i'$ is
a right deformation retract.
\end{prop}

\begin{proof}
Choose a map $h:\Delta^{1}\otimes B\to B$ which exhibits $i$ as
a right deformation retract. Since $p$ has the left path lifting
property, we can find a filler of the diagram 
\[\begin{tikzcd}
	{\Delta^1\otimes A'\amalg_{\{0\}\otimes A'}\{0\}\otimes B'} && {B'} \\
	{\Delta^1\otimes B'} & {\Delta^1\otimes B} & B.
	\arrow[from=1-1, to=2-1]
	\arrow[from=1-1, to=1-3]
	\arrow[from=1-3, to=2-3]
	\arrow["{\Delta^1\otimes p}"', from=2-1, to=2-2]
	\arrow["h"', from=2-2, to=2-3]
	\arrow[dashed, from=2-1, to=1-3]
\end{tikzcd}\]Any such filler exhibits $i'$ as a right deformation retract.
\end{proof}
\begin{lem}
\label{lem:pullback_finality}Let $\frak P$ be a commutative categorical
pattern on an $\infty$-category $\cal D$. Let $f:A\to B$ be a final
map of simplicial sets, and let 
\[\begin{tikzcd}[ampersand replacement=\&]
	{\overline{X}} \& {\overline{Y}} \\
	{A^\sharp \times \overline{\mathcal{D}}} \& {B^\sharp \times \overline{\mathcal{D}}}
	\arrow["g", from=1-1, to=1-2]
	\arrow["q", from=1-2, to=2-2]
	\arrow["{f\times{\operatorname{id}_\mathcal{D}}}"', from=2-1, to=2-2]
	\arrow["p"', from=1-1, to=2-1]
\end{tikzcd}\]be a pullback square of marked simplicial sets. If $\overline{Y}\in\SS_{/B\times\frak P}^{+}$
is fibrant, then $g$ is a $\frak P$-equivalence.
\end{lem}

\begin{proof}
By factoring $f$ into a right anodyne extension followed by a trivial
fibration, which is possible by \cite[Corollary 4.1.1.12]{HTT}, we
may assume that $f$ is a right anodyne extension. Consider the class
$\scr M$ of monomorphisms $C\to D$ of simplicial sets such that,
for any morphism $D\to B$ of simplicial sets, the map
\[
C^{\sharp}\times_{B^{\sharp}}\overline{Y}\to D^{\sharp}\times_{B^{\sharp}}\overline{Y}
\]
is a $\frak P$-equivalence. We claim that $\scr M$ contains all
right anodyne extensions. Since $\scr M$ is weakly saturated, it
will suffice to show that $\scr M$ contains a generating set of right
anodyne extensions. By \cite[Proposition 2.1.2.6]{HTT}, the set 
\[
S=\{\Delta^{1}\times\partial\Delta^{n}\cup\{1\}\times\Delta^{n}\to\Delta^{1}\times\Delta^{n}\}_{n\geq0}
\]
generates the class of right anodyne extensions. We claim that $S\subset\scr M$.

Let $n\geq0$, and set $D=\Delta^{1}\times\partial\Delta^{n}\cup\{1\}\times\Delta^{n}$
and $E=\Delta^{1}\times\Delta^{n}$. We wish to show that, for every
map $E\to B$ of simplicial sets, the map $D^{\sharp}\times_{B^{\sharp}}\overline{Y}\to E^{\sharp}\times_{B^{\sharp}}\overline{Y}$
is a $\frak P$-equivalence. Since $\frak P$ is commutative, Remark
\ref{rem:mlfib} shows that $q$ is a marked left fibration. Therefore,
the map $q$ has the left path lifting property as a morphism of $\SS_{/\frak P}^{+}$.
Now set $C=\{1\}\times\Delta^{n}$. The inclusions $C\to D$ and $C\to E$
are right deformation retracts of simplicial sets (Example \ref{exa:right_deformation_retracts}),
so the inclusions $C^{\sharp}\times\overline{\cal D}\to D^{\sharp}\times\overline{\cal D}$
and $\overline{C}^{\sharp}\times\overline{\cal D}\to E^{\sharp}\times\overline{\cal D}$
are right deformation retracts of $\SS_{/\frak P}^{+}$. It follows
from Proposition \ref{prop:rdr_base_change} that the maps $C^{\sharp}\times_{B^{\sharp}}\overline{Y}\to D^{\sharp}\times_{B^{\sharp}}\overline{Y}$
and $C^{\sharp}\times_{B^{\sharp}}\overline{Y}\to E^{\sharp}\times_{B^{\sharp}}\overline{Y}$
are right deformation retracts of $\SS_{/\frak P}^{+}$. In particular,
these maps are $\frak P$-equivalences. By the two out of three property
of $\frak P$-equivalences, we deduce that the map $D^{\sharp}\times_{B^{\sharp}}\overline{Y}\to E^{\sharp}\times_{B^{\sharp}}\overline{Y}$
is a $\frak P$-equivalence, completing the proof.
\end{proof}
We now arrive at the proof of Proposition \ref{prop:3.3.4.2}.
\begin{proof}
[Proof of Proposition \ref{prop:3.3.4.2}]First we prove the equivalence
of conditions (2) and (3). Since the forgetful functor $\SS_{/K^{\rcone}\times\frak P}^{+}\to\SS_{/\frak P}^{+}$
is left Quillen, the implication (3)$\implies$(2) is obvious. For
the converse, suppose that condition (2) is satisfied. Factor the
map $X_{\natural}\subset X'_{\natural}$ as
\[
X_{\natural}\xrightarrow{i}Y_{\natural}\xrightarrow{p}X'_{\natural},
\]
where $i$ is a trivial cofibration and $p$ is a fibration of $\SS_{/K^{\rcone}\times\frak P}^{+}$.
We wish to show that the map $p$ is a trivial fibration. By Proposition
\ref{prop:fiberwise_P_equiv}, it suffices to show that, for each
vertex $v\in K^{\rcone}$, the induced map $p_{v}:Y_{\natural}\times_{\pr{K^{\rcone}}^{\sharp}}\{v\}^{\sharp}\to X'_{\natural}\times_{\pr{K^{\rcone}}^{\sharp}}\{v\}^{\sharp}$
is a $\frak P$-equivalence. 

Since the functor $\Fun\pr{K^{\rcone},\frak P\-\Fib}\to\Fun\pr{K,\frak P\-\Fib}$
admits a fully faithful left adjoint, so does the functor $\frak P\-\Bund\pr{K^{\rcone}}\to\frak P\-\Bund\pr K$.
Therefore, the derived unit of the adjunction
\[
\SS_{/K^{\rcone}\times\frak P}^{+}\adj\SS_{/K\times\frak P}^{+}
\]
is an isomorphism. This means that the map $i$ induces a $K\times\frak P$-equivalence
$X'_{\natural}\to Y_{\natural}\times_{\pr{K^{\rcone}}^{\sharp}}K^{\sharp}$,
so the map $p_{v}$ is a $\frak P$-equivalence for every $v\in K$.
Also, Lemma \ref{lem:pullback_finality} shows that the inclusions
$Y_{\natural}\times_{\pr{K^{\rcone}}^{\sharp}}\{\infty\}^{\sharp}\subset Y_{\natural}$
and $X'_{\natural}\times_{\pr{K^{\rcone}}^{\sharp}}\{\infty\}^{\sharp}\subset X'_{\natural}$
are $\frak P$-equivalences, so the map $p_{\infty}$ is a $\frak P$-equivalence
if and only if $p$ is a $\frak P$-equivalence. Since $i$ and $pi$
are $\frak P$-equivalences, it follows that $p_{\infty}$ is a $\frak P$-equivalence.
Hence $p_{v}$ is a $\frak P$-equivalence for every vertex $v\in K^{\rcone}$,
as required.

Next we prove the equivalence of conditions (1) and (3). Condition
(1) is equivalent to the condition that the object $\overline{f}\in\Fun\pr{K^{\rcone},\frak P\-\Fib}$
belong to the essential image of the left adjoint $\Fun\pr{K,\frak P\-\Fib}\to\Fun\pr{K^{\rcone},\frak P\-\Fib}$
of the restriction functor. This is equivalent to the condition that
the object $X_{\natural}'\in\ho\pr{\SS_{/K^{\rcone}\times\frak P}^{+}}$
belong to the essential image of the total left derived functor $\bb Li_{!}:\ho\pr{\SS_{/K\times\frak P}^{+}}\to\ho\pr{\SS_{/K^{\rcone}\times\frak P}^{+}}$.
Since $\bb Li_{!}$ is fully faithful, this is equivalent to the condition
that the derived counit $X_{\natural}\to X'_{\natural}$ be a $K^{\rcone}\times\frak P$-equivalence,
and the proof is complete.

We now complete the proof by proving the equivalence of conditions
(2) and (4). By construction, the inclusion $X_{\natural}\hookrightarrow X'_{\natural}$
is left homotopic in $\SS_{/\frak P}^{+}$ to the composite $X_{\natural}\xrightarrow{\opn{Rf}}X_{\natural}\times_{\pr{K^{\rcone}}^{\sharp}}\{\infty\}\hookrightarrow X'_{\natural}$.
Since the inclusion $X_{\natural}\times_{\pr{K^{\rcone}}^{\sharp}}\{\infty\}\hookrightarrow X'_{\natural}$
is a $\frak P$-equivalence by Lemma \ref{lem:pullback_finality},
the claim follows from the two out of three property of $\frak P$-equivalences.
\end{proof}
We conclude this subsection with a corollary of Proposition \ref{prop:3.3.4.2}.
\begin{cor}
\label{cor:colimit_formula}Let $\frak P$ be a commutative categorical
pattern on an $\infty$-category $\cal D$, let $K$ be a small simplicial
set, and let $f:K\to\frak P\-\Fib$ be a diagram classifying the $\frak P$-bundle
$X\to K\times\cal D$. Let $X_{\natural}\to Y_{\natural}$ be a trivial
cofibration in $\SS_{/\frak P}^{+}$ such that $Y_{\natural}$ is
$\frak P$-fibered. Then $Y_{\natural}\in\frak P\-\Fib$ is a colimit
of $f$.
\end{cor}

\begin{proof}
This follows from Lemma \ref{lem:pullback_finality} and Proposition
\ref{prop:3.3.4.2}.
\end{proof}

\providecommand{\bysame}{\leavevmode\hbox to3em{\hrulefill}\thinspace}
\providecommand{\MR}{\relax\ifhmode\unskip\space\fi MR }
\providecommand{\MRhref}[2]{%
  \href{http://www.ams.org/mathscinet-getitem?mr=#1}{#2}
}
\providecommand{\href}[2]{#2}

\end{document}